\documentclass{amsart}

\usepackage{mathrsfs}
\usepackage[all]{xy}
 
\usepackage{mathtools}
\usepackage{amsmath}
\usepackage{mathbbol}
\usepackage{wasysym}
\usepackage{amssymb}

\usepackage{MnSymbol}
\usepackage{comment}
\usepackage{enumerate}
\usepackage{xcolor}

\usepackage{comment}
\usepackage{ushort}
\usepackage{enumitem} 
 
\usepackage{changepage}
\usepackage{algorithm}
\usepackage[noend]{algpseudocode}

\usepackage{wrapfig}

\usepackage{soul}
\usepackage{url}
\usepackage{tikz}

\usepackage{hyperref}
\hypersetup{colorlinks=true, linkcolor=blue, 
            citecolor=cyan, pdfauthor={YANG Wenyuan}}

\usepackage[capitalize]{cleveref}
\crefname{equation}{}{}

\usepackage[margin=1.1in]{geometry}

\usepackage{import}
\usepackage{xifthen}
\usepackage{pdfpages}
\usepackage{transparent}

\graphicspath{{./figures}}

\newtheorem{thm}{Theorem}[section]
\newtheorem{prop}[thm]{Proposition}
\newtheorem{cor}[thm]{Corollary}
\newtheorem{lem}[thm]{Lemma}

\newtheorem{claim}[thm]{Claim}

\newtheorem{conv}[thm]{Convention}

\theoremstyle{definition}
\newtheorem{defn}[thm]{Definition}

\theoremstyle{remark}

\newtheorem{rem}[thm]{Remark}

\newtheorem{prob}[thm]{Problem}

\AddToHook{env/prop/begin}{\crefalias{thm}{prop}}
\crefname{prop}{Proportion}{Proportions}
\AddToHook{env/cor/begin}{\crefalias{thm}{cor}}
\crefname{cor}{Corollary}{Corollarys}
\AddToHook{env/lem/begin}{\crefalias{thm}{lem}}
\crefname{lem}{Lemma}{Lemmas}
\AddToHook{env/conj/begin}{\crefalias{thm}{conj}}
\crefname{conj}{Conjecture}{Conjectures}
\AddToHook{env/claim/begin}{\crefalias{thm}{claim}}
\crefname{claim}{Claim}{Claims}
\AddToHook{env/fact/begin}{\crefalias{thm}{fact}}
\crefname{fact}{Fact}{Facts}

\AddToHook{env/conv/begin}{\crefalias{thm}{conv}}
\crefname{conv}{Convention}{Conventions}
\AddToHook{env/defn/begin}{\crefalias{thm}{defn}}
\crefname{defn}{Definition}{Definitions}

\AddToHook{env/quest/begin}{\crefalias{thm}{quest}}
\crefname{quest}{Question}{Questions}
\AddToHook{env/rem/begin}{\crefalias{thm}{rem}}
\crefname{rem}{Remark}{Remarks}
\AddToHook{env/example/begin}{\crefalias{thm}{example}}
\crefname{example}{Example}{Examples}
\AddToHook{env/examples/begin}{\crefalias{thm}{examples}}
\crefname{examples}{Examples}{Examples}
\AddToHook{env/prob/begin}{\crefalias{thm}{prob}}
\crefname{prob}{Problem}{Problems}

\newcommand{\bU}{\overline{ X}}
\newcommand{\hU}{\partial_h{ X}}
\newcommand{\pU}{\partial{ X}}
\newcommand{\ax}{\mathrm{Ax}}

\newcommand{\diam}{\mathrm{diam}}
\newcommand{\bigand}{\quad \text{and} \quad }
\newcommand{\CAT}{\mathrm{CAT(0)}}
\newcommand{\f}{\mathbb F}
\newcommand{\p}{\mathcal P}
\newcommand{\e}[1]{\omega_{#1}}
\newcommand{\act}{\curvearrowright}
\newcommand{\bR}{\partial_{\mathcal{R}}X}

\begin{document}

\title{Ancona inequalities along generic geodesic rays}

\author{Kairui Liu}	
 \address{Beijing International Center for Mathematical Research\\
Peking University\\
 Beijing 100871, China
P.R.}
 \email{liukairui@pku.edu.cn}
 
\author{Wenyuan Yang}

\address{Beijing International Center for Mathematical Research\\
Peking University\\
 Beijing 100871, China
P.R.}
\email{wyang@math.pku.edu.cn}

\thanks{(W.Y.) Partially supported by National Key R \& D Program of China (SQ2020YFA070059) and  National Natural Science Foundation of China (No. 12131009 and No.12326601)}

\begin{abstract} 
This paper presents several versions of the Ancona inequality for finitely supported, irreducible random walks on non-amenable groups. We first study a class of Morse subsets with narrow points and prove the Ancona inequality around these points in any finitely generated non-amenable group. This result implies the inequality along all Morse geodesics and recovers the known case for relatively hyperbolic groups.

We then consider any geometric action of a non-amenable group with contracting elements. For such groups, we construct a class of generic geodesic rays, termed proportionally contracting rays, and establish the Ancona inequality along a sequence of good points. This leads to an embedding of a full-measure subset of the horofunction boundary into the minimal Martin boundary.

A stronger Ancona inequality is established for groups acting geometrically on an irreducible CAT(0) cube complex with a Morse hyperplane. In this setting, we show that the orbital maps extend continuously to a partial boundary map from a full-measure subset of the Roller boundary into the minimal Martin boundary. Finally, we provide explicit examples, including right-angled Coxeter groups (RACGs) defined by an irreducible graph with at least one vertex not belonging to any induced 4-cycle.
\end{abstract}

\maketitle

\section{Introduction}
Let $G$ be a finitely generated group. Consider an irreducible probability  $\mu$   on $G$ so that the support generates $G$ as a group. This defines a $\mu$-random walk on $G$ as a Markov chain over the state set $G$ with transition probability $p(x,y)=\mu(x^{-1}y)$ for any $x,y\in G$. The spectral radius of the associated Markov operator is defined as 
$$\rho(\mu)=\limsup_{n\to \infty} \left(p_n(x,y)\right)^{1/n}$$ 
where the $n$-th convolution $p_n(x,y):=\mu^{\star n}(x^{-1}y)$ is the probability of $\mu$-random walk starting at $x$ and visiting $y$ in $n$ steps. Alternatively, the spectral radius is the convergence radius of the $r$-Green function $$\mathcal G_r(x,y):=\sum_{n=0}^{\infty} p_n(x,y) r^n$$ for $r> 0$. 
It is a landmark result of Kesten that $G$ is amenable if and only if $\rho(\mu)=1$ for any probability $\mu$ with symmetric support. If the group $G$ is non-amenable, then $\mathcal G(x,y):=\mathcal G_1(x,y)$ is finite and the random walk is transient.

It is an important question to understand the asymptotic behavior of $\mathcal G(x,y)$ as $x,y\to \infty$ (i.e. leaving every finite subset in $G$). Indeed, the  asymptotic behavior of normalized \textit{Martin kernels} :   
$$
K_y(x):=\mathcal G(x,y)/\mathcal G(o,y) \text{ as } y\to \infty$$
defines the Martin boundary of a $\mu$-random walk, which allows us to obtain all positive $\mu$-harmonic functions on $G$ via Poisson-Martin integration formula (\cite{Sawyer, Woess_2000}). Precisely, Martin boundary $\partial_{\mathcal M} G$ consists of all limits of functions $\lim_{y\to\infty}\mathcal G(\cdot,y)/\mathcal G(o,y)$ and forms a topological compactification of the group  of elements $y\in G$ viewed as $y\in G\mapsto\mathcal{G}(\cdot,y)/\mathcal G(o,y)$ (see \cref{sec preliminary of  boundary}). It has been a recurring theme to understand the Martin boundary via a concrete geometric boundary of $G$.  
\begin{prob}\label{MainPblm}
Let us endow the group $G$ with  a  geometric boundary or bordification $\partial G$ (e.g. from the Cayley graph or a geometric action on metric spaces). Does the identity map $\Psi: G\to G$ extend to a continuous map $\partial \Psi$ between  $\partial G$ and $\partial_{\mathcal M} G$? That is, $(g_n\to\xi\in \partial G) \Rightarrow (g_n\to \partial\Psi(\xi)\in \partial_{\mathcal M} G)$.
\end{prob}
We shall call $\partial \Psi$ (partial) boundary map, when it exists. Partial boundary maps have been instrumental in studying the comparison of harmonic and conformal measures, strict inequality of fundamental inequalities, etc \cite{BHM,GMM,DG20}.  

We now assume that $\mu$ is finitely supported, and briefly mention the results motivating this study.
In \cite{kral_positive_1988}, Ancona famously identified the Gromov boundary with Martin boundary for hyperbolic groups. The case for Fuchisan groups was previously dealt by Series. 
Martin boundary of several non-hyperbolic groups was studied by Woess and his collaborators \cite{PW87,W89,CSW}, who  mainly focused on infinitely-ended groups and obtained a nontrivial map from Martin boundary to the end boundary.  Recently, Gekhtman--Gerasimov--Potyagailo--Yang   proved the following in \cite{gekhtman2021martin}.
\begin{thm}\label{GGPYThm}
For any finitely generated group $G$, the identification $\Psi: G\to G$ extends continuously to a $G$-equivariant surjective map $\partial \Psi$   from the Martin boundary to Floyd boundary $\partial_{\mathcal F} G$ of $G$ (\cite{Floyd}). If $|\partial_{\mathcal F} G|\ge 3$, then $\partial \Psi$ is injective on the preimages of conical points in $\partial_{\mathcal F} G$.    
\end{thm} 
By a theorem of Gerasimov \cite{gerasimov_floyd_2012}, the Floyd boundary surjects to the Bowditch boundary for any non-elementary relatively hyperbolic group, for which we get a non-trivial map from Martin boundary to an infinite set. Moreover, the map is injective outside the preimages of countably many parabolic points. The class of infinitely-ended groups are relatively hyperbolic, whose end boundaries are natural quotients of Floyd boundary. 
\cref{GGPYThm} thus forms a common generalization of Ancona's and Woess' results.   However, most groups have trivial Floyd boundary of at most two points \cite{levcovitz_thick_2019}; it remains an open question whether a group with nontrivial Floyd boundary must be relatively hyperbolic \cite{olshanskii_lacunary_2009}. 

The goal of the present paper is to address \cref{MainPblm} for a large class of \emph{non-relatively hyperbolic} groups, which notably contains strongly contracting elements. It is easy to see that the inverse of $\partial \Psi$ in \cref{GGPYThm} is an embedding of    conical points (with subspace topology) in Floyd boundary into Martin boundary. Our main results shall generalize \cref{GGPYThm}  in this direction. 

\subsection{The partial boundary maps}
Assume that $G$ is a \textit{non-elementary} group which is not virtually cyclic. A proper and cocompact action on a proper geodesic metric space $(X,d)$ shall be referred to as a \textit{geometric action}. The notion of contracting elements has been recently under active research and usually thought of as hyperbolic directions in the ambient groups. We refer to \cite{arzhantseva2017characterizations, aougab2017pulling, yang2019statistically, yang2022conformal} for related studies.  

Let the non-elementary group $G$ act geometrically on $X$ with contracting elements. In  \cite{yang2022conformal}, in presence of contracting elements, a convergence-like  group theory concerning topological dynamics and a theory  of Patterson-Sullivan measures have been developed on the horofunction boundary $\hU$. 
We declare two horofunctions $b_\xi,b_\eta: X\to \mathbb R$ in $\hU$ to be \textit{equivalent} if the difference $|b_\xi(x)-b_\eta(x)|$ is uniformly bounded over $x\in X$. This defines an equivalent relation called \textit{finite difference relation} on $\hU$. The first main result of this paper allows us to embed a full measure set of boundary points of $\hU$ up to this relation into the \textit{minimal} Martin boundary (recalled below).
Recall that  the Martin boundary is defined using a finitely supported irreducible probability measure on $G$.

\begin{thm}[\cref{the map from horofunction to Martin}]\label{the map from horofunction to Martin intro}
Let $\nu_{ps}$ denote the canonical  measure class of Patterson-Sullivan measures supported on the horofunction boundary $\partial_h X$. Then there exists a $\nu_{ps}$-full measure subset $\Lambda$ of $\hU$, and a $G$-equivariant map $\partial\Psi$ from $\Lambda$ to the minimal Martin boundary  with the following properties:
    \begin{enumerate}
        \item $\partial\Psi$ is injective up to the finite difference equivalence.
        \item $\partial\Psi$ is continuous with respect to the direct  limit topology on $\Lambda$.
        \item For any point $\xi \in  \Lambda$, there is a proportionally contracting ray accumulating into $[\xi]$. 
    \end{enumerate}  
\end{thm}

The direct limit topology on $\Lambda$ and proportionally contracting rays will be formally defined in \cref{sec continuity} and \cref{def prop. contracting rays}.

\begin{rem}
By \cite[Cor 5.8, Lemma 8.7]{yang2022conformal}, the quotient of a generic set called Myrberg set by the finite difference relation is a metrizable separable space, on which  Patterson-Sullivan measures (with dimension at the critical exponent) are unique up to a bounded multiplicative Radon-Nikodym derivative. If we first intersect with the Myrberg set and then pass to the quotient, then the map in (1) becomes injective. The set of Morse geodesic rays is a $\nu_{ps}$-null set (\cite[Theorem 11.6]{yang2022conformal}). In this sense, our result  complements the result of Cordes-Gehktman-Dussaule \cite{cordes2022embedding} which embeds the Morse directions for a more restricted class of hierarchically hyperbolic groups. We actually improve their result in \cref{CDGMorseCor}.
\end{rem}
\begin{rem}\label{compare with GPPY}
For many groups,  the horofunction boundary appears as kind of universal boundary. For instance, the horofunction boundary covers Floyd boundary (\cite[Lemma 10.9]{yang2022conformal}), which in turn covers end boundary for any groups (\cite{gekhtman2021martin}). For relatively hyperbolic groups,  Floyd boundary  covers Bowditch boundary  (\cite{gerasimov_floyd_2012}). In these groups, finite difference relation is generally not trivial (e.g. a direct product $G\times \mathbb Z_2$ of a given group $G$ with $\mathbb Z_2$). So the above result is  weaker and actually follows from \cref{GGPYThm} (though our proof could be adapted to treat these boundaries).  
\end{rem}
\begin{rem}\label{mcg injective space}
The geodesic space $X$ could be weaken as a coarsely geodesic space. It is known from \cite{HHP} that mapping class groups, or more generally hierarchically hyperbolic groups act properly and cocompactly on a rough geodesic and coarsely injective space. These actions contain contracting elements \cite{SZ22}, so  \cref{the map from horofunction to Martin intro} applies.  Taking the injective hull further gives a metrically proper and co-bounded action on a possibly non-proper injective space. Using conical geodesic combing, one could prove the finite difference partition would become trivial on the set $\Lambda$. We leave the details  in future works. 
\end{rem}

In passing to applications,  let us digest this statement by introducing another point of view to Martin boundary from metric geometry. If we  define the so-called Green metric as in \cite{BB07} 
$$
\forall x,y\in G:\quad d_{\mathcal G}(x,y):=-\log \left(\frac{\mathcal G(x,y)}{\mathcal G(o,o)}\right),$$
then Martin kernels are exponential of horofunctions in Green metric, so the Martin boundary $\partial_{\mathcal M} G$ is exactly the horofunction boundary of $G$ in Green metric. When $\mu$ has sup-exponential moment, $\partial_{\mathcal M} G$ consists of harmonic functions (\cite[Lemma 7.1]{gekhtman2021martin}). A boundary point is called \textit{minimal} if it is  represented by  a minimal harmonic function, that is, any harmonic function within a bounded multiple is proportional to it. 
The resulting \textit{minimal Martin boundary} $\partial_{\mathcal M}^m G$ constitutes an essential subset of $\partial_{\mathcal M} G$ in potential theory and from the viewpoint of horofunctions, the finite difference relation restricts to the trivial relation on it. See \cref{sec preliminary of  boundary} for more detailed information.

Th class of geometric actions with contracting elements includes naturally occurring groups, such as groups with nontrivial Floyd boundary, relatively hyperbolic groups, CAT(0) groups with rank-one elements. See \cref{compare with GPPY} and \cref{mcg injective space}.

It is known that the visual boundary of a  proper CAT(0) space is homeomorphic to the horofunction boundary, for which two equivalent horofunctions are necessarily the same \cite[Lemma 11.1]{yang2022conformal}. That is, the finite difference relation becomes a trivial relation, so the following corollary is immediate.
\begin{thm}[\cref{the map from horofunction to Martin}]\label{the map from horofunction to Martin: intro}
Assume that $X$ is a proper $\mathrm{CAT(0)}$ space on which $G$ acts geometrically with contracting elements. Then there exists a $\nu_{ps}$-full measure subset ${\Lambda}$ of the visual boundary of $X$, and a $G$-equivariant injective map $\Psi$ from ${\Lambda}$ into the minimal Martin boundary. Moreover, $\Psi$ is continuous with respect to the direct limit topology  on $\Lambda$. 
\end{thm}



We remark  that the map $\partial \Psi$ is only defined on the boundary, i.e. we do not know whether the identification $\Psi$ on $G$ extends continuously to the map $\partial\Psi$ on $\Lambda$ as in \cref{MainPblm}. In the proof, we  only prove that $\Psi$ extends continuously in \textit{any linear neighborhood} (\cref{def linearnbhd}) of geodesic rays ending at $\Lambda$.  Secondly, the direct limit topology on $\Lambda$ might not be same as the subspace topology.

Our next main result proves a continuous extension in \cref{MainPblm}  for a large class of cubical groups under the natural subspace topology.

\begin{thm}[\cref{the map for RACG}]\label{the map for RACG: intro}
Let $\Gamma$ be an irreducible graph with at least three vertices and $G$ be the associated right-angled Coxeter group.
Assume that $\Gamma$ contains at least one vertex not contained in any induced 4-cycle. 
Then there exists a $\nu_{PS}$-full measure subset $\Lambda$ of the Roller boundary $\partial_\mathcal{R} \mathcal{D}_{\Gamma}$ of the Davis complex, and the identity $\Psi$ extends continuously to a boundary map from  $\Lambda$ to the minimal Martin boundary of $G$ $$\partial \Psi:\Lambda \to \partial^{m}_{\mathcal M} G$$  such that $\partial \Psi$ is a homeomorphism  to the image $\partial \Psi(\Lambda)$.
\end{thm}

To the best of the authors' knowledge, this provides the first examples in the literature of non-relatively hyperbolic groups where a full-measure subset of a geometric boundary admits an embedding into the Martin boundary.
Indeed, every relatively hyperbolic group has exponential divergence \cite[Theorem 1.3]{Sisto2012OnMR}, while Dani and Thomas \cite[Section 5]{dani_divergence_2015} provide examples of RACGs with polynomial divergence that satisfy the hypotheses of \cref{the map for RACG: intro}.

We further deduce that the action on the Martin boundary admits many elements with North-South dynamics, and consequently there is a unique  minimal $G$-invariant closed subset. The second half of ``moreover" statement was known in relatively hyperbolic groups (\cite[Prop 7.15]{gekhtman2021martin}).
\begin{thm}[\cref{subsec Minimal actions on Martin boundary}]\label{North south dynamics: intro}
Under the same assumption on $G$ in \cref{the map for RACG: intro}, there exists infinitely many contracting isometries with North-South dynamics on $\partial_{\mathcal M} G$.  Moreover, there exists a unique minimal $G$-invariant closed subset denoted as $\Lambda(G)$ in $\partial_{\mathcal M} G$ and the $[\cdot]$-closure of $\Lambda(G)$ covers  $\partial_{\mathcal M} G$. 
\end{thm}

\subsection{Ancona inequalities}\label{subsec ancona ineuqalities intro}
The main tool we use to prove the above theorems is a sub-multiplicative inequality of Green functions along certain geodesics, which we call \textit{Ancona inequality}. In \cite{kral_positive_1988}, Ancona proved that Green functions are coarsely multiplicative along every geodesic in every hyperbolic group, which is the key ingredient in obtaining the homeomorphism between Martin boundary and Gromov boundary. A natural generalization of Ancona's inequality in relatively hyperbolic groups is obtained in \cite{gekhtman2021martin} which is also crucial to show that Martin boundary covers Floyd boundary in \cref{GGPYThm}. The Floyd metric is instructive in this generalization, but fails to provide a useful inequality if the Floyd boundary is trivial. A novel contribution of this work is the development of a general set of conditions that ensure the Ancona inequality in a general metric space. Further, we present four distinct versions of the inequality under different circumstances; their logical relationships are outlined in \cref{fig:logicalflow}. 

We now introduce the main version of the inequality (for Morse subsets with narrow points), along with one weaker and one stronger variants. In particular, this  recovers the ones in \cite{gekhtman2021martin} without invoking Floyd metric. The fourth one (along geodesics with good points) will be discussed in \textsection \ref{subsec proofs}.
    
\begin{figure}
    \centering

\tikzset{every picture/.style={line width=0.75pt}} 

\begin{tikzpicture}[x=0.75pt,y=0.75pt,yscale=-1,xscale=1]

\draw   (3.33,84.33) -- (175.33,84.33) -- (175.33,129) -- (3.33,129) -- cycle ;
\draw   (3.67,229.67) -- (180.67,229.67) -- (180.67,274.33) -- (3.67,274.33) -- cycle ;
\draw   (220.67,84.33) -- (392.67,84.33) -- (392.67,129) -- (220.67,129) -- cycle ;
\draw   (90.67,157.67) -- (262.67,157.67) -- (262.67,202.33) -- (90.67,202.33) -- cycle ;
\draw   (222,229.67) -- (394,229.67) -- (394,274.33) -- (222,274.33) -- cycle ;
\draw   (413.33,145.33) -- (585.67,145.33) -- (585.67,206.67) -- (413.33,206.67) -- cycle ;

\draw  [dash pattern={on 4.5pt off 4.5pt}]  (264,180) -- (410,180) ;
\draw [shift={(412,180)}, rotate = 180] [color={rgb, 255:red, 0; green, 0; blue, 0 }  ][line width=0.75]    (10.93,-3.29) .. controls (6.95,-1.4) and (3.31,-0.3) .. (0,0) .. controls (3.31,0.3) and (6.95,1.4) .. (10.93,3.29)   ;
\draw    (132,157.67) -- (132,130.33) ;
\draw [shift={(132,128.33)}, rotate = 90] [color={rgb, 255:red, 0; green, 0; blue, 0 }  ][line width=0.75]    (10.93,-3.29) .. controls (6.95,-1.4) and (3.31,-0.3) .. (0,0) .. controls (3.31,0.3) and (6.95,1.4) .. (10.93,3.29)   ;
\draw    (239.33,157.67) -- (238.71,130.33) ;
\draw [shift={(238.67,128.33)}, rotate = 88.7] [color={rgb, 255:red, 0; green, 0; blue, 0 }  ][line width=0.75]    (10.93,-3.29) .. controls (6.95,-1.4) and (3.31,-0.3) .. (0,0) .. controls (3.31,0.3) and (6.95,1.4) .. (10.93,3.29)   ;
\draw    (130.67,201.67) -- (130.67,227.67) ;
\draw [shift={(130.67,229.67)}, rotate = 270] [color={rgb, 255:red, 0; green, 0; blue, 0 }  ][line width=0.75]    (10.93,-3.29) .. controls (6.95,-1.4) and (3.31,-0.3) .. (0,0) .. controls (3.31,0.3) and (6.95,1.4) .. (10.93,3.29)   ;
\draw    (238,201.67) -- (238.62,228.33) ;
\draw [shift={(238.67,230.33)}, rotate = 268.67] [color={rgb, 255:red, 0; green, 0; blue, 0 }  ][line width=0.75]    (10.93,-3.29) .. controls (6.95,-1.4) and (3.31,-0.3) .. (0,0) .. controls (3.31,0.3) and (6.95,1.4) .. (10.93,3.29)   ;
\draw  [dash pattern={on 4.5pt off 4.5pt}]  (394,108.33) -- (472.83,142.86) ;
\draw [shift={(474.67,143.67)}, rotate = 203.65] [color={rgb, 255:red, 0; green, 0; blue, 0 }  ][line width=0.75]    (10.93,-3.29) .. controls (6.95,-1.4) and (3.31,-0.3) .. (0,0) .. controls (3.31,0.3) and (6.95,1.4) .. (10.93,3.29)   ;
\draw  [dash pattern={on 4.5pt off 4.5pt}]  (395.33,249.67) -- (471.57,208.61) ;
\draw [shift={(473.33,207.67)}, rotate = 151.7] [color={rgb, 255:red, 0; green, 0; blue, 0 }  ][line width=0.75]    (10.93,-3.29) .. controls (6.95,-1.4) and (3.31,-0.3) .. (0,0) .. controls (3.31,0.3) and (6.95,1.4) .. (10.93,3.29)   ;
\draw    (182.33,253) -- (218.67,252.37) ;
\draw [shift={(220.67,252.33)}, rotate = 179] [color={rgb, 255:red, 0; green, 0; blue, 0 }  ][line width=0.75]    (10.93,-3.29) .. controls (6.95,-1.4) and (3.31,-0.3) .. (0,0) .. controls (3.31,0.3) and (6.95,1.4) .. (10.93,3.29)   ;
\draw    (68.67,229.67) -- (68.67,133) ;
\draw [shift={(68.67,131)}, rotate = 90] [color={rgb, 255:red, 0; green, 0; blue, 0 }  ][line width=0.75]    (10.93,-3.29) .. controls (6.95,-1.4) and (3.31,-0.3) .. (0,0) .. controls (3.31,0.3) and (6.95,1.4) .. (10.93,3.29)   ;

\draw (15.33,92.67) node [anchor=north west][inner sep=0.75pt]   [align=left] {I: Along Morse geodesics};
\draw (28.33,108.67) node [anchor=north west][inner sep=0.75pt]   [align=left] {(\cref{Ancona on Morse})};
\draw (4.33,237.67) node [anchor=north west][inner sep=0.75pt]   [align=left] {II: Along Morse subsets with};
\draw (4.33,253.67) node [anchor=north west][inner sep=0.75pt]   [align=left] {narrow points (\cref{Ancona on Morse subset})};
\draw (115.67,165.33) node [anchor=north west][inner sep=0.75pt]   [align=left] {Ancona Inequalities};
\draw (90.67,182.33) node [anchor=north west][inner sep=0.75pt]   [align=left] {Abstraction: \cref{Ancona general}};
\draw (227.67,92.33) node [anchor=north west][inner sep=0.75pt]   [align=left] {III: Along geodesics with};
\draw (227.67,107.33) node [anchor=north west][inner sep=0.75pt]   [align=left] {good points (\cref{Ancona on geodesic with good points})};

\draw (223.33,237.67) node [anchor=north west][inner sep=0.75pt]   [align=left] {IV:  With Morse hyperplanes};
\draw (223.33,253.67) node [anchor=north west][inner sep=0.75pt]   [align=left] { in C.C.C. (\cref{Ancona on geodesic in CAT(0)})};
\draw (265.67,162.67) node [anchor=north west][inner sep=0.75pt]   [align=left] {partial boundary map};
\draw (416.67,150) node [anchor=north west][inner sep=0.75pt]   [align=left] {Embed full measure subset};
\draw (430.67,168) node [anchor=north west][inner sep=0.75pt]   [align=left] {$\partial \Psi: \Lambda\subset\partial_h X \longrightarrow \partial^m_{\mathcal M}G$};
\draw (416.67,186) node [anchor=north west][inner sep=0.75pt]   [align=left] {\cref{the map from horofunction to Martin}, \cref{the map from horofunction to Martin in CAT}};
\end{tikzpicture}
    \caption{Logical flow of Ancona inequalities with application to Martin boundary. The C.C.C. stands for CAT(0) cube complex.}
    \label{fig:logicalflow}
\end{figure}
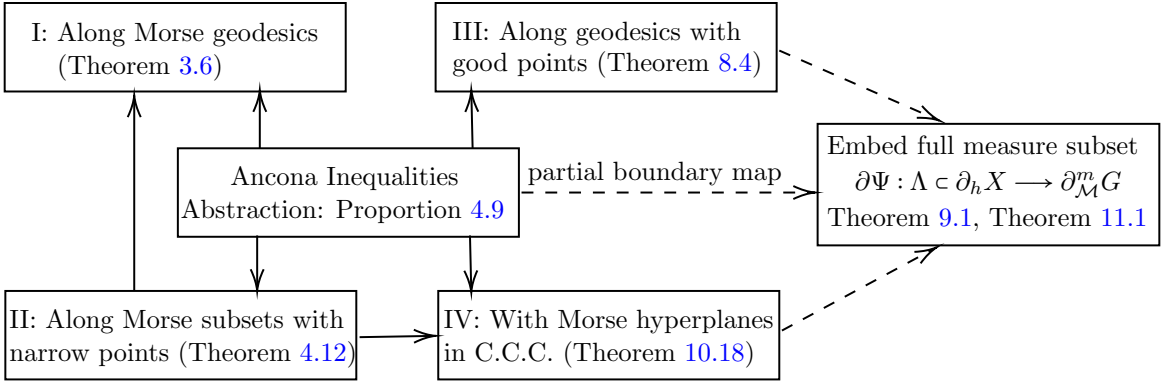

Assume that $G$ acts geometrically on a proper geodesic metric space $(X,d)$ with the base point $o \in X$. Intuitively, we say a subset $Y\subset X$ is \textit{narrow} at a point $z  \in Y$ if $Y$ can be divided into two parts $Y=Y_1\cup Y_2$ such that the geodesic connecting two points in $Y_1,Y_2$ intersects a fixed neighborhood of $z$. See \cref{sec Ancona abstraction} for precise definition. Furthermore, we say that two points $x,y \in X$ are \textit{$k$-antipodal along $(Y,z)$} if  $\pi_Y(x) \subset Y_1, \pi_Y(y) \subset Y_2$ and 
    \[d(x,Y) \leq k \cdot \mathrm{diam}\{\pi_Y(x)\cup z\} \quad \text{and}\quad  d(y,Y) \leq k \cdot \mathrm{diam}\{\pi_Y(y)\cup z\}.\]

\begin{thm}[\cref{Ancona on Morse subset}]
    For any constants $k , \mathfrak{s} \geq 0$ and a Morse gauge $\kappa$, there exists a constant $C =C(\kappa,\mathfrak{s},k)>0$ with the following property.
    Let $Y=Y_1\cup Y_2 $ be a $\kappa$-Morse subset and $\mathfrak{s}$-narrow at a point $zo \in X$. If $xo,yo \in X$ are $k$-antipodal along $(Y,zo)$ for some elements $x,y \in G$, then 
\begin{equation}\label{Ancona(II)Intro}
\tag*{Ancona (II)} C^{-1} \mathcal{G}(x,z)\mathcal{G}(z,y) \leq \mathcal{G}(x,y) \leq C \mathcal{G}(x,z)\mathcal{G}(z,y).    
\end{equation}
\end{thm}

Morse quasi-geodesics are narrow at every point on it (\cref{pass narrow point}). An immediate corollary gives the Ancona inequality along Morse quasi-geodesics in the Cayley graph in \cref{Sec Ancona on Morse geodesics}.
\begin{cor}[\cref{Ancona on Morse}]
    
    Let $\gamma$ be a $\kappa$-Morse $c$-quasi-geodesic in $X$. Then there exists a constant $C = C(\kappa,c)>0$ such that if points $xo, zo, yo$ lie on the quasi-geodesic $\gamma$ in this order for group elements $x,y$ and $z$, then
\begin{equation}\label{Ancona(I)Intro}
\tag*{Ancona (I)} C^{-1} \mathcal{G}(x,z)\mathcal{G}(z,y) \leq \mathcal{G}(x,y) \leq C \mathcal{G}(x,z)\mathcal{G}(z,y). 
\end{equation}
\end{cor}
 
A partial progress was made by Cordes-Gehktman-Dussaule in \cite{cordes2022embedding}, where the same inequalities were obtained for non-elementary hierarchically hyperbolic groups and non-elementary relatively hyperbolic groups whose parabolic subgroups have empty Morse boundary. The above result removes all these assumptions and applies to any group with Morse geodesics.
Using their inequality, the authors proved in  \cite{cordes2022embedding} that the Morse boundary with direct limit topology is topologically embedded  into the Martin boundary. Now using \cref{Ancona on Morse},  their result  extends to arbitrary groups.
\begin{cor}\label{CDGMorseCor}
    Consider an irreducible and finitely supported $\mu$-random walk on a non-amenable group $G$. 
    Then the identity map on $G$ extends to an injective map $\Psi$ from the Morse boundary to the Martin boundary which is continuous with respect to the direct limit topology.
\end{cor}

In \cref{Sec Admissible Sequence}, we introduce admissible sequences to construct Morse subsets with narrow points. Using this construction, we can prove the Ancona inequality in \cref{Ancona on transition points} for geodesics with transition points in relatively hyperbolic groups, which provides an alternative proof for part of \cite[Corollary 1.4]{gekhtman2021martin}. There, they actually derive the same conclusion for measure $\mu$ having super-exponential moment.

\begin{cor}\cite[Corollary 1.4]{gekhtman2021martin}
For any relatively hyperbolic group $G$, there exists a constant $\epsilon_0$ such that for any $\epsilon > \epsilon_0$, there exists a constant $L_0$ such that for any $L \geq L_0$, the following holds.

There exists a constant $C>0$ such that if points $x, z, y$ lie on a geodesic $\gamma$ in this order, and $z$ is an $(\epsilon,L)$-transition point along $\gamma$, then
    \[
    C^{-1} \mathcal{G}(x,z)\mathcal{G}(z,y) \leq \mathcal{G}(x,y) \leq C \mathcal{G}(x,z)\mathcal{G}(z,y).
    \]
\end{cor}

The   strongest version of Ancona inequality in this work are found in a $\CAT$ cube complex containing a Morse hyperplane. Using work of Caprace-Sageev \cite{Caprace_2011}, we are able to find a (actually infinitely many) contracting element $\mathfrak{g} \in G$ \textit{skewering} a Morse hyperplane (\cref{many elements fix regular points}). We say a geodesic segment contains an \textit{$(r,\mathfrak{g})$-barrier} if the geodesic passes through the $r$-neighborhood of a $G$-translate of $[o,\mathfrak{g} o]$ (\cref{def barrier-free}). The crux is that Morse hyperplanes entail the Morse half-spaces partitions of the space, so we can derive Ancona inequalities for all geodesics with $(r,\mathfrak{g})$-barriers from \cref{Ancona on Morse subset}.

\begin{thm}[\cref{Ancona on geodesic in CAT(0)}]\label{intro Ancona in CAT(0)}
    Let a non-elementary group $G$ act geometrically and essentially on an irreducible $\mathrm{CAT(0)}$ cube complex $X$ containing a Morse hyperplane. 
    For any ${r} > 0$, there exists a constant $C>0$ and a contracting element $\mathfrak{g} \in G$ satisfying the following property.
    Let $x,z,y \in G$ be three elements such that $z$ is an $({r},\mathfrak{g})$-barrier of the geodesic $[xo,yo]$. Then 
\begin{equation}\label{Ancona(IV)Intro}
\tag*{Ancona (IV)} C^{-1} \mathcal{G}(x,z)\mathcal{G}(z,y) \leq \mathcal{G}(x,y) \leq C \mathcal{G}(x,z)\mathcal{G}(z,y).
\end{equation}
\end{thm}

Furthermore, we verify that the action of certain right-angled Coxeter groups (RACGs) on their Davis complexes satisfies the hypotheses of \cref{intro Ancona in CAT(0)}. Specifically, this includes the class of RACGs whose defining graphs are irreducible and contain at least one vertex not belonging to any induced 4-cycle (\cref{RACG with Morse hyperplane}). The  Ancona inequality (IV) is the new ingredient (compared with \cref{the map from horofunction to Martin: intro})  to get the continuous extension of the partial boundary maps in \cref{the map for RACG: intro}.

\subsection{On the proofs of Theorems \ref{the map from horofunction to Martin: intro} and  \ref{the map for RACG: intro}}\label{subsec proofs}

The overall strategy is to use Ancona inequalities in \cref{subsec ancona ineuqalities intro} along a carefully selected class of generic geodesic rays, as described below, to establish the desired boundary maps.  

First, we introduce a distinguished class of geodesic rays termed {frequently contracting rays}, which  contain a sequence  of pairwise disjoint and long contracting segments (\cref{def frequently contracting}). Precisely, a geodesic ray $\gamma:[0,\infty)\to X$ is called \textit{$(D,L)$-frequently contracting} for $D,L>0$ if $\gamma$ contains a sequence of disjoint $D$-contracting segments $p_n$ with length greater than $L$. These rays are shown to converge to a unique point in the horofunction boundary, up to a finite difference  (\cref{frequently contracting to horofunction boundary}). 

The next step is to quantify the concept of frequently contracting rays: $\gamma$ is  further called \emph{$\theta$-proportionally $(D,L)$-contracting}  if  the proportion of the total length of $D$-contracting segments with length $\ge L$ in a sequence of initial segments  $\gamma[0,s_n]$ with $s_n\to\infty$ is greater than $\theta>0$. See \cref{def prop. contracting rays} for precise formulation. Under the hypothesis of a geometric group action admitting contracting elements, we prove  that for appropriate $\theta$ and $L\gg D$, the endpoints of $\theta$-proportionally $(D,L)$-contracting rays in the horofunction boundary are generic in Patterson-Sullivan measure. This is contained in \cref{the map is defined on a full mearsure subset} and uses extensively  the  work of \cite{yang2022conformal}.  

In the proof of \cref{the map from horofunction to Martin: intro}, the following result is essential.
\begin{thm}[\cref{good points on proportionally geodesic ray}+\cref{Ancona on geodesic with good points}]
For any $\theta\in (0,1]$ and $L\gg D$, every $\theta$-proportionally $(D,L)$-contracting  ray $\gamma$ contains an unbounded sequence of \emph{good points} $\tilde z_n$  so that for any points $xo,yo \in X$ that are $k$-antipodal along $(\gamma, z_n o)$ with $d(z_no,\tilde z_n)\le \varepsilon$,
\begin{equation}\label{Ancona(III)Intro}
\tag*{Ancona (III)} C^{-1} \mathcal{G}(x,z)\mathcal{G}(z_n,y) \leq \mathcal{G}(x,y) \leq C \mathcal{G}(x,z_n)\mathcal{G}(z_n,y)
\end{equation}
where $ \varepsilon$ is a universal constant given by the geometric action and $C$ depends only on $\theta, L,D$.
\end{thm}
Roughly speaking, a  good point $z_n$ here means that   \textit{any subpath $p$} longer than a fixed number on the ray starting from the good point  is proportionally contracting: the total length of $10D$-contracting segments with length $\ge L$ in $p$ takes a proportion at least $\theta/4$. See \cref{def good points} for a precise definition. Note that not every points in proportionally contracting rays are good.  Good points are good in the following sense: the Ancona inequalities hold for any pair of $k$-antipodal points around the good points, even though  proportionally contracting rays are not Morse in general (so \cref{Ancona on Morse subset} does not apply).
\begin{rem}
Experts would notice that the good points share certain similarity with the notion of linear progress points \cite[before Lemma 12.2]{MS20}, where a version of Ancona inequality is stated. Note that ours uses only contracting property of metric spaces and theirs requires an acylindrical-like action on hyperbolic spaces. Our  Ancona inequality (III) proved in \cref{Ancona on geodesic with good points} holds for a more general triple of points (namely $k$-antipodal points), which are not  necessarily on the geodesic in theirs. Another difference is \cref{good points on proportionally geodesic ray} that finds \emph{infinitely many} good points, which is crucial to define the boundary map.         
\end{rem}  

Using work of \cite{cordes2022embedding} and many good points, we prove  that each proportionally contracting ray converges to a unique point in the minimal Martin boundary, thereby inducing a well-defined mapping from the horofunction boundary to the minimal Martin boundary as described in \cref{def from horofunction to Martin}.  

The continuity of the boundary map requires more technical treatment. We use a stratification of proportionally contracting rays where each stratum exhibits uniform control over the distribution of good points.  
This stratified structure facilitates the proof of continuity with respect to the direct limit topology (\cref{continuity for prop. contracting rays}).

In the proof of \cref{the map for RACG: intro}, we employ the notion of $(r,F)$-\emph{admissible rays}  (\cref{admissible ray}) in geometric actions on an irreducible CAT(0) cube complex with Morse hyperplanes. These are rays containing an infinite sequence of $(r,F)$-barriers for appropriate parameter $r>0$ and a particular set $F$ of  three sufficiently long and independent contracting elements, where each element in $F$ skewers a Morse hyperplane. Existence of such $F$ uses the irreducibility of CAT(0) cube complex (\cref{many contracting elements skewer}) and has the following two consequences. 

First, we observe that admissible rays constitute a (still generic) subclass of frequently contracting rays, hence each converges to a locus in the horofunction boundary  as above-mentioned. This induces a mapping from a full-measure subset of the horofunction boundary to the minimal Martin boundary. 

Second, more importantly, the $(r,F)$-barriers allow us to use the strengthened Ancona inequalities along admissible rays (\cref{Ancona in CAT(0)}),  which demonstrate that this boundary map (and its inverse) is  a continuous extension of the identity mapping on the interior (\cref{the conical case is continuous} and \cref{inverse continuous}).

\subsection*{Organization of the paper}

In \cref{Sec Ancona on Morse geodesics}, we prove the Ancona inequality along Morse quasi-geodesics (\cref{Ancona on Morse}). Building on this result, \cref{sec Ancona abstraction} develops abstract conditions for establishing the Ancona inequality, which we apply to several key settings.

First, in \cref{Sec Ancona on Morse subset}, we introduce Morse subsets with narrow points and prove the associated Ancona inequality(\cref{Ancona on Morse subset}). The framework is further expanded in \cref{Sec Admissible Sequence}, where admissible sequences are shown to construct such Morse subsets, enabling the proof of Ancona inequalities along geodesics with transition points in relatively hyperbolic groups (\cref{Ancona on transition points}).

Having established the generalized Ancona inequality, we now examine its consequences for the Martin boundary. \Cref{sec preliminary of  boundary} provides the necessary background on both Martin and horofunction boundaries.

Next, we consider groups acting geometrically with contracting elements. In \cref{sec prop contracting geodesic}, we define proportionally contracting geodesics and establish their genericity.
The following \cref{Sec proportionally contracting geodesic} establishes the Ancona inequality for proportionally contracting rays (\cref{Ancona on geodesic with good points}).
The key consequence, presented in \cref{Sec embedding}, is that the linear neighborhood of any such geodesic ray converges to a minimal Martin boundary point(\cref{EFCG converge}), which yields an embedding of a full measure subset of horofunction boundary into Martin boundary (\cref{the map from horofunction to Martin}).

Finally, we specialize to geometrical action on an irreducible $\CAT$ cube complexes containing Morse hyperplanes. \cref{sec Ancona in CAT(0)} proves Ancona inequalities for geodesics with a barrier (\cref{Ancona on geodesic in CAT(0)}) and show some dynamic feature of the action on Martin boundary (\cref{subsec Minimal actions on Martin boundary}). In \cref{sec The map (II)}, we embed a full-measure subset of the Roller boundary into the Martin boundary and present applications to RACGs.

\subsection*{Acknowledgments}
We are grateful to Jason Behrstock for insightful comments about RACGs, Mathieu Dussaule for answering us questions regarding \cite{cordes2022embedding} and Mark Hagen and Alex Sisto for helpful discussions on Morse separations. 
We also thank Shaobo Gan and Jiarui Song for helpful conversations on good points.

\section{Preliminaries}\label{Preliminaries}

\subsection{Convention and notations}
In this paper, we mostly consider a proper geodesic metric space  $(X,d)$. That is, there exists a (possibly non-unique) geodesic connecting any pair of
points $x,x'\in X$.
Let $Y$ be a subset of $X$. Denote  $\mathrm{diam}(Y) := \text{sup}\{d(x,y):x\in Y,y \in Y\}$ and $N_D(Y) := \{ x : d(x, Y) \leq D \}$.
The \textit{Hausdorff distance} between subsets $Y$ and $Z$ of $X$, denoted by $d_{Haus}(Y,Z)$, is the infimum of $C \geq 0$ such that
$Y \subset N_C(Z)$ and $Z \subset N_C(Y)$.

Denote the ball with center $x$ of radius $r$ as $B(x,r)$.
Given a constant $c\geq 1$, a map $\varphi: X \to Y$ between metric spaces is called a \textit{$c$–quasi-isometric embedding} if for every $x,x' \in X$,
\[c^{-1}d(x,x')-c \leq  d(\varphi(x),\varphi(x')) \leq  c d(x,x') + c.\]
It is called a \textit{$c$–quasi-isometry} if, in addition, $Y= N_D(\varphi(X))$ for some $D>0$.
A $c$–quasi-isometric embedding of an interval is called a $c$–quasi-geodesic.

Let $G$ be  a finitely generated group  with a finite symmetric generating set $S$. Assume that the group identity $e \notin S$ and $S = S^{-1}$. The Cayley graph $\mathrm{Cay}(G, S)$ is a graph where the vertices are the elements of $G$, and the edges connect pairs $(x, xs)$ for $x \in G$ and $s \in S$. 
The word distance $d_S(x, y)$ is defined as the graph distance on $\text{Cay}(G, S)$, which is the number of edges in a shortest path between the vertices $x$ and $y$.

\subsection{Contracting and Morse subsets}

Let $Z$ be a closed subset of $ X$ and $x$ be a point in $ X$. We define the set of nearest-point projections from $x$ to $Z$ as follows
\[ \pi_{Z}(x) : = \big \{ y\in Z: d(x, y) = d(x, Z) \big \} \]
where  $d(x, Z) : = \inf \big \{ d(x, y): y \in Z \big \}.$
Since $X$ is a proper metric space, 
$\pi_{Z}(x)$ is non empty.  

\begin{defn} \label{Def:Contracting}
We say a closed subset $Z \subset X$ is \emph{$C$--contracting} for a constant $D>0$ if,
for all pairs of points $x, y \in X$, we have
\[
d(x, y) \leq d(x, Z) \quad  \Longrightarrow  \quad \mathrm{diam}\{\pi_{Z}(x)\cup \pi_{Z}(y)\} \leq  D.
\]
Any such $D$ is called a \emph{contracting constant} for $Z$.
\end{defn}

We summarize several key properties that will be frequently used in subsequent arguments. These results follow  from the contracting property, and their proofs are omitted for brevity (see \cite{yang2014growthtightness} and references therein).
\begin{lem}\label{contracting properties}
Let $Z$ be a closed $D$-contracting subset for some $D\ge 0$. Then 
\begin{enumerate}
    \item 
    There is a constant still denoted as $D$ depending only on $D$ so that any geodesic outside $N_{D}(Z)$ has $D$-bounded projection to $Z$. 
    \item 
    \cite[Lemma 2.12]{wan2025marked} If $W$ is a closed subset with Hausdorff distance at most $C$ to $Z$, then $W$ is $(6C+7D)$-contracting.
    \item For $x,y \in X$ with $\bar x \in \pi_{Z}(x), \bar y \in \pi_{Z}(y)$, if $d(\bar x, \bar y) \geq 4D$, then $\bar x \subset N_{2D}([x,y])$ and $\bar y \subset N_{2D}([x,y])$ for any geodesic $[x,y]$.
    \item If $x\in Z, y \in X$, then $\pi_{Z}(y) \subset N_{2D}([x,y])$.
    \item $\diam \{ \pi_Z(x) \cup \pi_Z(y) \} \leq d(x,y) + D, \; \forall x,y \in X$.
\end{enumerate}
\end{lem}

\begin{lem}\cite[Prop 2.2]{yang2020genericity}\label{subgeodesic of contracting geodesic}
    If $\alpha$ is a $D$-contracting geodesic  for some $D>0$, then any subsegment of $\alpha$ is $2D$-contracting geodesic.
\end{lem}

\begin{lem}\label{geodesic near a contracting geodesic}
    Let $\alpha$ be a $D$-contracting geodesic for some $D>0$. Let  $x, y\in X$ be two points with $d(x,\alpha) \leq \epsilon, d(y,\alpha) \leq \epsilon$ for $\epsilon>0$. Then $[x,y]$ is a $(86D + 24\epsilon)$-contracting geodesic, and  $[x,y] \subset N_{6D+2\epsilon}(\alpha)$.
\end{lem}

\begin{proof}
    Set  $D'=2D$. By \cref{subgeodesic of contracting geodesic} and taking a sub-geodesic, we can assume that $\alpha$ is a $D'$-contracting geodesic with endpoints $\alpha_{-}, \alpha_{+}$ and $d(x,\alpha_{-}) \leq \epsilon, d(y,\alpha_{+}) \leq \epsilon$. First we show $[x,y] \subset N_{3D' + 2\epsilon}(\alpha)$. 

    Indeed, let  $[w,v]$ be a connected component of  $[x,y]\setminus N_{D'+\epsilon}(\alpha)$, which must exist,  so $d(w,\alpha) = d(v, \alpha) =  D'+\epsilon$ and $d([w,v],\alpha) \geq   D'+\epsilon$. If $[w,v]$ is not contained in $N_{3D' + 2\epsilon}(\alpha)$, then $d(w,v)>2\epsilon + 4D'$. On the other hand,  by \cref{contracting properties}(1), $\diam\{\pi_{\alpha}([w,v])\} \leq D'$ and  $d(w,v) \leq d(w,\alpha) + d(v,\alpha) + \diam\{\pi_{\alpha}([w,v])\} \leq 2\epsilon + 3D'$, giving a contradiction. Hence, $[w,v]\subseteq N_{3D' + 2\epsilon}(\alpha)$ and $[x,y] \subset N_{3D' + 2\epsilon}(\alpha)$ follows.

    Notice that $\alpha$ and $[x,y]$ are two geodesics, and $d(x,\alpha_{-}) \leq \epsilon, d(y,\alpha_{+}) \leq \epsilon, [x,y] \subset N_{3D' + 2\epsilon}(\alpha)$. Using a connected argument we can derive $\alpha \subset N_{6D' + 4\epsilon}([x,y])$. Then $[x,y]$ is $(86D+24\epsilon)$-contracting by \cref{contracting properties}(2).
\end{proof}

Contracting subsets provide natural examples with Morse property we now define. 

\begin{defn}[Morse property]
    A subset $Y$ is \textit{$\kappa$-Morse} for a function $\kappa$: $[1,\infty) \rightarrow [0,\infty)$ if for every $c$-quasi-geodesic $\alpha$ with endpoints on $Y$ remains within the $\kappa(c)$-neighborhood  of $Y$.
\end{defn} 

The main theorem in \cite{arzhantseva2017characterizations} tells us the Morse property is equivalent to the sublinear contracting property.
A function $\rho: [0,\infty) \rightarrow [0,\infty)$ is called \textit{sublinear}, if $\lim_{r \to \infty} \rho(r)/r = 0$.

\begin{lem}\cite[Theorem 1.4]{arzhantseva2017characterizations}\label{Morse is sublinear contracting}
    Let $Y$ be a subset of a geodesic metric space $X$. The following are equivalent:
     \begin{enumerate}
        \item 
        There exists $\kappa: [1,\infty) \rightarrow [0,\infty)$, such that $Y$ is $\kappa$–Morse.
        \item 
        There exists a sublinear function $\rho$ such that $\diam \{\pi_Y(B(x,d(x,Y)))\} \leq \rho(d(x,Y)) $.       
    \end{enumerate}
    Moreover, we bound the parameters of the conclusion in terms of the parameters of the hypothesis, independent of $Y$.
\end{lem}

In particular, every contracting subset has Morse property.
Now, we consider the divergence of Morse subset. The following projection estimates follow from the sublinear contraction of Morse subsets. Let $\|\alpha\|$ denote the length of a path $\alpha$.

\begin{lem}\label{path far away from Morse subset}
    Let $Y$ be a $\kappa$-Morse subset in $X$. For any constant $K>0$, there exists a constant $D=D(\kappa,K)>0$ such that for any path $\alpha$ with endpoints $x,y$, if $d(\alpha, Y) \geq D$, then 
    \[  d(\alpha,Y) + \|\alpha\| \geq K\cdot \mathrm{diam}\{\pi_Y(x)\cup \pi_Y(y)\}. \]
\end{lem}

\begin{proof}
    We can find points $ x = x_0, x_1, \ldots, x_m $ in the natural order along the path $ \alpha $ such that $ d(x_i, x_{i+1}) = d(x_i, Y) $ for $0\le i\le m-1 $, and $ d(x_m, y) \leq d(x_m, Y) $. According to \cref{Morse is sublinear contracting}, there exists a sublinear function $ \rho $ depending on $\kappa$ such that 
    \[
    \mathrm{diam}\{\pi_Y(p)\cup \pi_Y(q)\} \leq \rho(d(p, Y))
    \]
    for any $ p, q \in X $ with $ d(p, q) \leq d(p, Y) $. 
    Then we have $\mathrm{diam}\{\pi_Y(x_i)\cup\pi_Y(x_{i+1})\} \leq \rho(d(x_i,Y))$ for $ 0 \leq i \leq m-1$ and $\mathrm{diam}\{\pi_Y(x_m)\cup\pi_Y(y)\} \leq \rho(d(x_m,Y))$. 
    As noted in the paragraph following Proposition 2.1 of \cite{aougab2017pulling}, we may assume that $ \rho(r)/r $ is non-increasing and tends to zero.
    Since $d(x_i,Y) \geq d(\alpha,Y) \geq D$, we obtain
    \[ \frac{\rho(d(x_i,Y))}{d(x_i,Y)} \leq \frac{\rho(D)}{D}. \]
    Thus, we get
    \begin{align*}
        \mathrm{diam}\{\pi_Y(x)\cup\pi_Y(y)\} 
        &\leq \sum_{i = 0}^{m-1} \mathrm{diam}\{\pi_Y(x_i)\cup\pi_Y(x_{i+1})\} + \mathrm{diam}\{\pi_Y(x_m)\cup\pi_Y(y)\} \\
        &\leq \sum_{i = 0}^{m-1} \rho(d(x_i,Y)) + \rho(d(x_m,Y)) \\
        &\leq \sum_{i = 0}^{m-1} \frac{\rho(D)}{D}  d(x_i,Y) + \frac{\rho(D)}{D}  d(x_m,Y) .
    \end{align*}
    The sum $\sum_{i = 0}^{m-1}d(x_i,Y) = \sum_{i = 0}^{m-1}d(x_i,x_{i+1})$ is bounded by the total length $\|\alpha\|$, and the distance $d(x_{m},Y) $ is bounded by $\|\alpha\| + d(\alpha,Y)$ using triangle inequality.
    Hence, we have
    \begin{align}\label{path away long}
        \mathrm{diam}\{\pi_Y(x)\cup\pi_Y(y)\} \leq \frac{\rho(D)}{D} (\|\alpha\| + \|\alpha\|+ d(\alpha,Y)) \leq  \frac{2\rho(D)}{D} (\|\alpha\| + d(\alpha,Y)).
    \end{align}
    We can choose  $D$ large enough such that $\frac{2\rho(D)}{D} \leq \frac{1}{K}$.
\end{proof}

With an eye to \cref{defofdivergence}, we record the following corollary for Morse geodesics.
\begin{cor}\label{main constant control}
    Let $\gamma$ be a $\kappa$-Morse geodesic. For any $c \geq 1$ , there are constants $D = D(c,\kappa), R = R(c,\kappa)$ such that the following holds.

    Let $\alpha$ be any path intersecting $N_D(\gamma)$ only at its endpoints $g,h\in N_{D}(\gamma)$.   Let $\overline{g}$ (resp. $\overline{h}$) be a closest point to $g$ (resp. $h$) on $\gamma$. If  $d(g,h) \geq R$, then $\|\alpha\| \geq c d(\overline{g}, \overline{h})$.
\end{cor}

\begin{proof}
    Let $D = D(\kappa,2c)$ from \cref{path far away from Morse subset}. Then we only need to prove $ D = d(\alpha,\gamma) < c d(\overline{g}, \overline{h})$. Set $R = D/c + 2D$ such that $d(\overline{g}, \overline{h}) \geq d(g,h) - 2D \geq R-2D \geq D/c$.
\end{proof}

\subsection{Groups with contracting elements}\label{sec groups with contracting}
Let $G$ be a countable group which is non-elementary. Assume that $G$ acts isometrically and properly on a geodesic metric space $X$.
An element $h\in G$ of infinite order is called \textit{contracting} if for any $o\in X$, the orbit map $n\in\mathbb Z\mapsto h^no\in X$ is a quasi-isometric embedding and the image $\{h^no: n\in \mathbb Z\}$ is a contracting subset in $X$.

\begin{lem}\cite[Lemma 2.11]{yang2019statistically}\label{elementarygroup}
For any contracting element $h\in G$, the cyclic subgroup $\langle h \rangle$ has finite index in  
$$
E(h)=\{g\in G: d_{Haus}(\langle h \rangle o, g\langle h \rangle o ) < \infty \}.
$$
\end{lem}

Keeping in mind the base point $o\in X$, the \textit{axis} of $h$  is defined as the following quasi-geodesic 
\begin{equation}\label{axisdefn}
\ax(h)=\{f o: f\in E(h)\}.
\end{equation} Notice that $\ax(h)=\ax(k)$ and $E(h)=E(k)$    for any contracting element   $k\in E(h)$.
 
Let $\f$ be a contracting system, that is a family of uniformly contracting subsets.
In this paper, we are interested in the  \textit{$\tau$-bounded intersection} property for a function $\mathcal R: \mathbb R_{\ge 0}\to \mathbb R_{\ge 0}$ so that the
following holds
$$\forall U\ne V\in \f: \quad \mathrm{diam}(N_r (U) \cap N_r (V)) \le \mathcal R(r)$$
for any $r \geq 0$. This is, in fact, equivalent to a \textit{bounded projection
property} of $\f$:  there exists a constant $B>0$ such that the
following holds
$$\mathrm{diam}(\pi_{U}(V)) \le B$$
for any $U\ne V \in \f$.  See \cite[Lemma 2.3]{yang2014growthtightness} for a proof.

Two  contracting elements $h_1, h_2\in G$  are called \textit{independent} if $E(h_1)\ne E(h_2)$ and the collection $\{g\ax(h_i): g\in G;\ i=1, 2\}$ is a contracting system with bounded intersection. Note that two conjugate contracting elements with disjoint fixed points are not independent in our sense.
A group $G$ is called \textit{elementary} if it is virtually a cyclic group. 
\begin{lem}\cite[Lemma 4.6]{yang2014growthtightness}\label{non-elementary}
Assume that $G$ is a non-elementary group with a contracting element. Then $G$ contains infinitely many pairwise independent contracting elements.     
\end{lem}

If $G$ is non-elementary, then $G$ contains at least two independent contracting elements. A ping-pong argument constructs    a free group of rank 2 in $G$, so $G$ is non-amenable.

\subsection{Random walks on groups}
Let $\mu$ be a probability measure on $G$, and let $\gamma = (g_0, g_1, g_2, \dots, g_n)$ be a finite sequence of group elements. We call $\gamma$  a \textit{$\mu$-trajectory} (or simply \textit{trajectory} if $\mu$ is understood) of length $n$ if its weight defined below is positive:
\[
\omega(\gamma) =  p(g_0, g_1) p(g_1, g_2) \cdots p(g_{n-1}, g_n).
\]
If there exists another trajectory $\gamma' = (g_n, g_{n+1}, \dots, g_m)$ with initial point $g_n$, we denote the concatenation by $\gamma \circ \gamma' = (g_0, g_1, \dots, g_m)$, which is clearly also a trajectory. It is easy to verify that $$\omega(\gamma \circ \gamma') = \omega(\gamma) \cdot \omega(\gamma').$$ Given any set $K$ of trajectories, let $\omega(K)$ denote the sum of the weights of all trajectories in $K$. 

Let $W(x, y)$ denote the set of all trajectories from $x$ to $y$, and let $W(x, y; \Omega)$ denote the set of all trajectories $\gamma = (x, g_1, \dots, g_{n-1}, y)$ with $g_i\in \Omega$.

The \textit{Green function} is given by
\[
\mathcal{G}(x, y) = \sum_{n=0}^{\infty}  p_n(x, y) = \sum_{\gamma \in W(x, y)} \omega(\gamma) = \omega( W(x, y)).
\]

\begin{rem}
As mentioned in the introduction, Kesten proved in \cite{kesten1959full} that a finitely generated group $G$ is non-amenable if and only if the spectral radius of every symmetric irreducible random walk is strictly less than 1. This result was partially extended in \cite{day1964convolutions}, showing that every irreducible random walk on a finitely generated non-amenable group has a spectral radius less than 1. According to a result by Guivarc'h (see \cite[p.85, remark b]{AST_1980__74__47_0}), if the group is non-amenable, then the Green function can be defined for all $r \in [1, \rho(\mu)^{-1}]$. 
In this work, we assume that group $G$ is non-amenable, which ensures that $\mathcal{G}(x, y)$ is finite.    
\end{rem}
 
The \textit{restricted Green function} with respect to a subset $\Omega$ is defined as 
\[
\mathcal{G}(x, y; \Omega) = \sum_{\gamma \in W(x, y; \Omega)} \omega(\gamma).
\]
Given three points $x, y, z$, every trajectory from $x$ to $y$ passing through $z$ can be uniquely split at the first occurrence of $z$. It follows that 
\begin{equation}\label{basicGreenfunction}
 \begin{array}{c}
  \mathcal{G}(x, y) \geq \mathcal{G}(x, z; \{z\}^c) \mathcal{G}(z, y),\\
    \mathcal{G}(x, z) = \mathcal{G}(x, z; \{z\}^c) \mathcal{G}(z, z) = \mathcal{G}(x, z; \{z\}^c) \mathcal{G}(e, e).
\end{array}   
\end{equation}

Thus, we also have
\begin{equation}
    \mathcal{G}(x, y) \mathcal{G}(e, e) \geq \mathcal{G}(x, z) \mathcal{G}(z, y).
    \label{Greenfunction}
\end{equation}
If we fix an upper bound of $d_S(z,y)$, then $z^{-1}y$ is contained in a finite ball. Therefore, $\mathcal{G}(z, y) = \mathcal{G}(e, z^{-1}y)$ is bounded by a function of $d_S(e,z^{-1}y) = d_S(z,y)$. So there exists an increasing function $f$ depending on $G$ and $\mu$ such that 
\begin{equation}
    \frac{1}{f(d_S(y,z))} \leq \frac{\mathcal{G}(x, y)}{\mathcal{G}(x, z)}  \leq  f(d_S(y,z)) \bigand \frac{1}{f(d_S(x,z))} \leq \frac{\mathcal{G}(x, y)}{\mathcal{G}(z, y)}  \leq  f(d_S(x,z)).
    \label{like Harnack}
\end{equation}

First, we introduce some estimates for the weights of trajectories from \cite{gekhtman2021martin}. Let $W_{\ge M}(x, y) = \{\gamma \in W(x, y) : \|\gamma\| \geq M\}$ denote the subset of trajectories from $x$ to $y$ with length at least $M$.

\begin{lem}\label{weight bound}
      Consider an irreducible and finitely supported $\mu$-random walk on a group $G$ with spectral radius $\rho(\mu) < 1$. Set $r < \rho(\mu)^{-1}$ and fix a symmetric generating set $S$. Then there exist constants $\phi \in (0,1)$ and $L > 1$ such that for all $ x\neq y \in G$ and $ M > 0$, the following holds
    \begin{align}
    \label{lwb}
        \omega(\gamma) &\geq L^{-\|\gamma\|}, \;\;\forall    \gamma \in W(x,y),\\
    \label{upb}
         \mathop{\sum}_{\gamma \in W_{\ge M}(x,y)} \omega(\gamma) &\leq \phi^M  L^{d_S(x,y)}.
    \end{align}

\end{lem}

\begin{proof}
    These estimates come from the Harnack inequality. We refer to \cite{gekhtman2021martin} for details, where the inequality \cref{lwb} corresponds to the case $t = 1$ in Corollary 2.2. Inequality \eqref{upb} is established in the proof of Proposition 2.3 of \cite{gekhtman2021martin}. 
\end{proof}

\section{Ancona inequality (I) along Morse quasi-geodesics}\label{Sec Ancona on Morse geodesics}

In this paper, we shall prove several versions of the Ancona inequalities. This section gives the basic version in \cref{Ancona on Morse}, which follows immediately from the one in next section. We include it here for the proof motivates the others and also it was the starting point of this investigation. 

We introduce some terminologies here. Let $\mathrm{Cay}(G, S)$ denote the Cayley graph of a group $G$ with respect to a finite symmetric generating set $e\notin S$. 
Given a sequence $\gamma = (g_0, g_1, \dots, g_n)$, we say that $\gamma$ is a (discrete) \textit{geodesic} if $d_S(g_i, g_j) = |i - j|$ for all $0 \leq i, j \leq n$.
We say that $\gamma$ is a (discrete) $c$-\textit{quasi-geodesic} for $c \geq 1$ if 
\[
c^{-1} (j - i) - c \leq d_S(g_i, g_j) \leq c (j - i) + c, \quad \forall 0 \leq i \leq j \leq n.
\]
We call $\gamma$ a path if $d_S(g_i, g_{i+1}) = 1$ for $0 \leq i \leq n-1$. If $g_i, g_j$ ($j \geq i$) are two points on $\gamma$, we denote the sub-path with endpoints $g_i$ and $g_j$ by $[g_i, g_j]_{\gamma} = (g_i, g_{i+1}, \dots, g_j)$.

For the sake of simplicity in the proof, we assume that the generating set $S$ contains the support of the measure $\mu$.
Thus, every trajectory corresponds to a path in the Cayley graph, but not every path is a trajectory due to the possible asymmetry of $\mu$.
To address this, we introduce the following notation to modify a path into a trajectory within a uniformly bounded Hausdorff distance.

For each generator $s\in S$, we fix a shortest $\mu$--trajectory from the identity element $e$ to $s$, denoted by $\tau(s)$. This extends to any path $\gamma = (g_0,g_1,g_2,\dots,g_{n})$ in the Cayley graph: $\tau(\gamma)$ is the concatenation of the trajectories $\tau(g_{i-1}^{-1}g_i)$ for $1\le i\le n$ in the natural order.  By construction, $\tau(\gamma)$ is a trajectory with a Hausdorff distance from $\gamma$ bounded above by a constant 
\[
\Lambda = \max\{\|\tau(s)\| : s \in S\} \geq 1.
\]
And the length of $\tau(\gamma)$ has the bound $\|\gamma\| \leq\|\tau(\gamma)\| \leq \Lambda\|\gamma\| $.


\begin{lem}\label{lemma:main}
Let $\gamma$ be a $\kappa$-Morse geodesic in the Cayley graph. Then there exists a constant $R = R(\kappa)$ such that for any three points $x,z,y \in \gamma$ in this order, we have
    $$
    \mathcal{G}(x,y;B(z,R)^c) \leq \frac{1}{2} \mathcal{G}(x,y)
    .$$
\end{lem}

\begin{proof}
Let $\phi$ and $L$ satisfy the inequalities \eqref{lwb} and \eqref{upb} in \cref{weight bound}. Fix two positive constants $N$ and $D$ subject to the following conditions:
\begin{equation}\label{lwb of D}
\phi^{N} L^{4\Lambda} \leq \frac{1}{2}, \quad \text{and} \quad D > D_1(N,\kappa) + 100\Lambda.
\end{equation}
We then choose $R$ large enough such that 
\begin{equation}\label{lwb of R}
    R > \max\{10D + R_1(N,\kappa), 6D + \log_2(2\Lambda^2 \cdot D^2)\}.
\end{equation}
Here $D_1$ and $R_1$ are two constants from \cref{main constant control}.

\begin{figure}[ht]
    \centering
    \def\svgwidth{0.8\columnwidth}
    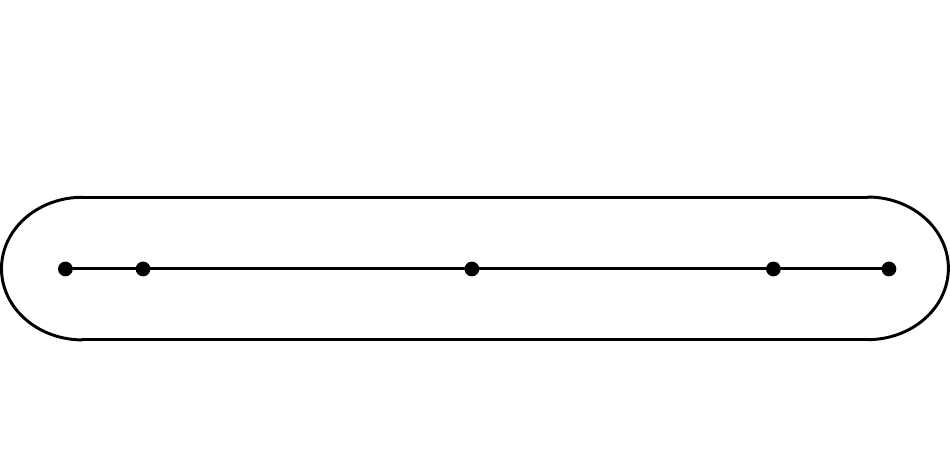

    \label{morsegeodesic}
\end{figure}

The ball $B(z, R)$ centered at $z\in\gamma$ divides   the $D$-neighborhood of $\gamma$   into two parts:
\[
\Omega_1 := N_D([x, z]_{\gamma}) \setminus  B(z, R) \quad \text{and} \quad \Omega_2 := N_D([z, y]_{\gamma}) \setminus  B(z, R).
\]
We note that $\Omega_1, \Omega_2$ are sufficiently separated.
\begin{claim}
$d_S(\Omega_1, \Omega_2) \geq 2R - 4D$.    
\end{claim}
\begin{proof}
Indeed, if not, we would have $d_S(z_1, z_2) < 2R - 4D$ for some $z_1 \in \Omega_1$ and $z_2 \in \Omega_2$. Let $\overline{z}_1 \in \gamma_{[x, z]}$ and $\overline{z}_2 \in \gamma_{[z, y]}$ be points such that $d_S(\overline{z}_1, z_1), d_S(\overline{z}_2, z_2) \leq D$. This implies that $d_S(\overline{z}_1, \overline{z}_2) < 2R - 2D$. Since $z_1$ and $z_2$ are not in $B(z, R)$, we have $d_S(\overline{z}_1, z) \geq d_S(z_1, z) - d_S(\overline{z}_1, z_1) \geq R - D$ and similarly $d_S(\overline{z}_2, z) \geq R - D$. This contradicts the fact that $\overline{z}_1, z, \overline{z}_2$ lie on the geodesic $\gamma$ in this order and that $d_S(\overline{z}_1, \overline{z}_2) < 2R - 2D$.    
\end{proof} 

For any $g \in \Omega_1$ and $h \in \Omega_2$, let $\overline{g}$ be a closest point to $g$ on $\gamma$, and $\overline{h}$ be a closest point to $h$ on $\gamma$. We construct a trajectory $\mathrm{Tr}(g, h) \in W(g, h)$  by concatenating three trajectories:
\[
\mathrm{Tr}(g, h) = \tau([g, \overline{g}]) \circ \tau([\overline{g}, \overline{h}]_{\gamma}) \circ \tau([\overline{h}, h]),
\]
Note that $\mathrm{Tr}(g, h)$ passes through the ball $B(z, R)$.
Recall that $\omega(K)$ denotes the sum of the weights of all trajectories in $K$. 
\begin{claim}
For any $g \in \Omega_1$ and $h \in \Omega_2$,
\begin{equation}\label{avoiding ball trs small}
   \mathcal{G}(g,h;B(z,R)^c \cap N_D(\gamma)^c) \leq \frac{1}{2\Lambda^2 \cdot D^2} \omega(\mathrm{Tr}(g,h)).
\end{equation}    
\end{claim}
\begin{proof}[Proof of the claim]
Let $\beta \in W(g, h; B(z, R)^c \cap N_D(\gamma)^c)$ be any trajectory from $g$ to $h$. Since $\beta$ intersects $N_D(\gamma)$ only at endpoints $g, h$, while $d_S(g,h) \geq d_S(\Omega_1,\Omega_2)$, the choice of constants \cref{lwb of D}, \cref{lwb of R} implies  $\|\beta\|  \geq {N} d_S(\overline{g},\overline{h})$ by \cref{main constant control}, so $W(g, h; B(z, R)^c \cap N_D(\gamma)^c) \subset W_{\geq {N} d_S(\overline{g},\overline{h})}(g,h)$.

By the definition of $\Omega_1, \Omega_2$, we have $ \min\{d_S(g,z),d_S(h,z)\} \geq R$ and $\max\{d_S(g, \overline g), d_S(h, \overline h)\} \leq D$, so we have $$d_S(\overline{g}, \overline{h}) = d_S(\overline g,z) + d_S(\overline h,z) \geq 2R-2D$$
and by triangle inequality, 
\begin{align*}
\|\mathrm{Tr}(g,h)\|& \leq \Lambda(d_S(g, \overline g) + d_S(\overline g, \overline h) +d_S(\overline h, h)) \\
&\leq \Lambda(d_S(\overline g, \overline h) + 2D ) \leq 2 \Lambda d_S(\overline{g},\overline{h}).   
\end{align*} 
Recalling $W(g, h; B(z, R)^c \cap N_D(\gamma)^c) \subset W_{\geq {N} d_S(\overline{g},\overline{h})}(g,h)$, using \cref{weight bound}, we get
\begin{align*}
    & \omega(W(g,h;B(z,R)^c \cap N_D(\gamma)^c))/\omega(\mathrm{Tr}(g,h))  \\
    \leq  &\phi^{N d_S(\overline{g},\overline{h})}L^{d_S(g,h)+\|\mathrm{Tr}(g,h)\|} \\
    \leq & \phi^{N d_S(\overline{g},\overline{h})}L^{d_S(\overline{g},\overline{h})+2D+2\Lambda d_S(\overline{g},\overline{h})} \\
    \leq & (\phi^{N }L^{4\Lambda})^{d_S(\overline{g},\overline{h})}
    \leq  (1/2)^{d_S(\overline{g},\overline{h})}.
\end{align*}
Since $d_S(\overline{g},\overline{h}) \geq 2R-6D \geq \log_2(2\Lambda^2 \cdot D^2)$ according to \eqref{lwb of  R},  the inequality \eqref{avoiding ball trs small} is proved.   
\end{proof}

We complete the proof of the lemma by the above claim. 

Let $\beta$ be a trajectory in $W(x, y; B(z, R)^c)$. Denote by $g$ its last exit point in $\Omega_1$ and by $h$ be the first point entering in $\Omega_2$ after $g$ (which exists since $\Omega_1, \Omega_2$ are separated  by the ball $B(z, R)$). Thus,  $[g, h]_{\beta}$ intersects the $D$-neighborhood of $\gamma$ only at endpoints. That is to say, $[g, h]_{\beta} \in W(g, h; B(z, R)^c \cap N_D(\gamma)^c)$.
Set the boundary of $\Omega_1$ and $\Omega_2$ as: \[ \partial \Omega_1 = \{g \in \Omega_1: d_S(g,\gamma) = D\} \bigand \partial \Omega_2 = \{g \in \Omega_2: d_S(g,\gamma) = D\}. \]
Hence, we may estimate the weight of $\beta$ from above: 
\begin{equation}\label{GreenfunctionUBEQ}
\begin{array}{rl}
&\mathcal{G}(x,y;B(z,R)^c) = \omega(W(x,y;B(z,R)^c))\\
&\\
    \leq &\sum_{(g,h)}  \omega(W(x,g;B(z,R)^c))\cdot \omega(W(g,h;B(z,R)^c \cap N_D(\gamma)^c))\cdot \omega(W(h,y;B(z,R)^c)) \\
&\\    
    \leq &\frac{1}{2\Lambda^2 \cdot D^2}  \sum_{(g,h)}  \omega(W(x,g;B(z,R)^c)) \cdot \omega(\mathrm{Tr}(g,h))\cdot \omega(W(h,y;B(z,R)^c)),
\end{array}     
\end{equation}
where the sum is taken over all pairs $(g,h) \in \partial \Omega_1 \times \partial \Omega_2$ and the middle term of the last inequality uses the inequality \eqref{avoiding ball trs small}. 

Now, we wish to combine the three trajectories in $\omega(\cdot)$ of the last line as a single trajectory. Namely,   we define
\begin{equation}\label{joint map}
\begin{array}{rl}
\chi_{g,h}:\quad W(x,g;B(z,R)^c)\times W(h,y;B(z,R)^c) &\longrightarrow  W(x,y)  \nonumber\\
    (\beta_1,\beta_2) &\longmapsto \hat \beta := \beta_1\circ \mathrm{Tr}(g,h)\circ \beta_2.
\end{array}     
\end{equation}
Note that this is an injective map. Let $\mathrm{Im}(\chi_{g,h}) \subset W(g,h)$ be the image of the map $\chi_{a,b}$.
 
We first focus on the summand on a fixed pair $(g, h) \in  \Omega_1 \times  \Omega_2$  in the last line of (\ref{GreenfunctionUBEQ}) :
\begin{align*}
      & \omega(W(x,g;B(z,R)^c)) \cdot \omega(\mathrm{Tr}(g,h))\cdot \omega(W(h,y;B(z,R)^c)) \\
     = &\sum_{\beta_1,\beta_2} \omega(\beta_1)\omega(\mathrm{Tr}(g,h))\omega(\beta_2) \\
    = &\sum_{\beta_1,\beta_2} \omega(\beta_1\circ \mathrm{Tr}(g,h)\circ \beta_2) \\
    \leq & \sum_{\hat\beta \in \mathrm{Im}(\chi_{g,h})}\omega(\hat\beta),
\end{align*}
where the sum is taken over all pairs $(\beta_1,\beta_2)$ in $W(x,g;B(z,R)^c)\times W(h,y;B(z,R)^c)$.
The last inequality follows from the injectivity  of the above map $\chi_{g,h}$.

Summing over all $(g,h) \in \partial\Omega_1\times \partial\Omega_2$, we continue the estimates in (\ref{GreenfunctionUBEQ}) :
\begin{equation}\label{GreenfunctionUBEQ2}
\begin{array}{rl}
    &\mathcal{G}(x,y;B(z,R)^c) \\
    &\\
    \leq &\frac{1}{2\Lambda^2 \cdot D^2} \sum_{(g, h) \in  \partial\Omega_1 \times  \partial\Omega_2}  \omega(W(x,g;B(z,R)^c)) \cdot \omega(\mathrm{Tr}(g,h))\cdot \omega(W(h,y;B(z,R)^c))\\
    &\\
     \leq &\frac{1}{2\Lambda^2 \cdot D^2}  \sum_{\hat\beta \in W(x,y)} \mathcal{A}(\hat\beta) \omega(\hat\beta),
\end{array}       
\end{equation}
where the term $\mathcal{A}(\hat\beta)$ is the number of different pairs $(g,h)\in \partial\Omega_1\times \partial\Omega_2$ such that $\hat \beta \in \mathrm{Im}(\chi_{a,b})$. It now remains to prove that $\mathcal{A}(\hat{\beta})$ admits a uniform upper bound as follows.

\begin{claim}\label{finite many g}
$\mathcal{A}(\hat\beta)\leq \Lambda^2 \cdot D^2$
\end{claim}

\begin{proof}[Proof of the claim]
Recall that the trajectory $\hat{\beta}$ is a path in the Cayley graph. 
For $\hat{\beta} \in \mathrm{Im}(\chi_{g,h})$ with $(g,h) \in \partial\Omega_1 \times \partial\Omega_2$, we have the decomposition:
$\hat{\beta} = \beta_1 \circ \mathrm{Tr}(g,h) \circ \beta_2$
with key observations:
\begin{itemize}
    \item The path $\hat \beta$ only intersects $B(z,R)$ in the sub-path $\mathrm{Tr}(g,h)$.
    \item By construction $\mathrm{Tr}(g, h) = \tau([g, \overline{g}]) \circ \tau([\overline{g}, \overline{h}]_{\gamma}) \circ \tau([\overline{h}, h])$.
    \item The sub-path $\tau([\overline{g},\overline{h}]_{\gamma}) \subset N_\Lambda(\gamma)$ is contained in the neighborhood of $\gamma$.
\end{itemize}
Let $p$ be the first exit point of $\hat{\beta}$ from $N_{2\Lambda}(\gamma)$ after passing through $z$, which is unique on $\hat{\beta}$. By construction, we know that $p \in \tau([\overline{h}, h])$ and $d(p,h) \leq \|\tau([\overline{h}, h])\| \leq \Lambda \cdot D$.
Therefore, the number of possible choices for $h$ is at most the length $\Lambda \cdot D$, and similarly, the number of choices for $g$ is also bounded by $\Lambda \cdot D$. Hence, we conclude that $\mathcal{A}(\hat{\beta}) \leq \Lambda^2 \cdot D^2$.
\end{proof}

By the \cref{finite many g} and the equation (\ref{GreenfunctionUBEQ2}), we get   
\begin{align*}
    &\mathcal{G}(x,y;B(z,R)^c) 
     \leq 
     \frac{1}{2\Lambda^2 \cdot D^2}  \sum_{\hat\beta \in W(x,y)} \mathcal{A}(\hat\beta) \omega(\hat\beta)
     \leq 
     \frac{1}{2}  \sum_{\hat\beta \in W(x,y)}  \omega(\hat\beta)
     =
     \frac{1}{2} \mathcal{G}(x,y),
\end{align*}
which is our desired inequality.
\end{proof}

Now we can deduce easily the Ancona inequality along any Morse geodesic.

\begin{lem}\label{Deviation to Ancona}
    Let $x,y,z \in G$ be three  elements satisfying  
    \[  \mathcal{G}(x,y;B(z,R)^c) \leq \frac{1}{2} \mathcal{G}(x,y). \]
    for some $R>0$.
    Then there exists a constant $C=C(R)$ only depending on $R$, such that 
    \[C^{-1} \mathcal{G}(x,z)\mathcal{G}(z,y) \leq \mathcal{G}(x,y) \leq C \mathcal{G}(x,z)\mathcal{G}(z,y). \]
\end{lem}

\begin{proof}
    We only prove the right-hand part, as the left-hand part is proved in \eqref{Greenfunction}. 
    We have
    \begin{align*}
        \mathcal{G}(x,y) &= \mathcal{G}(x,y;B(z,R)^c) + \sum_{z' \in B(z,R)} \mathcal{G}(x,z';B(z,R)^c) \mathcal{G}(z',y)\\
        & \leq \frac{1}{2}\mathcal{G}(x,y) + \sum_{z' \in B(z,R)} \mathcal{G}(x,z')\mathcal{G}(z',y)\\
        & \leq \frac{1}{2}\mathcal{G}(x,y) + \sum_{z' \in B(z,R)} C_1^2 \mathcal{G}(x,z)\mathcal{G}(z,y)\\
        & \leq \frac{1}{2}\mathcal{G}(x,y) + C_2 C_1^2 \mathcal{G}(x,z)\mathcal{G}(z,y),
    \end{align*}
    where $C_1$ only depends on $R$ from \cref{like Harnack} and $C_2$ is the number of elements in $B(z,R)$. The lemma is proved by setting $C = 2C_2C_1^2$.
\end{proof}

As the Morse property of subsets is quasi-isometric invariant, we formulate our theorem in its most general setting.

\begin{thm}\label{Ancona on Morse}
    Consider an irreducible and finitely supported $\mu$-random walk on a non-amenable group $G$.
    Assume that $G$ acts geometrically on a proper geodesic metric space $(X,d)$ with the base point $o \in X$.
    Let $\gamma$ be a $\kappa$-Morse $c$-quasi-geodesic in $X$. Then there exists a constant $C = C(\kappa,c)$ such that if points $xo, zo, yo$ lie on the quasi-geodesic $\gamma$ in this order for group elements $x,y$ and $z$, then
    \[C^{-1} \mathcal{G}(x,z)\mathcal{G}(z,y) \leq \mathcal{G}(x,y) \leq C \mathcal{G}(x,z)\mathcal{G}(z,y). \]
\end{thm}

\begin{proof}
    Since the orbit map $g \mapsto go$ is a quasi-isometry, it suffices to prove the theorem in the case where $\gamma$ is a quasi-geodesic in the Cayley graph. 
    Because $\gamma$ is Morse, the segment $[x, y]$ is also Morse, and the Hausdorff distance between $\gamma_{[x, y]}$ and $[x, y]$ is bounded above by a constant $D = D(\kappa, c)$. Thus, we can find a point $\hat{z} \in [x, y]$ such that $d_S(z, \hat{z}) \leq D$.
    Now the three points $x, \hat{z},y$ lie on a Morse geodesic. Then use \cref{lemma:main} and \cref{Deviation to Ancona}.
\end{proof}

\section{Ancona inequality (II) along  a Morse subset with narrow points}\label{sec Ancona abstraction}
Let $(X,d)$ be a (possibly non-proper) geodesic metric space on which the group $G$ acts metrically proper and co-boundedly.  In this section  we shall prove a more general version of the Ancona inequality along a Morse subset with narrow points. We first formulate a general set of  assumptions which are necessary in its proof and which we believe may be of independent interests.

Let $\mu$ be an irreducible probability with finite support on $G$.
Fix a finite symmetric generating set $S$ of $G$, which contains the support of $\mu$. Fix a base point $o \in X$. By Milnor-Svarc Lemma, the orbit map $\Phi : g  \longmapsto go$ is a quasi-isometry from $\mathrm{Cay}(G,S)$ to $X$.  Let us make the following convention to facilitate the discussion below.  

\begin{conv}[Constant $\varepsilon$]\label{cstvarepsilon}
Choose $\varepsilon>1$ such that $X=N_\varepsilon(Go)$ and  
\[ \varepsilon^{-1} d(go,ho) - \varepsilon \leq d_S(g,h) \leq \varepsilon d(go,ho) + \varepsilon . \] 
We also take $\varepsilon$ large enough such that, for any path $\alpha$ in $X$, there exists a path $\beta$ in  $\mathrm{Cay}(G, S)$ such that:
\begin{enumerate}
    \item $\Phi(\beta)$ and $\alpha$ are contained in the $\varepsilon$-neighborhood of each other;
    \item $ \varepsilon^{-1} \|\alpha\| - \varepsilon \leq \|\beta\| \leq \varepsilon \|\alpha\| + \varepsilon.$
\end{enumerate}
Moreover, for any path $\beta$ in the $\mathrm{Cay}(G, S)$, there is a $\mu$-trajectory $\tau(\beta)$ with a Hausdorff distance from $\beta$ bounded by $\Lambda$ with $\|\beta\| \leq \|\tau(\beta)\| \leq \Lambda\|\beta\|$. (See the paragraph before \cref{lemma:main}.)

By enlarging $\varepsilon$ (depending on $\Lambda$), we may assume $\beta$ to be a $\mu$-trajectory (by substituting $\beta$ by $\tau(\beta)$) still satisfying conditions $(1)$ and $(2)$. For notational simplicity, we denote  $T(\alpha):=\beta$.
\end{conv}
\subsection{Main definitions: narrowness and divergence}
We start with the main definitions in the assumptions.

\begin{defn}[$\mathfrak{s}$-narrowness]\label{defofsnarrow}
    Let $Y\subset X$ be a subset and $\mathfrak s>0$. We say that a point $z\in X$ is \textit{ $\mathfrak s$-narrow} for $Y$ if $Y = Y_1 \cup Y_2$ is the union of two subsets $Y_1,Y_2$ so that for any $y_1 \in Y_1$ and $y_2 \in Y_2$, we have $$d(y_1,y_2)+\mathfrak{s} \geq d(y_1,z) + d(z,y_2).$$ 
    For brevity, we shall say that $Y=Y_1\cup Y_2$ is \textit{$\mathfrak s$-narrow at $z$} (with $Y_1, Y_2$ given as above).
\end{defn}
From the definition, we immediately obtain the following basic properties:
\begin{enumerate}

    \item Any geodesic is $0$-narrow at any point on it.

    \item The intersection $Y_1\cap Y_2$ is contained in the ball $B(z,\mathfrak s/2)$.

    \item If $Y$ is $\mathfrak{s}$-narrow at $z$, then for any $z' \in X$, $Y$ is $\mathfrak{s'}$-narrow at $z'$ with $\mathfrak{s'}=\mathfrak{s} + 2d(z,z')$.
\end{enumerate}
By adding $z$ into $Y_1$ and $Y_2$, we may assume the narrow point $z$ is contained in $Y_1\cap Y_2$. 

The following immediate consequence will be used.
\begin{lem}\label{narrowseparation}
Assume that $Y=Y_1\cup Y_2$ is {$\mathfrak s$-narrow at a point $z$}. Then for any $R>0$ and denoting $\Omega_i=Y_i\setminus B(z,R)$ with $i=1,2$ we have 
$$
d(\Omega_1,\Omega_2)\ge 2R-\mathfrak s.
$$
\end{lem}
\begin{proof}
For any $p\in\Omega_1,q\in\Omega_2$ we have $d(p,z),d(q,z)>R$ and then $d(p,q)\ge d(p,z)+d(z,q)-\mathfrak s\ge 2R-\mathfrak s$.    
\end{proof}

To explain the terminology of narrowness, let us  provide an equivalent but more easily verified condition of narrow points restricted to Morse subsets. Roughly speaking, a point is narrow if and only if a geodesic connecting different parts must pass through its fixed neighborhood.
\begin{lem}\label{pass narrow point}
    Assume that $Y$ is a $\kappa$-Morse subset. Then the following statements are equivalent: 
    \begin{enumerate}
        \item 
        $Y=Y_1\cup Y_2$ is $\mathfrak s$-narrow at a point $z \in Y$ for some $\mathfrak{s}$.
        \item 
        There exist a constant $ R_0>0$ and a decomposition $Y=Y_1\cup Y_2$ of subsets such that for any $x \in Y_1$ and $y \in Y_2$, we have $d([x,y], z) \leq R_0$.
        \item 
        $Y=Y_1\cup Y_2$ can be written  as a union of two Morse subsets $Y_1, Y_2$ with bounded intersection containing $z$.
    \end{enumerate}
\end{lem}

\begin{proof}
    The direction (2)$\Rightarrow$(1) is obvious by setting $\mathfrak{s}= 2R_{0}$. 
    
    The direction (1)$\Rightarrow$(2): we shall show $R_0=\mathfrak s + 3\kappa_1 +1 $ with $\kappa_1 := \kappa(1)$ is the desired constant. 
    We first claim that the subsets $ N_{\kappa_1}(Y_1) \setminus B(z, R_0) $ and $ N_{\kappa_1}(Y_2) \setminus B(z, R_0) $ are disjoint.

    Indeed, let $ p \in N_{\kappa_1}(Y_1) \setminus  B(z, R_0) $ and $ q \in N_{\kappa_1}(Y_2) \setminus  B(z, R_0) $. Let $\bar p$ be a closest point to $p$ on $Y_1$, and $\bar q$ be a closest point to $q$ on $Y_2$, so we have $d(p,\bar p), d(q,\bar q)\le \kappa_1$. Consequently, we have $ \bar p, \bar q \in Y \setminus  B(z, \mathfrak s + 2\kappa_1 + 1)$, which implies $ d(\bar p, \bar q) \geq d(\bar p, z) + d(z, \bar q) -\mathfrak{s} \geq 2\kappa_1 + 1$ by $\mathfrak s$-narrowness (\cref{narrowseparation}), and thus $ d(p,q) \geq 1$. Hence the claim is proved.

     Without loss of generality, we assume that $ d(x, z) $ and $ d(y, z) $ are both greater than $R_{0}$; otherwise there is nothing to do. Since $ Y $ is $\kappa$-Morse, we have $ [x,y] \in N_{\kappa_1}(Y) $ for any $x\in Y_1, y\in Y_2$.  Because $N_{\kappa_1}(Y_1) \setminus  B(z, R_0)$ and $N_{\kappa_1}(Y_2) \setminus  B(z, R_0)$ are   disjoint subsets, any geodesic $ [x,y] $ must intersect $ B(z, R_0)$ and so (2) is proved.

     The direction (1)$\Rightarrow$(3):
     Following the proof above, for any $k>0$, we can prove that $N_k(Y_1) \setminus  B(z, R) $ and $ N_k(Y_2) \setminus  B(z, R)$ are disjoint with $R = \mathfrak s + 3k +1$. This proves the bounded intersection since $N_k(Y_1) \cap N_k(Y_2) \subset B(z, R)$.
     
     Next we show $Y_1$ is Morse (and $Y_2$ is Morse by symmetry). Let $\gamma$ be a $c$-quasi-geodesic with endpoints in $Y_1$. Set $R' = \mathfrak{s} + 3\kappa(c) + 1$. By construction, $\gamma \subset N_{\kappa(c)}(Y)$ and the $\kappa(c)$-neighborhoods of $Y_1,Y_2$ are disjoint outside $B(z,R')$.  
     If $\gamma \cap B(z,R') = \emptyset$, then $\gamma \subset N_{\kappa(c)}(Y_1)$.  Otherwise, let $u$ be the first entering point and $v$ be the last existing point of $B(z,R')$ along $\gamma$. Since $B(z,R')$ is bounded, the sub-quasi-geodesic $[u,v]_{\gamma}$ has length less than $2cR'+c$, while the remainder stays in $N_{\kappa(c)}(Y_1)$. Thus $\gamma \subset N_{M(c)}(Y_1)$ for $M(c) = \kappa(c) + 2cR' + c$, proving $Y_1$ is Morse.

     The direction (3)$\Rightarrow$(2): the geodesic $[x,y]$ is contained in $N_{\kappa_1}(Y) = N_{\kappa_1}(Y_1) \cup N_{\kappa_1}(Y_2)$ with two endpoints $x \in N_{\kappa_1}(Y_1)$ and $y \in N_{\kappa_1}(Y_2)$, then $[x,y]$ must intersect $N_{\kappa_1}(Y_1) \cap N_{\kappa_1}(Y_2)$, which is a bounded neighborhood of $z$.
\end{proof}

Denote $\mathbf d_Y(x,y)=\mathrm{diam}\{\pi_Y(x)\cup\pi_Y(y)\}$ for two points $x,y\in X$.

\begin{defn}[$\mathfrak{f}$-divergence around narrow points]\label{defofdivergence}
     Let $Y = Y_1 \cup Y_2$ be a subset with a narrow point $z$.
     Let $\mathfrak{f}$ be a function $\mathbb{R}_{+} \to \mathbb{R}_{+}$. We say that $Y$ has \textit{$\mathfrak{f}$-divergence} around $z$ if for every $K>0$ and every path $\gamma$  with endpoints $\gamma_- $ and $ \gamma_+$ satisfying:
     \begin{enumerate}
        \item $\pi_Y(\gamma_-) \subset Y_1$,
        \item $\pi_Y(\gamma_+) \subset Y_2$,
        \item $d(\gamma,Y) \geq \mathfrak{f}(K)$,
    \end{enumerate}
    the following inequality holds:
     \[ K\cdot \mathbf{d}_Y(\gamma_-,\gamma_+) \leq \|\gamma\| + d(\gamma,Y). \]
\end{defn}
\begin{rem}
If $Y$ is a Morse subset, then $Y$ has $\mathfrak{f}$-divergence around every narrow point by \cref{path far away from Morse subset}.
We also prove in Lemma \ref{path far away from proportionally contracting and cross good point} that if a geodesic is proportionally contracting then it has  $\mathfrak{f}$-divergence around good points.  
\end{rem}

\subsection{Formulation of Ancona inequality}
We need another two auxiliary definitions.
\begin{defn}\label{linearnbhdDefn}
    Given $k\ge 0$, the \textit{$k$-linear neighborhood} of a subset $Y$ relative to a reference point $z$ is defined  by
    \[ \mathcal{N}_k(Y,z) := \left\{ x\in X :d(x,Y) \leq k\cdot \mathbf d_Y(x,z) \right\}. \]
\end{defn}

\begin{defn}[$k$-antipodal points]
Let $Y=Y_1\cup Y_2$ be $\mathfrak s$-narrow at a point $z\in Y$. Given $k>0$, a pair of points $x,y \in X$ with $\pi_Y(x) \subset Y_1 $ and $ \pi_Y(y) \subset Y_2$ is called \textit{$k$-antipodal along $(Y,z)$} if  $x \subset \mathcal{N}_k(Y_1,z)$ and $ y \subset \mathcal{N}_k(Y_2,z)$; that is  
    \[d(x,Y) \leq k \cdot \mathbf d_{Y_1}(x,z) \quad \text{and}\quad  d(y,Y) \leq k \cdot \mathbf d_{Y_2}(y,z).\] 
We refer to Figure \ref{k-separate} for illustration. 

\end{defn}

\begin{figure}[ht]
    \centering
    \def\svgwidth{0.8\columnwidth}
    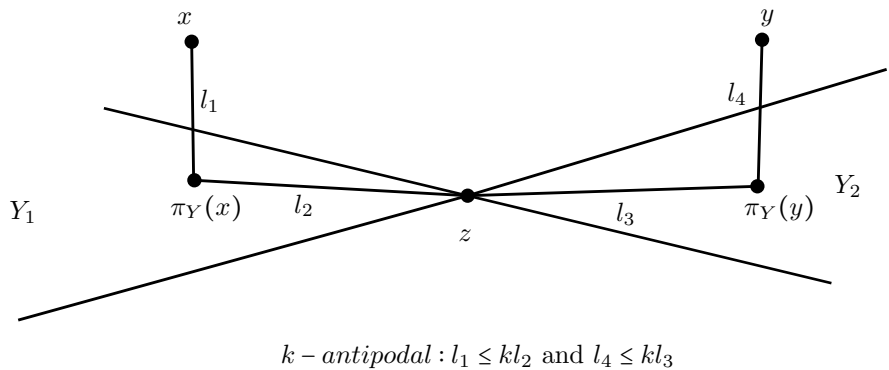

    \caption{$x$ and $y$ are $k$-antipodal along $Y$ by the narrow point $z$.}
    \label{k-separate}
\end{figure}

We say that a subset $Y$ is \textit{$\iota$-quasi-geodesically-connected} at a point $z$ for $\iota > 0$ if there exists a point $z' 
\in Y$ with $d(z,z') \leq \iota$ such that for every $y \in Y$, there exists a geodesic $[y, z']$ contained in the $\iota$-neighborhood $N_{\iota}(Y)$.
This condition is introduced for technical purposes. In this paper, we will utilize the fact that both geodesics and Morse subsets satisfy this property.

We are ready to formulate the necessary conditions to get Ancona inequality.
\begin{prop}\label{Ancona general}
    Let $Y = Y_1 \cup Y_2$ be a subset in $X$  and  $zo \in Y$ be an orbit point for some $z\in G$ with the following properties:
    \begin{enumerate}
        \item \textbf{Narrowness: }$Y = Y_1 \cup Y_2$ is $\mathfrak{s}$-narrow at $zo$.
        \item \textbf{Divergence: }$Y = Y_1 \cup Y_2$ has $\mathfrak{f}$-divergence around $zo$.
        \item \textbf{Quasi-geodesically-connectedness: }
        $Y_1,Y_2$ are $\iota$-quasi-geodesically-connected subsets at $zo$.
    \end{enumerate}
    Then for any constant $k>0$, there exists a constant $C = C(\mathfrak{s},\mathfrak{f},\iota,k,\varepsilon)$ with the following property.
    Let $x,y \in G$  such that $xo,yo$ are $k$-antipodal along $(Y,zo)$. 
    Then 
    \[
    C^{-1} \mathcal{G}(x,z)\mathcal{G}(z,y) \leq \mathcal{G}(x,y) \leq C \mathcal{G}(x,z)\mathcal{G}( z,y)
    .\]
\end{prop}

\begin{rem}
    The proposition also holds for $r$-Green function  $\mathcal{G}_r$ with any $r<R_\mu$, yielding the corresponding inequality. One simply needs to follow exactly the same proof while using the $r$-dependent version of \cref{weight bound}. For simplicity, we only present the case $r=1$ here.
\end{rem}

By \cref{Deviation to Ancona}, \cref{Ancona general} could be derived  from the following deviation result. 
\begin{lem}
     Let $Y \subset X$ and $x,y,z \in G$ be as in the setup of \cref{Ancona general}. Then there exists a constant $R>0$ such that  
    \[
    \mathcal{G}(x,y;B(z,R)^c) \leq \frac{1}{2} \mathcal{G}(x,y)
    .\]
    Here $B(z,R)$ is a ball in the Cayley graph of $G$.
\end{lem}

\begin{proof}
We follow the proof scheme of \cref{lemma:main}. However, the difficulty comes with the fact that the $k$-antipodal points $xo,yo$ may not be contained in $Y$. First of all, let us fix the constants.
\begin{enumerate}[label=\textbf{(C\arabic*)}] 
\item  \label{cstK} 
Let $\phi \in (0,1)$ and $L>1$ be as  in Lemma \ref{weight bound} satisfying the inequalities \eqref{lwb} and \eqref{upb}. Fix a positive constant $K$ such that
\[
\phi^{\varepsilon^{-1}(K - k - 1)-\varepsilon} \cdot L^{4\varepsilon(k+2)} \leq \frac{1}{2}.
\]
\item \label{cstD} Set the constant $D = \mathrm{max}\{ \mathfrak{f}(K) + \varepsilon, \iota + 2\varepsilon \}$.

\item \label{cstR} 
 Finally, set $R > (k+1)(2\mathfrak{s}+2D+10) + 7D + \mathrm{log}_2({16\varepsilon^2 \cdot D^2})$. 
\end{enumerate}

Let us now begin the proof.

For any subset $\mathcal{D} \subset X$, let $W'(x,y;\mathcal{D})$ denote the set of trajectories $\beta \in W(x,y)$ whose orbit images $\Phi(\beta)$ have interior points in $\mathcal{D}$.
Since the orbit map is an $\varepsilon$-quasi-isometry, we have $W(x,y;B(z,R')^c) \subset W'(x,y; B(zo,R)^c)$ for $R' = \varepsilon R + \varepsilon^2$. 
Let $z' \in Y$ with $d(zo,z') \leq \iota$ given by the condition $(3)$. 
Hence we only need to prove $$\omega(W'(x,y; B(z',R)^c)) \leq \frac{1}{2} \mathcal{G}(x,y).$$

Without loss of generality, we may assume $z' \in Y$ and $d(xo,z'), d(yo,z') \geq R$. Indeed, if (say) $d(xo,z') < R$, then $\omega(W'(x,y; B(z',R)^c)) = 0$.

Define the subsets 
\begin{align}\label{OmegaSetsDefn}
\Omega_1 := N_D(Y_1) \setminus B( z', R),\quad \Omega_2 := N_D(Y_2) \setminus B(z',R),    
\end{align} which by \cref{narrowseparation} ensures $d(\Omega_1, \Omega_2) \geq 2R-\mathfrak{s} - 2D > 5D$, since $R > \mathfrak{s} + 7D$.

For every point $a \in X$, let $\bar{a} \in \pi_{Y}(a)$ be a closest point to $a$ on $Y$.
Given any two group elements $g,h \in G$, we construct a trajectory $\mathrm{Tr}(g,h) \in W(g,h)$ by concatenating four trajectories:
\[
\mathrm{Tr}(g,h) = T([go, \overline{go}]) \circ T([\overline{go}, z']) \circ T([z', \overline{ho}]) \circ T([\overline{ho}, ho]),
\]
where:
\begin{itemize}
    \item The geodesics $[\overline{go}, z']$ and $[z', \overline{ho}]$ are contained in $N_{\iota}(Y)$ as guaranteed by condition $(3)$;
    \item $T(\alpha)$ denotes a trajectory whose image $\Phi(T(\alpha))$ stays within $\varepsilon$-Hausdorff distance of $\alpha$ (as guaranteed by \cref{cstvarepsilon});
    \item Without loss of generality, we assume the projection $\overline{a}$ realizes the projection distance $\mathbf{d}_{Y}(a,z')$ for $a \in X$;
    \item The endpoints are adjusted within $\varepsilon$-distance to ensure proper concatenation.
\end{itemize}
Note that $\Phi(\mathrm{Tr}(g,h))$ passes through the ball $B(z',\varepsilon) \subset B(z', R)$.

We will subdivide the trajectories $W'(x,y;B(z',R)^c)$ into four subsets and estimate them separately:
\begin{align*}
    \Gamma_1 : = \{ \beta \in W'(x,y;B(z',R)^c): \Phi(\beta) \cap \Omega_1 = \emptyset ,\Phi(\beta) \cap \Omega_2 = \emptyset  \},\\
    \Gamma_2 : = \{ \beta \in W'(x,y;B(z',R)^c): \Phi(\beta) \cap \Omega_1 \neq \emptyset ,\Phi(\beta) \cap \Omega_2 \neq \emptyset  \},\\
    \Gamma_3 : = \{ \beta \in W'(x,y;B(z',R)^c): \Phi(\beta) \cap \Omega_1 = \emptyset ,\Phi(\beta) \cap \Omega_2 \neq \emptyset  \},\\
    \Gamma_4 : = \{ \beta \in W'(x,y;B(z',R)^c): \Phi(\beta) \cap \Omega_1 \neq \emptyset ,\Phi(\beta) \cap \Omega_2 = \emptyset  \}.
\end{align*} 
In the setup of \cref{lemma:main}, the sets $\Gamma_1, \Gamma_3,\Gamma_4$ are vacuous.
To conclude the proof in the current setting, it suffices to prove $$\omega(\Gamma_i) \leq \frac{1}{8} \mathcal{G}(x,y)$$ for each $i = 1,2,3,4$.

\textbf{Case (1).} For every $\beta \in \Gamma_1$, we have $d(\Phi(\beta), Y) \geq D$ by the definition of $\Gamma_1$. 
Since $\Phi(\beta)$ is a path with $k$-antipodal endpoints $xo,yo$
and $D \geq \mathfrak{f}(K)$ (by \ref{cstD}), by the $\mathfrak{f}$-divergence of $Y$ we have
\begin{align}\label{Case1 divergence}
K \cdot \mathbf{d}_{Y}(xo,yo) \leq d( \Phi(\beta), Y) + \| \Phi(\beta)\| \leq d(xo,Y)+ \|\Phi(\beta)\|.
\end{align}
Since $xo$ and $yo$ are $k$-antipodal, by triangle inequality, we obtain 
\begin{align*}
    \mathbf{d}_{Y}(z',yo) \geq d(z',yo) - d(yo,Y) \geq d(z',yo)-k\cdot \mathbf{d}_{Y}(z',yo),
\end{align*} 
which implies
\begin{align}\label{Case1diamzyo}
    \mathbf{d}_{Y}(z',yo) \geq (k+1)^{-1}d(z',yo) \geq (k+1)^{-1}R.
\end{align}
Using $k$-antipodality of $xo$ and $yo$ again, we obtain
\begin{equation}\label{Case1dxoY}
    \begin{split}
        d(xo,Y) &\leq k\cdot\mathbf{d}_{Y}(xo,z') \\
&\leq k \cdot \big( \mathbf{d}_{Y}(xo,z') + \mathbf{d}_{Y}(z',yo) -(k+1)^{-1}R\big) \\
&\leq k \cdot \big(\mathbf{d}_{Y}(xo,yo) + \mathfrak s - (k+1)^{-1}R \big) \\
&\leq k \cdot \mathbf{d}_{Y}(xo,yo)
    \end{split}
\end{equation}
where the third inequality follows from the $\mathfrak{s}$-narrowness and the last inequality follows from \ref{cstR}.
Substituting \cref{Case1dxoY} into \cref{Case1 divergence}, we obtain that for every $\beta\in \Gamma_1$:
\begin{align*}
    \|\Phi(\beta)\| &\geq K \, \mathbf{d}_{Y}(xo,yo) - d(xo,Y) \\
    &\geq (K - k)\mathbf{d}_{Y}(xo,yo).
\end{align*}
Since the orbit map is an $\varepsilon$-quasi-isometry, we have $\|\beta\| \geq \varepsilon^{-1}(K - k)\mathbf{d}_{Y}(xo,yo) - \varepsilon$.  
This yields, by \cref{weight bound}(\ref{upb}),
\begin{align}\label{Case1smallweight}
   \omega(\Gamma_1) \leq \phi^{\varepsilon^{-1}(K - k)\mathbf{d}_{Y}(xo,yo) - \varepsilon } \cdot L^{d_S(x,y)}. 
\end{align}

Next, we estimate the length of $\mathrm{Tr}(x,y)$ which by construction is a concatenation of 4 trajectories:
\[
\|\mathrm{Tr}(x,y)\| = \|T([xo, \overline{xo}])\| + \|T([\overline{xo}, z'])\| + \| T([z', \overline{yo}])\| + \| T([\overline{yo}, yo])\|.
\]
Recall that $T$ converts any path $\alpha$ to a trajectory $T(\alpha)$ such that $\|T(\alpha)\| \leq \varepsilon \|\alpha\|+ \varepsilon$. 
Hence, by the $k$-antipodality of $xo$ and $yo$, we obtain
\begin{align*}
    \|\mathrm{Tr}(x,y)\|  &\leq \varepsilon \big(d(xo, \overline{xo}) + d(\overline{xo}, z') + d(z', \overline{yo}) + d(\overline{yo}, yo)\big) + 4\varepsilon \\
    &\leq \varepsilon \big(d(xo,Y)  + \mathbf{d}_{Y}(xo,z') + \mathbf{d}_{Y}(z',yo)+ d(yo,Y)\big) + 4\varepsilon \\
    &\leq \varepsilon \big(k\cdot\mathbf{d}_{Y}(xo,z')  +\mathbf{d}_{Y}(xo,z') + \mathbf{d}_{Y}(z',yo)+ k \cdot \mathbf{d}_{Y}(z',yo)\big) + 4\varepsilon \\
    &\leq \varepsilon(k+1)\big(\mathbf{d}_{Y}(xo,yo) + \mathfrak{s}\big) + 4\varepsilon\\
    &\leq 2\varepsilon(k+1) \cdot \mathbf{d}_{Y}(xo,yo),
\end{align*}
where the last inequality uses the fact that 
\begin{align}\label{Case1bigprojection}
    \mathbf{d}_{Y}(xo,yo) \geq \mathbf{d}_{Y}(xo,z') + \mathbf{d}_{Y}(z',yo) - \mathfrak{s} \geq (k+1)^{-1}R - \mathfrak{s}
\end{align}
large enough by the $\mathfrak{s}$-narrowness, \cref{Case1diamzyo} and \ref{cstR}.
According to \cref{weight bound}(\ref{lwb}), we have $$\omega(\mathrm{Tr}(x,y))\geq L^{-\|\mathrm{Tr}(x,y)\|} \geq L^{-2\varepsilon(k+1) \cdot \mathbf{d}_{Y}(xo,yo)}.$$

Since $ d_S(x,y) \leq \|\mathrm{Tr}(x,y)\| $, combining with \cref{Case1smallweight}, we obtain
\begin{align*}
    \omega(\Gamma_1) / \omega(\mathrm{Tr}(x,y)) &\leq \phi^{\varepsilon^{-1}(K - k)\mathbf{d}_{Y}(xo,yo) - \varepsilon} \cdot L^{d_S(x,y) + \|\mathrm{Tr}(x,y)\|} \\
    &\leq \phi^{\varepsilon^{-1}(K - k)\mathbf{d}_{Y}(xo,yo) - \varepsilon} \cdot L^{2 \|\mathrm{Tr}(x,y)\|}\\
    &\leq \left( \phi^{\varepsilon^{-1}(K - k) - \varepsilon} \cdot L^{4\varepsilon(k+1)} \right)^{\mathbf{d}_{Y}(xo,yo) } \\
    (\ref{cstK},\cref{Case1bigprojection} \,\text{and}\, \ref{cstR})\quad &\leq \left(\frac{1}{2}\right)^{\mathbf{d}_{Y}(xo,yo)} \leq \frac{1}{8}.
\end{align*}
Since $\mathrm{Tr}(x,y)$ is a trajectory from $x$ to $y$, we have proved that $$\omega(\Gamma_1) \leq  \frac{1}{8}\omega(\mathrm{Tr}(x,y))\le  \frac{1}{8}\mathcal{G}(x,y).$$

\textbf{Case (2).} Let $\Gamma_2'$ denote the set of trajectories $\beta$ in $\Gamma_2$ whose image $\Phi(\beta)$ first intersects $\Omega_1$ and then $\Omega_2$. The definitions of $\Omega_1,\Omega_2$ are given in \cref{OmegaSetsDefn}. Thus, for every $\beta \in \Gamma_2'$, we can find elements $g,h \in \beta$ such that $\Phi([g,h]_{\beta})$ is a sub-path of $\Phi(\beta)$ only intersects $\Omega_1$ and $\Omega_2$ at endpoints:
\[\Phi(g) \in \Omega_1, \Phi(h) \in \Omega_2 \quad \text{and} \quad [g,h]_{\beta}  \in W'(g,h;B(z',R)^c\cap N_D(Y)^c). \]
We have that $go$ and $ho$ are contained in the \textit{boundaries} of $\Omega_1$ and $\Omega_2$ defined as:  
\[ \partial \Omega_1 = \{go \in \Omega_1: d(go,Y) \geq D-\varepsilon\} \bigand \partial \Omega_2 = \{go \in \Omega_2: d(go,Y) \geq D-\varepsilon\}. \]
We first observe that $\pi_{Y}(go) \subset Y_1$. Indeed, we have $d(go,Y_1) \leq D$ by $go\in \Omega_1$, and $$d(go,Y\setminus Y_1) \geq d(go,Y_2) \geq d(\Omega_1,\Omega_2)\geq 2 D > d(go,Y_1).$$ This implies $\pi_{Y}(go) \subset Y_1$. The same holds for $\pi_{Y}(ho) \subset Y_2$ by symmetry. 

Let $\beta' \in W'(g,h;B(z',R)^c\cap N_D(Y)^c)$ be a trajectory from $g$ to $h$ with $\Phi$-image $\Phi(\beta')$ outside $B(z',R)\cup N_D(Y)$. Recalling that $D\geq \mathfrak{f}(K) + \varepsilon$ by \ref{cstD},   the $\mathfrak{f}$-divergence implies
\[K \cdot \mathbf{d}_{Y}(go,ho) \leq d(\Phi(\beta'), Y) + \|\Phi(\beta')\| \leq D + \|\Phi(\beta')\|.\] 
By the triangle inequality, the diameter 
\begin{align}\label{Case2bigprojection}
    \mathbf{d}_{Y}(go,ho) \geq d(go,ho)-2D \geq d(\Omega_1,\Omega_2)-2D\geq R-6D
\end{align}
is greater than $D$, and thus we obtain
\begin{align*}
    \|\Phi(\beta')\| &\geq K \, \mathbf{d}_{Y}(go,ho) - D \geq (K - 1)\mathbf{d}_{Y}(go,ho).
\end{align*}
Recalling the orbit map $\Phi$ is an $\varepsilon$-quasi-isometry, we have $\|\beta'\| \geq \varepsilon^{-1}(K - 1)\mathbf{d}_{Y}(go,ho) - \varepsilon$.
This yields, by \cref{weight bound}(\ref{upb}),
\begin{align}\label{Case2smallweight}
    \omega(W'\left(g,h;B\left(z',R\right)^c \cap N_D\left(Y\right)^c\right)) \leq \phi^{\varepsilon^{-1}(K - 1)\mathbf{d}_{Y}(go,ho) -\varepsilon} \cdot L^{d_S(g,h)}.
\end{align}
For the length of $\mathrm{Tr}(g,h)$, following the same estimate as in \textbf{Case (1)}, we obtain
\begin{align*}    
    \|\mathrm{Tr}(g,h)\|   &\leq \varepsilon \big(d(go,Y) + \mathbf{d}_{Y}(go,z') + \mathbf{d}_{Y}(z',ho)+ d(ho,Y)\big) + 4\varepsilon \\
    &\leq \varepsilon \big(2D + \mathbf{d}_{Y}(go,ho) +\mathfrak{s} \big) + 4\varepsilon \\
    &\leq 2\varepsilon \cdot \mathbf{d}_{Y}(go,ho),
\end{align*}
where the last inequality holds because $\mathbf{d}_{Y}(go,ho)$
is large enough by \cref{Case2bigprojection} and \ref{cstR}.
According to \cref{weight bound}(\ref{lwb}), we have $$\omega(\mathrm{Tr}(g,h))\geq L^{-\|\mathrm{Tr}(g,h)\|} \geq L^{-2\varepsilon \cdot \mathbf{d}_{Y}(go,ho)}.$$
Hence, combining with \cref{Case2smallweight}, we obtain
\begin{align*}
    \omega(W'\left(g,h;B\left(z',R\right)^c \cap N_D\left(Y\right)^c\right)) / \omega(\mathrm{Tr}(g,h)) &\leq \phi^{\varepsilon^{-1}(K - 1)\mathbf{d}_{Y}(go,ho) -\varepsilon} \cdot L^{d_S(g,h) + \|\mathrm{Tr}(x,y)\|} \\
    &\leq \left( \phi^{\varepsilon^{-1}(K - 1)-\varepsilon} \cdot L^{4\varepsilon} \right)^{\mathbf{d}_{Y}(go,ho) } \\
    (\ref{cstK} ,\cref{Case2bigprojection} \,\text{and}\, \ref{cstR} ) \quad &\leq \left(\frac{1}{2}\right)^{\mathbf{d}_{Y}(go,ho)} \leq \frac{1}{16\varepsilon^2 \cdot D^2}.
\end{align*}

Like \cref{GreenfunctionUBEQ} and \cref{GreenfunctionUBEQ2} in \cref{lemma:main}, we  decompose $\beta\in \Gamma_2'$ at points $g$ and $h$ into three sub-trajectories, and  the weight of $\beta$ is estimated as follows
\begin{align*}
      & \omega\left(\Gamma_2'\right) = \sum_{\beta \in \Gamma_2'} \omega\left(\beta\right) \\
    \leq &\sum_{\left(g,h\right)}  \omega\left(W'\left(x,g;B\left(z',R\right)^c\right)\right)\cdot \omega\left(W'\left(g,h;B\left(z',R\right)^c \cap N_D\left(Y\right)^c\right)\right)\cdot \omega\left(W'\left(h,y;B\left(z',R\right)^c\right)\right) \\
    \leq &\frac{1}{16\varepsilon^2 \cdot D^2}  \sum_{\left(g,h\right)}  \omega\left(W'\left(x,g;B\left(z',R\right)^c\right)\right) \cdot \omega\left(\mathrm{Tr}\left(g,h\right)\right)\cdot \omega\left(W'\left(h,y;B\left(z',R\right)^c\right)\right)\\
    \leq &\frac{1}{16\varepsilon^2 \cdot D^2}  \sum_{\hat\beta \in W\left(x,y\right)} \mathcal{A}\left(\hat\beta\right) \omega\left(\hat\beta\right)
\end{align*} 
where the sum is taken over all pairs $(g,h) \in \partial \Omega_1 \times \partial \Omega_2$ and the term $\mathcal{A}(\hat\beta)$ accounts for the overcounting when concatenating the three segments into a single path in the last inequality.

Notice that the sub-trajectory $T([\overline{go}, z']) \circ T([z', \overline{ho}])$ of $\mathrm{Tr}(g,h) $ is contained in the $(\varepsilon + \iota)$ neighborhood of $Y$, which has no intersection with $\partial \Omega_1$ and $\partial \Omega_2$.
An argument similar to \cref{finite many g} shows the bound $\mathcal{A}(\hat\beta) \leq \varepsilon^2 \cdot D^2$, and thus
\[ \omega(\Gamma_2') \leq\frac{1}{16\varepsilon^2 \cdot D^2}  \sum_{\hat\beta \in W(x,y)} \mathcal{A}(\hat\beta) \omega(\hat\beta) \leq \frac{1}{16} \sum_{\hat\beta \in W(x,y)} \omega(\hat\beta) \leq \frac{1}{16}\mathcal{G}(x,y).\]

If $\beta \in \Gamma$ intersects $\Omega_2$ before $\Omega_1$, i.e. $\beta \notin \Gamma_2'$, we have the same inequalities as above by symmetry.
Combining both cases gives the estimate:
$\omega(\Gamma_2) \leq \frac{1}{8}\mathcal{G}(x,y).$
  
\begin{figure}[ht]
    \centering
    \def\svgwidth{0.8\columnwidth}
    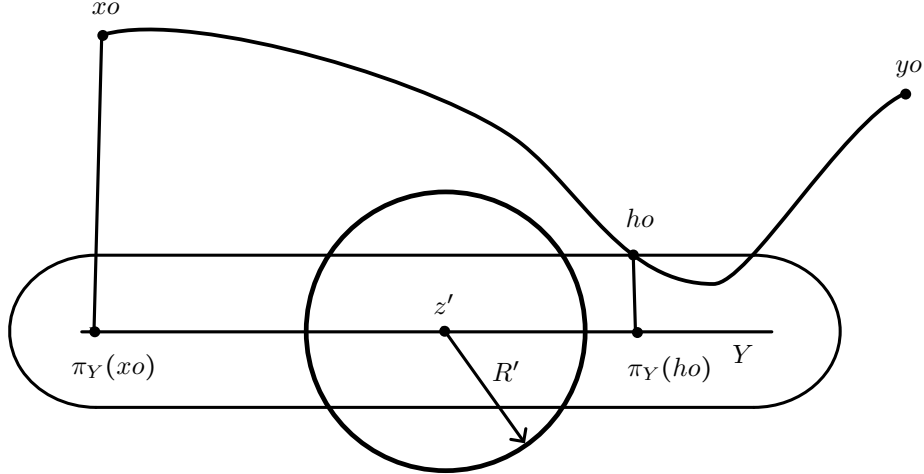

    \caption{For \textbf{Case (3)}.}
    \label{Morseset}
\end{figure}

\textbf{Case (3).}  For any $\beta \in \Gamma_3$, we can find the first point in $\Omega_2$ along $\Phi(\beta)$, denoting it as $ho \in \Phi(\beta)$ such that 
\[ho \in \Omega_2 \quad \text{and} \quad [x,h]_{\beta}  \in W'(x,h;B(z',R)^c\cap N_D(Y)^c). \]
For any $\beta' \in  W'(x,h;B(z',R)^c\cap N_D(Y)^c)$, we have
\begin{align}\label{Case3projection}
    K \cdot \mathbf{d}_{Y}(xo,ho) \leq d( \beta', Y) + \|\Phi(\beta')\| \leq D + \|\Phi(\beta')\|
\end{align}
according to $\mathfrak{f}$-divergence and $D \geq \mathfrak{f}(K) + \varepsilon$ in \ref{cstD}.
By the $\mathfrak{s}$-narrowness and \ref{cstR}, the diameter
\begin{equation}\label{Case3bigprojection}
\begin{split}
     \mathbf{d}_{Y}(xo,ho) &\geq \mathbf{d}_{Y}(xo,z') + \mathbf{d}_{Y}(z',ho)- \mathfrak{s} \\
    &\geq \mathbf{d}_{Y}(z',ho)- \mathfrak{s} \\
    &\geq R-D-\mathfrak{s}
\end{split}
\end{equation}
is large enough, and thus we obtain the lower bound from \cref{Case3projection}
\begin{align*}
    \|\Phi(\beta')\| &\geq  K \, \mathbf{d}_{Y}(xo,ho) - D \geq (K - 1)\mathbf{d}_{Y}(xo,ho).
\end{align*}
Then $\|\beta'\| \geq \varepsilon^{-1}(K - 1)\mathbf{d}_{Y}(xo,ho) - \varepsilon$.
By \cref{weight bound}(\ref{upb}), we get
\begin{align}\label{Case3smallweight}
    \omega(W'(x,h;B(z',R)^c\cap N_D(Y)^c)) \leq \phi^{\varepsilon^{-1}(K - 1)\mathbf{d}_{Y}(xo,ho) -\varepsilon} \cdot L^{d_S(x,h)}.
\end{align}

Because $\mathbf{d}_{Y}(xo,ho)$ is large by \cref{Case3bigprojection} and $Y$ is $\mathfrak{s}$-narrow at $z'$, the length of $\mathrm{Tr}(x,h)$ has the upper bound:
\begin{align*}
    \|\mathrm{Tr}(x,h)\|  &\leq \varepsilon\big(d(xo,Y) + d(ho,Y) +\mathbf{d}_{Y}(xo,z')+ \mathbf{d}_{Y}(z',ho) \big)+4\varepsilon \\
    & \leq \varepsilon\big(k \cdot \mathbf{d}_{Y}(xo,z')+D +\mathbf{d}_{Y}(xo,z')+ \mathbf{d}_{Y}(z',ho) \big)+4\varepsilon \\
    &\leq\varepsilon\big(D + (k + 1)\mathbf{d}_{Y}(xo,ho) + (k+1)\mathfrak{s}\big)+4\varepsilon\\
    &\leq2\varepsilon(k + 1)\mathbf{d}_{Y}(xo,ho).
\end{align*}
Using \cref{weight bound} again, combining \cref{Case3smallweight},we have
\begin{align*}
    \omega(W'(x,h;B(z',R)^c\cap N_D(Y)^c)) / \omega(\mathrm{Tr}(x,h)) &\leq \phi^{\varepsilon^{-1}(K - 1)\mathbf{d}_{Y}(xo,ho) -\varepsilon } \cdot L^{d_S(x,h) + \|\mathrm{Tr}(x,h)\|} \\
    &\leq \left( \phi^{\varepsilon^{-1}(K - 1)-\varepsilon} \cdot L^{4\varepsilon(k+1)} \right)^{\mathbf{d}_{Y}(xo,ho) } \\
   (\ref{cstK},\cref{Case3bigprojection} \, \text{and}\,  \ref{cstR}) \quad &\leq \left(\frac{1}{2}\right)^{\mathbf{d}_{Y}(xo,ho)} \leq \frac{1}{16\varepsilon^2 \cdot D^2}.
\end{align*}
Following \textbf{Case (2)}, we now divide $\beta \in \Gamma_3$ at points $h$ into two segments to obtain:
\begin{align*}
     \omega(\Gamma_3) 
    \leq &\sum_{h\in \partial\Omega_2}   \omega(W'(x,h;B(z',R)^c\cap N_D(Y)^c))\cdot \omega(W'(h,y;B(z',R)^c)) \\
    \leq &\frac{1}{16\varepsilon^2 \cdot D^2}  \sum_{h\in \partial\Omega_2}  \omega(\mathrm{Tr}(x,h))\cdot \omega(W(h,y;B(z,R)^c))\\
    \leq &\frac{1}{16\varepsilon^2 \cdot D^2}  \sum_{\hat\beta \in W(x,y)} \mathcal{A}(\hat\beta) \omega(\hat\beta)
\end{align*} 
where $\mathcal{A}(\hat\beta)$ accounts for the overcounting when concatenating the path segments in the final estimate.
An argument similar to \cref{finite many g} shows the bound $\mathcal{A}(\hat\beta) \leq \varepsilon^2 \cdot D^2$, and thus
\[ \omega(\Gamma_3) \leq \frac{1}{16} \sum_{\hat\beta \in W(x,y)} \omega(\hat\beta) \leq \frac{1}{8}\mathcal{G}(x,y).\]

\textbf{Case (4).}  It is the same as \textbf{Case (3)} by  symmetry.

We have established the upper bound $\omega(\Gamma_i) \leq \frac{1}{8} \mathcal{G}(x,y)$ for each $i = 1,2,3,4$, which completes the proof of the lemma.
\end{proof}

\subsection{Ancona inequalities along a Morse subset with narrow points}\label{Sec Ancona on Morse subset}

We  prove now the Ancona inequality along a Morse subset with narrow points by verifying the assumptions in \cref{Ancona general}. 

\begin{thm}\label{Ancona on Morse subset}
    For any constants $k , \mathfrak{s} \geq 0$ and a Morse gauge $\kappa$, there exists a constant $C =C(\kappa,\mathfrak{s},k)$ with the following property.
    Let $Y=Y_1\cup Y_2 $ of $X$ be a $\kappa$-Morse subset with an $\mathfrak{s}$-narrow point $zo \in X$ for some $z\in G$. If $xo,yo \in X$ are $k$-antipodal along $(Y,zo)$ for some elements $x,y \in G$, then 
    \[
    C^{-1} \mathcal{G}(x,z)\mathcal{G}(z,y) \leq \mathcal{G}(x,y) \leq C \mathcal{G}(x,z)\mathcal{G}(z,y)
    .\]
\end{thm}
\begin{proof}
Indeed, the  two assumptions, narrowness and quasi-geodesically-connection, follow automatically from the fact that $Y$ is Morse. The $\mathfrak{f}$-divergence assumption \cref{Ancona general} for some $\mathfrak{f}(K) = D(\kappa,K)$ follows from  \cref{path far away from Morse subset}. The proof is complete by applying \cref{Ancona general}.    
\end{proof}

As Morse geodesics are narrow at every point, this forms a generalization of    \cref{Ancona on Morse} in a much greater context. Applications will be given in \cref{Sec Admissible Sequence} and \cref{sec Ancona in CAT(0)}.

\section{Construction of Morse subsets with narrow points}\label{Sec Admissible Sequence}
This section offers a method to construct Morse subsets with narrow points. As an application in relatively hyperbolic groups, we prove that the saturation along any geodesic forms a Morse subset which is narrow at all transition points.
\subsection{Admissible sequence}
We start with a notion of admissible sequence of subsets, which is an adaption of admissible paths in \cite{yang2014growthtightness} with purpose to construct Morse subsets with narrow points.

\begin{defn}\label{def admissible sequence}
    Let $(P_1,Q_1,P_2,\dots,Q_n,P_{n+1})$ be a sequence of subsets in a geodesic metric space $X$. We say this sequence is \textit{$(D,B)$-admissible} for $D,B>0$ if it satisfies the following conditions:
    
    \begin{enumerate}
        \item  For each $1 \leq i \leq n+1$, $P_i$ is $D$-contracting.
        
        \item For each $1 \leq i \leq n$, the intersections $P_i \cap Q_i$ and $Q_i \cap P_{i+1}$ are non-empty.
        
        \item For each $2 \leq i \leq n$, $ d(Q_{i-1} \cap P_i, P_i \cap Q_i) \geq 2B + 10D. $
        
        \item For each $1 \leq i \leq n+1$,  we have
        \[ \diam \{\pi_{P_i}(P_{i-1} \cup Q_{i-1})\} \leq B \quad \text{and} \quad \diam \{\pi_{P_i}(Q_i \cup P_{i+1})\} \leq B, \]
        provided that  $(P_{i-1}, Q_{i-1})$ or $(Q_i, P_{i+1})$ exists. 
        In particular, for each $1 \leq i \leq n$, $$\diam\{P_i \cap Q_i\} ,\; \diam\{Q_i \cap P_{i+1}\}\le B.$$
    \end{enumerate}
    
    We allow $P_1$ and $P_{n+1}$ to be singletons. Every points in $P_i \cap Q_i$ and $Q_i \cap P_{i+1}$ with $1\le i\le n$ are called \textit{transversal points}.
\end{defn}

\begin{rem}
\begin{enumerate}
\item 
It is clear that the above  set of conditions involves at most 5 consecutive $P_i$ and $Q_i$, so  naturally extends to  an infinite admissible sequence: $(\dots,P_1,Q_1,\dots,P_n,Q_n, \dots)$. 
The results established in this subsection remain valid for infinite admissible sequences.
 
\item 
If the subsets $P_i$ and $Q_i$ are geodesic segments which intersect only at their endpoints, the above notion coincides with that of admissible paths introduced in \cite{yang2014growthtightness}. 

\item 
The term of transversal points is made distinguished with the notion of transitional points in relatively hyperbolic groups (though they are very closely related).
\end{enumerate}
\end{rem}

The main result of this section says that the union of an admissible sequence forms a Morse subset which is narrow at all transversal points, provided that all $Q_i$ are Morse. 
\begin{prop}\label{Morse admissible sequence}
    Let $(P_1,Q_1,P_2 \dots ,P_{n+1})$ be a $(D,B)$-admissible sequence. Assume  each $Q_i$ is $\kappa$-Morse. Then there exist a function $\kappa' = \kappa'(D,B, \kappa)$ and a constant $\mathfrak s = \mathfrak s  (D,B, \kappa)$ such that the union $Y = \bigcup_{i=1}^{n} (P_i\cup Q_i)\cup   P_{n+1}$ is $\kappa'$-Morse and $\mathfrak s$-narrow at every transversal point.
\end{prop}
This proposition also holds for infinite admissible sequences, since all quantities in the proof are independent of $n$.
In the sequel, we call such admissible sequences as $(D,B,\kappa)$-\textit{admissible sequences}.

Our proof, which differs from the approach in \cite{yang2014growthtightness}, establishes this result through a series of elementary lemmas. 

\begin{lem}\label{right to right}
    Let $(P_1,Q_1,P_2)$ be a $(D,B)$-admissible sequence. Set $\widetilde{D} = 3D + B$. For any point $x \in X$, if $d(\pi_{P_2}(x), Q_1 \cap P_2) \geq \widetilde{D}$, then $d(\pi_{P_1}(x), P_1 \cap Q_1) \leq \widetilde{D}$.
\end{lem}

\begin{proof}
Let $x_2 \in \pi_{P_2}(x)$ be a closest  projection point of $x$ on $P_2$. We examine the following two cases.

    \textbf{Case 1.}   $[x,x_2] \cap N_D(P_1) = \emptyset$.
    By the contracting property, we have $\diam\{\pi_{P_1}(x) \cup \pi_{P_1}(x_2)\} \leq D$. Since $\diam\{\pi_{P_1}(Q_1 \cup P_2)\} \leq B$ by Definition \ref{def admissible sequence}, it follows that
    \[ d(\pi_{P_1}(x), P_1 \cap Q_1) \leq D + d(\pi_{P_1}(x_2), P_1 \cap Q_1) \leq D + \diam\{\pi_{P_1}(P_2) \cup \pi_{P_1}(Q_1)\} \leq  D + B \leq \widetilde{D}. \]

    \textbf{Case 2.}   $[x,x_2] \cap N_D(P_1) \neq \emptyset$. 
    If we take a point $x_1 \in [x,x_2] \cap N_D(P_1)$, then $x_2 \in \pi_{P_2}(x_1)$ is a closest projection point to $x_1$ (since $x_2\in \pi_{P_2}(x)$ is closest to $x$). Because $d(x_1,P_1) \leq D$, the $1$-Lipschitz projection property of contracting subset (\cref{contracting properties}) implies
    \begin{align*}
        d(x_2, \pi_{P_2}(P_1)) = d(\pi_{P_2}(x_1), \pi_{P_2}(P_1)) \leq D + d(x_1, P_1) \leq 2D.
    \end{align*}
    Consequently,
    \[ d(x_2, Q_1 \cap P_2) \leq d(x_2, \pi_{P_2}(P_1)) + \diam\{\pi_{P_2}( P_1 \cup Q_1)\} \leq 2D + B \leq \widetilde{D}, \]
    which yields a contradiction.
The lemma is then proved.
\end{proof}

Write $Y=\cup_{i=1}^n (P_i \cup Q_i) \cup  P_{n+1}$. For later convenience, the union $\mathcal L(P_i)=P_1 \cup Q_1 \cup \cdots \cup Q_{i-1}$  will be refereed to as the \textit{left side} of $P_i$, and $\mathcal R(P_i)=Q_i \cup P_{i+1} \cup \cdots \cup P_{n+1}$ is the \textit{right side} of $P_i$. The left and right sides $\mathcal L(Q_i), \mathcal R(Q_i)$ of $Q_i$ are defined analogously.

\begin{figure}[ht]
    \centering
    \def\svgwidth{0.5\columnwidth}
\begingroup%
  \makeatletter%
  \providecommand\color[2][]{%
    \errmessage{(Inkscape) Color is used for the text in Inkscape, but the package 'color.sty' is not loaded}%
    \renewcommand\color[2][]{}%
  }%
  \providecommand\transparent[1]{%
    \errmessage{(Inkscape) Transparency is used (non-zero) for the text in Inkscape, but the package 'transparent.sty' is not loaded}%
    \renewcommand\transparent[1]{}%
  }%
  \providecommand\rotatebox[2]{#2}%
  \newcommand*\fsize{\dimexpr\f@size pt\relax}%
  \newcommand*\lineheight[1]{\fontsize{\fsize}{#1\fsize}\selectfont}%
  \ifx\svgwidth\undefined%
    \setlength{\unitlength}{342.53738283bp}%
    \ifx\svgscale\undefined%
      \relax%
    \else%
      \setlength{\unitlength}{\unitlength * \real{\svgscale}}%
    \fi%
  \else%
    \setlength{\unitlength}{\svgwidth}%
  \fi%
  \global\let\svgwidth\undefined%
  \global\let\svgscale\undefined%
  \makeatother%
  \begin{picture}(1,0.68695236)%
    \lineheight{1}%
    \setlength\tabcolsep{0pt}%
    \put(0,0){\includegraphics[width=\unitlength,page=1]{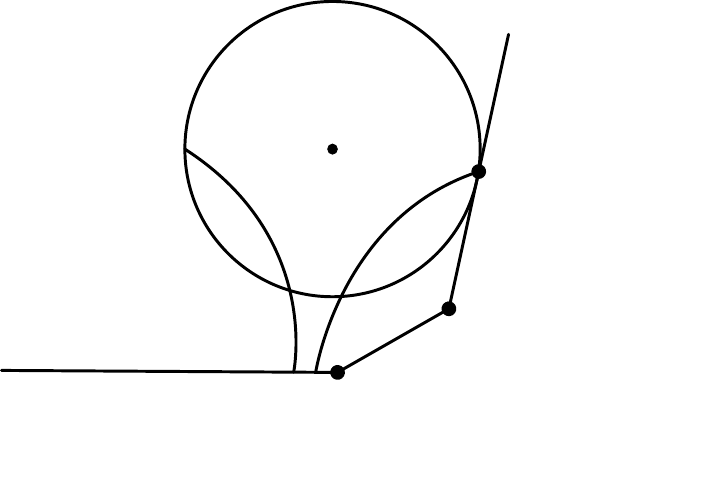}}%
    \put(0.18041294,0.08428128){\color[rgb]{0,0,0}\makebox(0,0)[lt]{\lineheight{1.25}\smash{\begin{tabular}[t]{l}$P_i$\end{tabular}}}}%
    \put(0.53107129,0.15647755){\color[rgb]{0,0,0}\makebox(0,0)[lt]{\lineheight{1.25}\smash{\begin{tabular}[t]{l}$Q_i$\end{tabular}}}}%
    \put(0.70862649,0.55316812){\color[rgb]{0,0,0}\makebox(0,0)[lt]{\lineheight{1.25}\smash{\begin{tabular}[t]{l}$P_{i+1}$\end{tabular}}}}%
    \put(0.44792897,0.50227507){\color[rgb]{0,0,0}\makebox(0,0)[lt]{\lineheight{1.25}\smash{\begin{tabular}[t]{l}$x$\end{tabular}}}}%
    \put(0,0){\includegraphics[width=\unitlength,page=2]{left_and_right.pdf}}%
    \put(0.46459411,0.01077707){\color[rgb]{0,0,0}\makebox(0,0)[lt]{\lineheight{1.25}\smash{\begin{tabular}[t]{l}$\leq D$\end{tabular}}}}%
    \put(0.67638114,0.42418961){\color[rgb]{0,0,0}\makebox(0,0)[lt]{\lineheight{1.25}\smash{\begin{tabular}[t]{l}$x_1$\end{tabular}}}}%
  \end{picture}%
\endgroup%

    \caption{If $\pi_{Y}(x) \cap P_{i+1} \neq \emptyset$, then $x$ is to the right of $P_i$.}
    \label{figure left_and_right}
\end{figure}

\begin{lem}\label{Left and Right}
    For any point $x \in X$ and each $1 \leq i \leq n+1$:
    \begin{enumerate}
        \item If $\pi_Y(x)$ intersects the left side $\mathcal L(P_i)$ of $P_i$, then $d(\pi_{P_i}(x), P_i \cap Q_{i-1}) \leq  3D + B$.
        \item If $\pi_Y(x)$ intersects the right side $\mathcal R(P_i)$ of $P_i$, then $d(\pi_{P_i}(x), P_i \cap Q_{i}) \leq  3D + B$. 
    \end{enumerate}
\end{lem}

\begin{proof}
    By symmetry, we only prove the second statement. Suppose $\pi_Y(x)$ intersects the right side $Q_i \cup P_{i+1} \cup \cdots \cup P_{n+1}$ of $P_i$. There are two cases to consider.

    \textbf{Case 1:} $\pi_Y(x) \cap (Q_i \cup P_{i+1}) \neq \emptyset$. Let $x_1 \in \pi_Y(x) \cap (Q_i \cup P_{i+1})$. As $x_1$ is a shortest projection point to $Y$, we have $d(x,x_1) \leq d(x,A)$ for any $A \subset Y$ and thus $d(x,x_1) \leq d(x,P_i)$. The contracting property for $P_i$ then yields
    \[ \diam\{\pi_{P_i}(x)\cup \pi_{P_i}(x_1)\} \leq D. \]
    See \cref{figure left_and_right} for an illustration. By Definition \ref{def admissible sequence}, we have
    \[ d(\pi_{P_i}(x_1), P_i \cap Q_i) \leq \diam\{\pi_{P_i}(Q_i \cup P_{i+1})\} \leq B. \]
    Therefore, $d(\pi_{P_i}(x), P_i \cap Q_i) \leq \diam\{\pi_{P_i}(x)\cup \pi_{P_i}(x_1)\} + d(\pi_{P_i}(x_1), P_i \cap Q_i) \leq D + B \leq \widetilde{D}$.

    \textbf{Case 2:} $\pi_Y(x) \cap (Q_j \cup P_{j+1}) \neq \emptyset$ for some $j > i$. Then by \textbf{Case 1}, we obtain $d(\pi_{P_j}(x), P_j \cap Q_j)\leq \widetilde{D}$, which implies $d(\pi_{P_j}(x),Q_{j - 1} \cap P_j) \geq \widetilde{D}$ by condition $(3)$ in \cref{def admissible sequence}.
    So applying \cref{right to right} inductively from $P_j$  to $P_i$, we conclude that $d(\pi_{P_i}(x), P_i \cap Q_i)\leq \widetilde{D}$.
\end{proof}

\begin{lem}\label{projection small}
    Assume that a metric ball $B$ does not intersect $Y=\cup_{i=1}^n (P_i \cup Q_i) \cup  P_{n+1}$. Then there exists $0 \leq i \leq n$ such that the projection satisfies $\pi_Y(B) \subset P_i \cup Q_i \cup P_{i+1}$. 
\end{lem}

\begin{proof}
First we prove a claim: For any $1 \leq j \leq n+1$, the projection $\pi_Y(B)$ does not intersect both the left and right sides of $P_j$. Indeed, if $\pi_Y(B)$ intersects both the left and right sides of $P_j$, by \cref{Left and Right}, there exist two points $b_1,b_2 \in B$ so that
    \[ d(\pi_{P_j}(b_1),P_j \cap Q_{j-1}) \leq 3D + B \bigand d(\pi_{P_j}(b_2), P_j \cap Q_j) \leq 3D + B. \]
By Definition \ref{def admissible sequence}, $d(P_j \cap Q_{j-1}, P_j \cap Q_j) \geq 2B + 10D$, which then implies $d(\pi_{P_j}(b_1),\pi_{P_j}(b_2)) \geq 4D$ by the triangle inequality. 
This  contradicts the $D$-contracting property of $P_j$ which would give:
    \[
    d(\pi_{P_j}(b_1),\pi_{P_j}(b_2)) \leq \diam\{\pi_{P_j}(B)\} \leq D.
    \]
So the claim follows. 
According to the claim, it is clear that there exists $i$ satisfying the requirement.
\end{proof}

We shall prove that if all $Q_i$ are uniformly Morse, then the union $Y$ of the admissible sequence is also Morse. This is based on the following observation.

\begin{lem}\label{Morse subset union}
    Let $M_1,M_2$ be two $\kappa$-Morse subsets with $M_1 \cap M_2 \neq \emptyset$. Then $M_1 \cup M_2$ is a $\kappa'$-Morse subset,  where $\kappa'$ only depends on $\kappa$.
\end{lem}

\begin{proof}
    Let $m \in M_1 \cap M_2$ be a common point and $\alpha$ be an arbitrary $c$-quasi-geodesic with endpoints $\alpha_-, \alpha_+$ in $M_1$ and $M_2$ respectively. Let $m_1 \in \pi_\alpha(m)$ be the projection of $m$ to $\alpha$. We claim that the concatenated paths
    \[
    \gamma_1 = [m,m_1] \circ {[m_1,\alpha_-]}_\alpha \quad \text{and} \quad \gamma_2 = [m,m_1] \circ {[m_1,\alpha_+]}_\alpha
    \]
    are $3c$-quasi-geodesics.
    In fact, for every points $x \in [m,m_1]$ and $y \in {[m_1,\alpha_-]}_\alpha$, we have 
    \begin{align*}
        d(x,m_1) + \|[m_1,y]_{\alpha}\| &\leq d(x,m_1) + cd(m_1,y) + c \\
        &\leq (c+1)d(x,m_1) + cd(x,y) + c \\
        &\leq (c+1)d(x,y) + cd(x,y) + c \\
        &\leq 3cd(x,y) + c.
    \end{align*}
    Since $\gamma_1$ and $\gamma_2$ are two $3c$-quasi-geodesics with endpoints on $\kappa$-Morse subsets $M_1$ and $M_2$ respectively, $\alpha \subset \gamma_1 \cup \gamma_2$ is contained in the $\kappa(3c)$-neighborhood of $M_1 \cup M_2$.
\end{proof}

\begin{lem}\label{admissible Union is Morse}
    If every $Q_i$ is $\kappa$-Morse, then the union $Y=\cup_{i=1}^n(P_i \cup Q_i) \cup  P_{n+1}$ is a $\kappa'$-Morse subset, where $\kappa'$ only depends on $(D,\kappa)$.
\end{lem}

\begin{proof}
Consider any two points $x,y \in X$ with $d(x,y) \leq d(x,Y)$. By Lemma \ref{projection small}, there exists an index $i$ such that:
\[ \pi_Y(x) \subseteq P_i \cup Q_i \cup P_{i+1} \quad \text{and} \quad \pi_Y(y) \subseteq P_i \cup Q_i \cup P_{i+1}. \]

We make three observations: 
\begin{enumerate}
    \item $P_i$ and $P_{i+1}$ are $D$-contracting, then they are $\kappa_1$-Morse where $\kappa_1$ depends only on $D$ by \cref{Morse is sublinear contracting};
    \item $Q_i$ is $\kappa$-Morse by assumption;
    \item The intersections $P_i \cap Q_i$ and $Q_i \cap P_{i+1}$ are non-empty (by admissibility).
\end{enumerate}

Applying \cref{Morse subset union} to $P_i \cup Q_i \cup P_{i+1}$, we obtain that this union of three subsets is $\kappa_2$-Morse, where $\kappa_2$ depends only on $D$ and $\kappa$. 

Using \cref{Morse is sublinear contracting} again, there exists a sublinear function $\rho$ depending on $\kappa_2$ such that the projection satisfies:
\begin{align*}
\diam\{\pi_Y(x)\cup \pi_Y(y)\} &= \diam\{\pi_{P_i \cup Q_i \cup P_{i+1}}(x)\bigcup \pi_{P_i \cup Q_i \cup P_{i+1}}(y)\} \\
&\leq \rho(d(x,P_i \cup Q_i \cup P_{i+1})) \\
&= \rho(d(x,Y)) .
\end{align*}
By \cref{Morse is sublinear contracting}, this implies that $Y$ is $\kappa'$-Morse for some $\kappa'$ depending only on $\rho$. The proof is now complete.
\end{proof}

We are now ready to complete the proof  of \cref{Morse admissible sequence}.

\begin{proof}[Proof of \cref{Morse admissible sequence}]
The union $Y$ is $\kappa'$-Morse by \cref{admissible Union is Morse}. We shall establish narrowness at an arbitrary transversal point $x \in P_i \cap Q_i$ and the case where $x \in Q_i \cap P_{i+1}$ is symmetric. 

To that end, decompose $Y=Y_1\cup Y_2$ as the union of $Y_1=P_i \cup \mathcal L(P_i)$ and  $Y_2=\mathcal R(P_i)$. Set  $R = 10D + 10B$.
By \cref{pass narrow point}, it suffices to prove that $B(x,R)$  intersects any geodesic $\alpha$ between $\alpha_1 \in Y_1$ and $\alpha_2 \in Y_2$. 

Note that $\alpha_2 \in \mathcal R(P_i)$ and $\diam\{P_i \cap Q_i\} \leq B$, so $d(\pi_{P_i}(\alpha_2), x) \leq 3D + 2B$ by \cref{Left and Right}. 

If $\alpha_1 \notin P_i$, then $\alpha_1 \in \mathcal{L}(P_i)$.  
By \cref{Left and Right}, we have
\[
d(\pi_{P_i}(\alpha_1), \pi_{P_i}(\alpha_2)) \geq d(P_i \cap Q_{i-1}, P_i \cap Q_i) - 6D - 2B \geq 4D.
\]
Hence, by property~$(3)$ of \cref{contracting properties}, the geodesic $\alpha$ intersects the ball $B(\pi_{P_i}(\alpha_2), 2D)$. Consequently, we obtain $\alpha \cap B(x, 5D + 2B) \neq \emptyset$.

Otherwise, $\alpha_1 \in P_i$ and  the contracting property~$(4)$ of \cref{contracting properties} again forces $\alpha$ to intersect $B(\pi_{P_i}(\alpha_2), 2D)$ and $B(x, 5D + 2B)$. The proof is complete.
\end{proof}

\subsection{Geodesic rays in relatively hyperbolic groups}
For completeness, we  give one  definition of relatively hyperbolic groups from boundary point of view. For a comprehensive overview of other equivalent definitions, we refer to \cite{Hru}.

Assume that a group $G$ admits a proper action on a proper Gromov-hyperbolic space $X$. 
The \textit{limit set} $\Lambda G$ of $G$ is the set of accumulation points of the orbit $Go$ for some (or any) $o\in X$ in the Gromov boundary $\partial X$. A limit point $\xi \in \Lambda G$ is \textit{conical} if there exists a sequence of elements $\{g_i\in G\}$ and $\eta_1,\eta_2 \in \Lambda G$ such that $g_i(\xi) \to \eta_1$ and $g_i(\xi') \to \eta_2$ for any $\xi' \in \Lambda G \setminus \{\xi\} $. 
A limit point $\xi \in \Lambda G$ is \textit{bounded parabolic} if the stabilizer $\mathrm{Stab_G(\xi)}$ acts cocompactly on $\Lambda G \setminus \{\xi\}$, and $\mathrm{Stab_G(\xi)}$ contains no loxodromic elements. In this case,  $\mathrm{Stab_G(\xi)}$ is called a \textit{maximal parabolic group}.
We call the action \textit{geometrically finite} if $\partial X=\Lambda G$ and every limit point  is either conical or bounded parabolic. By \cite[Proposition 6.15]{relatively_hyperbolic_groups}, there are only finitely many conjugacy classes of all maximal parabolic subgroups. If $\mathcal{P}$ denotes a full set of conjugacy representatives of  maximal parabolic subgroups, we say that $G$ is  \emph{hyperbolic relative to} $\mathcal P$.

Assume that $G$ is a finitely generated group. Let $\mathbb{P} = \{gP:g \in G,P \in \mathcal{P}\}$ be the left cosets of maximal parabolic subgroups. Let $X=Cay(G,S)$ denote the Cayley graph of $G$ with respect to a finite symmetric generating set $S$. 

\begin{lem}\cite[Propositions 5.1.4 \& 8.2.5]{gerasimov2011quasiconvexityrelativelyhyperbolicgroups}
    There exist a constant $D_{\mathcal{P}}$ and a function $\mathcal R$ such that $\mathbb{P}$ is a $D_{\mathcal{P}}$-contracting system with the $\mathcal R$-bounded intersection property: 
    $$\forall U, V\in \mathbb{P},\; \forall r>0: \quad \mathrm{diam}(N_r (U) \cap N_r (V)) \le \mathcal R(r).$$
\end{lem}
This motivates the following definition introduced in \cite{Hru}.
\begin{defn}
    For a geodesic $\gamma$ in $X$ and parameters $r,L >0$, we say a point $x \in \gamma$ is $(r,L)$-\textit{deep} if $B(x,L) \cap \gamma \subset N_{r}(gP) $ for some $gP \in \mathbb{P} $. Otherwise, we say $x\in \gamma$ is an \textit{$(r,L)$-transition point} (along $\gamma$).
\end{defn}

By the bounded projection property, we  fix $L$ sufficiently large so that each $(r,L)$-deep point $x$ admits a unique $gP \in \mathbb{P}$ with $N_r(gP) \supset B(x,L) \cap \gamma$.
Fix constants $r > 10D_{\mathcal{P}}$ and $L \geq 10r + \mathcal{R}(r)$ for the remainder of this section.  

Let $\gamma$ be a geodesic segment (or a geodesic ray) in $X$. We proceed to construct a $(D,B,\kappa)$-admissible sequence from $\gamma$. 
\subsubsection*{Construction} We write  $\gamma = p_1q_1p_2q_2\dots p_{N}q_N$ with $1\le N\le \infty$ as union of geodesic segments so that
\begin{itemize}
    \item Each $p_i$ consists of $(r,L)$-deep points in a unique coset $P_i' \in \mathbb{P}$;
    \item 
    Each $q_i$ consists of $(r,L)$-transition points.
\end{itemize}
We allow $p_1, q_N$ to be trivial. If $\gamma$ is a geodesic ray, it is possible that $N=\infty$.  

Let $P_i = N_{r}(P_i')$.
For each $q_i$, let $a_i$ be the first exit point from $P_i$ (if $P_i$ exists) and $b_i$ be the last entry point to 
$P_{i+1}$ (if $P_{i+1}$ exists).
Denote $Q_i = [a_i,b_i]_{q_i}$ the subpath from $a_i$ to $b_i$ on $q_i$.
This yields a covering $\gamma \subset \bigcup_{i} (P_i \cup Q_i)$.
Note that if $P_i$ intersects $P_{i+1}$, $q_i$ may enter $P_{i+1}$ first and then exit $P_i$. In this simpler case, we adopt the convention that $Q_i$ reduces to an arbitrary single point in $P_i \cap P_{i+1} \cap q_i$, where the diameter $\diam\{ P_i \cap P_{i+1} \cap q_i \} \leq \mathcal{R}(r)$ is uniform bounded.

In what follows, we will prove that the sequence $\{P_i, Q_i: 1\le i \le N \}$ forms a $(D,B,\kappa)$-admissible sequence for appropriate parameters.

\begin{lem}\label{small interscet}
    $\diam \{ P_i \cap Q_i \} \leq 4r + 10D_{\mathcal{P}} $ and $\diam \{ P_{i+1} \cap Q_i \} \leq 4r + 10D_{\mathcal{P}}$. Moreover, $P_i \neq P_{i+1}$.
\end{lem}

\begin{proof}
    We only prove $\diam \{ P_i \cap Q_i \} \leq 4r + 10D_{\mathcal{P}}$; the other inequality is similar by symmetry.
    Assume that $Q_i$ is not reduced to a point, then $a_i$ is an endpoint of the geodesic $Q_i$. 
    For any point $y_i \in P_i \cap Q_i$, we prove $d(a_i,y_i) \leq 2r + 5D_{\mathcal{P}}$, then the lemma is proved. Denote $x_i$ as the common endpoint of $p_i$ and $q_i$. 
    For the sake of brevity, now we omit the subscript $i$ in this proof.
    
    If the assertion would not hold, then $d(a,y) \geq 2r + 5D_{\mathcal{P}}$.
    Since $x$ is a $(r,L)$-deep point and $a$ is the first exist point from $P_i$, we have $d(x,a) = L$ and $a \in {[x,y]}_q$, which implies $d(x,y) \geq L$. Notice that $d(x,P') \leq r$ and $d(y,P') \leq r$, so we get $d(\pi_{P'}(x), \pi_{P'}(y)) \geq L - 2r \geq 10 D_{\mathcal{P}} $. By the property of contracting subset we can find $u,v \in [x,y]_q$ such that $d(u, \pi_{P'}(x))\leq 2D_{\mathcal{P}}$ and $d(v, \pi_{P'}(y))\leq 2D_{\mathcal{P}}$. 
    Considering $d(x,u) \leq r+2D_{\mathcal{P}} \leq L \leq d(x,a)$ and 
    $ d(y,v) \leq r+2D_{\mathcal{P}} \leq 2r + 5D_{\mathcal{P}} \leq d(y,a)$, we can see that $a \in [u,v]_q$ as shown in \cref{rel_hyp}. By the quasi-convexity of $P'$ (\cref{geodesic near a contracting geodesic}), we see $[u,v]_q \subset N_{10D_{\mathcal{P}}}(P')$, which is a contradiction to the fact that $[u,v]_q$ exits $N_r(P')$ at the point $a$.

    Next we prove $P_i \neq P_{i+1}$. If not, let $y' \in p_{i+1}$ be an $(r,L)$-deep points in $P_i$. Then $y' \in P_{i}$ and $d(a,y') \geq L  \geq 2r + 5 \mathcal{D}_{\mathcal{P}}$, which is impossible as in the last paragraph.
\end{proof}

\begin{figure}[ht]
    \centering
    \def\svgwidth{0.7\columnwidth}
    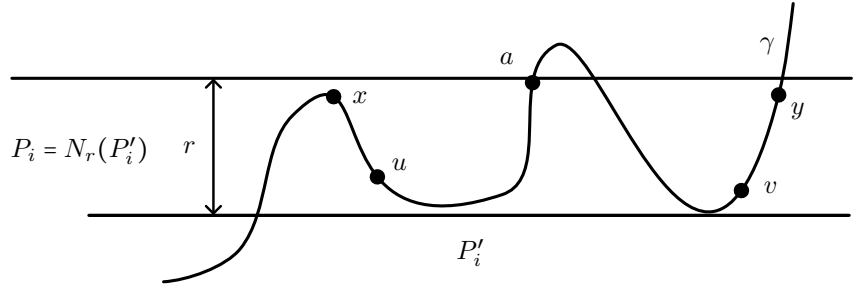

    \caption{The distance $d(a,y)$ is uniformly bounded.}
    \label{rel_hyp}
\end{figure}

\begin{lem}
    For any $r > 10D_{\mathcal{P}}$, there exists a constant $L_0$ such that for any $L \geq L_0$ the following holds.
    There are constants $(D,B)$ depending on $r$ such that
    $\{P_i, Q_i: 1\le i \le N \}$ is a $(D,B)$-admissible sequence.
\end{lem}

\begin{proof}
    Given $r > 10D_{\mathcal{P}}$,   $\mathbb{P}_r = \{N_{r}(A): A\in \mathbb{P}  \}$ remains a contracting system, and thus there exist constants $(D,B)$ such that each $P_i$ is $D$-contracting and $\diam \{\pi_{P_i}(P_j)\} \leq B $ when $P_i \neq P_j$.
    Choosing $B \geq 4r + 10D_{\mathcal{P}}$ and $L$ large enough, we have $P_i \cap Q_i \neq \emptyset$ , $Q_i \cap P_{i+1} \neq \emptyset$ and $\diam \{P_i \cap Q_i\} \leq B$, $\diam \{Q_i \cap P_{i+1}\} \leq B$ by \cref{small interscet}. Hence, $(P_1,Q_1,P_2,Q_2 \dots)$ satisfy the first two condition of admissible sequence.  
    
    Let us consider the condition $(4)$ in \cref{def admissible sequence}. According to $\diam \{\pi_{P_i}(P_j)\} \leq B$ when $P_i \neq P_j$, we only need to prove that $\diam \{\pi_{P_i}(Q_i)\} \leq B$. 
    As $P_i'$ is $D_\mathcal{P}$-contracting and $Q_i$ is a geodesic, by \cref{small interscet}, we have $\diam\{\pi_{P_i'}(Q_i)\} \leq 4r + 12 D_{\mathcal{P}}$ and
    $$\diam\{N_{D}(P_i) \cap Q_i\} = \diam\{N_{r + D}(P_i') \cap Q_i\} \leq \diam\{\pi_{P_i'}(Q_i)\} + 2r + 2D \leq 6r + 2D + 12 D_{\mathcal{P}}.$$
    Now since $P_i$ is $D$-contracting, $\diam\{\pi_{P_i}(Q_i)\} \leq 6r + 4D + 12 D_{\mathcal{P}}$.
    We simply enlarge $B$ such that $B \geq 6r + 4D + 12 D_{\mathcal{P}}$.

    Next choose $L \geq 4B + 10D + 2 \mathcal{R}(r)$ such that $d(Q_{i-1} \cap P_i, P_i \cap Q_i) \geq L - \diam\{Q_{i-1} \cap P_i\} - \diam\{ P_i \cap Q_i\} - 2 \mathcal{R}(r)\geq 2B + 10D$, so the condition $(3)$ in \cref{def admissible sequence} is also satisfied.
\end{proof}

Next, we show that each $Q_i$ is a Morse geodesic. Let us recall.

\begin{lem}\cite[Lemme 2.19]{Potyagailo_2019}\label{transition geodesic}
    There exists a constant $D_1$ depending on $(r,L)$ with the following property. Let $\gamma$ be a geodesic so that  every point on $\gamma$ is $(r,L)$-transitional. Then $\gamma$ is $D_1$-contracting. 
\end{lem}
In particular, $Q_i$ is $D_1$-contracting, hence $Q_i$ is $\kappa_1$-Morse for some function $\kappa_1$ depending on $D_1$.
According to \cref{Morse admissible sequence} we have the following desired proposition.
\begin{prop}\label{relatively hyperbolic admissible sequence.}
    For any $r > 10D_{\mathcal{P}}$, there exists a constant $L_0$ such that for any $L \geq L_0$ the following holds.
    There are constants $(D,B)$ and function $\kappa$ depending on $(r,L)$ such that
    $\{P_i, Q_i: 1\le i\le N\}$ is a $(D,B,\kappa)$-admissible sequence.
\end{prop}

\cref{Morse admissible sequence} implies that the set $ Y = \cup_{1\le i\le N}(P_i \cup Q_i)$  is narrow at the endpoints of each \( Q_i \). 
Recall that \( Y \) is constructed from a geodesic \( \gamma \). Furthermore, as demonstrated below, \( Y \) is also narrow at every \( (r, L) \)-transition point of \( \gamma \).

\begin{lem}
    If $(r,L)$ satisfy the conditions in \cref{relatively hyperbolic admissible sequence.}, then there are two functions $(\kappa,\mathfrak{s})$ such that $Y = \cup_{1\le i\le N}(P_i \cup Q_i)$ is $\kappa$-Morse and $\mathfrak{s}$-narrow at every $(r,L)$-transition point of $\gamma$.
\end{lem}

\begin{proof}
    Let $q$ be an $(r,L)$-transition point on $\gamma$ contained in some $q_i$. By the definition of $Q_i$, we have $d_{Haus}(q_i,Q_i) \leq L$, so we only need to prove $Y$ is narrow at any point of $Q_i$. 
    Since $Q_i$ is a Morse geodesic and $Y$ is narrow at the two endpoints of $Q_i$, $Y$ must be also narrow at every point of $Q_i$ by the Morse property and \cref{pass narrow point}.
\end{proof}

Combined with \cref{Ancona on Morse subset}, we recover the following result, which was first proved in \cite{gekhtman2021martin} using Floyd metric.  
\begin{prop}\label{Ancona on transition points}
    Let $G$ be a relatively hyperbolic group. Then there exists a constant $r_0$ such that for any $r > r_0$, there exists a constant $L_0$ such that for any $L \geq L_0$, the following holds.

    There exists a constant $C$ such that if points $x, z, y$ lie on a geodesic $\gamma$ in this order, and $z$ is an $(r,L)$-transition point, then
    \[
    C^{-1} \mathcal{G}(x,z)\mathcal{G}(z,y) \leq \mathcal{G}(x,y) \leq C \mathcal{G}(x,z)\mathcal{G}(z,y).
    \]
\end{prop}

\section{Martin boundary and Patterson-Sullivan measures on horofunction boundary}\label{sec preliminary of  boundary}

This section is of preliminary nature and introduces the notion of Martin boundary defined by random walks.  For later purpose, it would be helpful to understand Martin boundary as  horofunction boundary in Green metric.  

\subsection{Horofunction boundary}
Let $X$ be a proper metric space.
We recall the notion of horofunction boundary. Fix a base point $o\in X$. For  each $y \in  X$, we define a Lipschitz map $b_y:  X\to \mathbb R$ by $$\forall x\in X:\quad b_y(x)=d(x, y)-d(o,y).$$ This family of $1$-Lipschitz functions sits in the set $C(X,o)$ of continuous functions on $ X$ vanishing at $o$.

Let  $C(X,o)$ be endowed with the compact-open topology. By Arzela-Ascoli Lemma, the closure  of $\{b_y: y\in  X\}$  gives a compact metrizable space $\overline X^h$, in which  $ X$ is open and  dense. The complement   $\overline X_h\setminus X$ is called  the \textit{horofunction boundary} of $ X$ and is denoted by $\hU$. 



A \textit{Buseman cocycle} $B_\xi:  X\times X \to \mathbb R$ (independent of $o$) is given by $$\forall x_1, x_2\in  X: \quad B_\xi(x_1, x_2)=b_\xi(x_1)-b_\xi(x_2).$$
The topological type of horofunction boundary is independent of the choice of a base point. Every isometry $\phi$ of $X$ induces a homeomorphism on $\bU$:  
$$
\forall y\in X:\quad\phi(\xi)(y):=b_\xi(\phi^{-1}(y))-b_\xi(\phi^{-1}(o)).
$$
Depending on the context, we may use both $\xi$ and $b_\xi$ to denote a point in the horofunction boundary.

\subsubsection{\textbf{Finite difference relation}}
Two horofunctions $b_\xi, b_\eta$ have   \textit{finite difference} if the $L^\infty$--norm of their difference is finite: $$\|b_\xi-b_\eta\|_\infty < \infty.$$ 
The   \textit{locus} of     $b_\xi$ consists of  horofunctions $b_\eta$ so that $b_\xi, b_\eta$ have   finite difference.  The loci   $[b_\xi]$  of    horofunctions $b_\xi$ form a \textit{finite difference equivalence relation} $[\cdot]$ on $\hU$. The \textit{locus} $[\Lambda]$ of a subset $\Lambda\subseteq \hU$ is the union of loci of all points in $\Lambda$.
If $x_n\in X\to \xi\in \partial_h X$ and  $y_n\in X\to\eta\in \partial_h X$ are sequences with $\sup_{n\ge 1}d(x_n, y_n)<\infty$, then  $[\xi]=[\eta]$.

The following lemma will be useful later on.
\begin{lem}\cite[Lemma 5.5]{yang2022conformal}\label{non-pinched points}
Every boundary point $\xi$ in the horofunction boundary $\hU$ is \emph{non-pinched}:  if $x_n,y_n\in X$ tends to $[\xi]$, then the sequence of geodesics $[x_n,y_n]$ leaves every compact subset.   
\end{lem} 


Assume that a group $G$ acts  properly on $X$.
The limit set $\Lambda (Go)$  for a fixed base point $o\in X$ is defined as the set of  accumulation points of the orbit $Go$ in $\hU$. It may depend on $o$ but the $[\cdot]$-locus $[\Lambda (Go)]$ does not.

\subsection{Martin boundary}

Recall that the \textit{Green metric} on $G$ is defined by 
$$ d_{\mu}(x,y)=-\ln \frac{\mathcal{G}(x,y)}{\mathcal{G}(e,e)},$$
where $e$ is the identity in $G$.
By \cref{Greenfunction}, this satisfies the triangle inequality, though it is typically asymmetric and non-geodesic. 
As shown in \cite[Proposition 7.8]{gekhtman_entropy_2020}, the Green metric is quasi-isometric to the world metric.
We say that elements $\{x_n\} \subset G$ tend to infinity if it leaves every finite set.
In particular, we have 
\begin{align}\label{TheGreenfunctiondecay}
    \lim_{x \to \infty} \mathcal{G}(e,x) = 0.
\end{align}

The horofunction compactification of $(G,d_\mu)$ is called the \textit{Martin compactification} and denoted by $\overline{ G}_{\mathcal M}$. 
The boundary $$\partial_{\mathcal M} G=\overline{ G}_{\mathcal M}\setminus G$$ is called the \textit{Martin boundary} of $(G,\mu)$ \cite{Sawyer}. 
Explicitly, $\partial_{\mathcal M} G$ consists of all functions
$\psi: G \to \mathbb R$ such that there exists an unbounded sequence $x_{n}\in G$ with $$\psi(x)=\lim_{n\to \infty} d_{\mu}(x,x_{n})-d_{\mu}(o,x_{n})=-\lim_{n\to \infty} \ln \frac{\mathcal G(x,x_n)}{\mathcal G(e,x_n)}$$ 
for all $x\in  G$. 
For $p,q, x\in G$,  we  set  $$B_x(p,q)= d_{\mu}(p,x)-d_\mu(q,x)$$ which  extends  by continuity for $\alpha \in \partial G_{\mathcal M}$: $\displaystyle B_\alpha(p,q)=\lim_{\substack{{x_n\to\alpha}\\{x_n\in G}}} B_{x_n}(p,q).$

The Martin boundary is intimately related to the set of {$\mu-$harmonic} functions on $( G, \mu)$.
Recall that a function $h: G \to \mathbb R$ is called \textit{$\mu${-}harmonic} (or simply harmonic when there is no ambiguity) if for all $x\in  G$,
$$\sum_{g\in  G}h(xg)\mu(g)=h(x).$$

We recall the following result. When $\mu$ has finite support this is noted by Woess in \cite[Lemma 24.16]{Woess_2000}.
\begin{lem}\cite[Lemma 7.1]{gekhtman2021martin}\label{Martinharmonic}
If $\mu$ has superexponential moment, then the function defined by  $$\displaystyle K_{\alpha}(\cdot)=e^{-B_\alpha(\cdot,\, o)}=\lim_{x_n\to\alpha}\frac{\mathcal{G}(\cdot, x_n)}{ \mathcal{G}(e, x_n)}$$  is harmonic for all $\alpha\in \partial_{\mathcal M}{G}.$
\end{lem}

A harmonic function $h$ is called \textit{minimal} if for any harmonic function $g$ with $C^{-1} h\le g\le C h$ for some constant $C$, we have $g\equiv c h$ for some constant $c$. A boundary point in $\partial_{\mathcal M}{G}$ is called \textit{minimal} if the corresponding harmonic function is minimal in the above sense.
The set of minimal Martin kernels in $\partial_{\mathcal M}{G}$ forms the  minimal Martin boundary $\partial^m_{\mathcal M}{G}$.

Viewed as horofunction boundary, we could equip $\partial_{\mathcal M}{G}$ with the finite difference relation $[\cdot]$. The  minimal Martin boundary  is exactly the set of boundary points whose $[\cdot]$-closure are singletons.

\subsection{Patterson-Sullivan measures and Shadow Lemma}\label{SSshadowlem}

In \cite{Patterson}, Patterson first constructed a class of conformal measures on the limit set of Fuchsian groups, which was subsequently extended  to the higher dimension Kleinian groups by Sullivan \cite{sullivan_density_1979}. We here follow \cite{yang2022conformal} closely to study the Patterson-Sullivan measures on the horofunction boundary for any group action  with contracting elements. 

Fix a base point $o \in X$. Consider the  growth function of the ball of radius $R>0$:
$$N(o, R):=\{v\in Go: d(o, v)\le n\}.$$ The \textit{critical exponent} $\e \Gamma$ for a subset $\Gamma \subseteq G$ defined as 
\begin{equation}\label{criticalexpo}
\omega_\Gamma = \limsup\limits_{R \to \infty} \frac{\log \sharp(N(o, R)\cap \Gamma o)}{R}
\end{equation}
  is independent of the choice of $o \in X$, and intimately related to the Poincar\'e series 
\begin{equation}\label{PoincareEQ}
s\ge 0, x,y\in X, \quad \p_\Gamma(s,x, y) = \sum\limits_{g \in \Gamma} e^{-sd(x, gy)}
\end{equation}
as $\p_{\Gamma}(s,x,y)$ diverges for $s<\e \Gamma$ and converges for $s>\e \Gamma$. Thus, the  action  $G\act X$  is called of \textit{divergent type} (resp.
\textit{convergent type})   if $\p_{G}(s,x,y)$ is divergent (resp. convergent) at
$s=\e G$. 

\begin{thm}\cite[Lemma 6.3]{yang2022conformal}\label{ConformalDensityExists}
Suppose that $G$ acts properly on a proper geodesic space $X$ compactified with horofunction boundary $\hU$. Then there is an $\omega_G$-dimensional $G$-equivariant conformal density supported on $[\Lambda Go]$. That is, there is a family $\{\nu_x\}_{x\in X}$ of finite positive Borel measures on $\pU$ with
$$
\begin{aligned}
\forall g\in G, x\in X : &\quad g_\star \nu_x=\nu_{gx}\\
\nu_y-\mathrm{a.e.}\;
\xi\in \pU : & \quad  \frac{d\nu_{x}}{d\nu_{y}}(\xi) = \mathrm{e}^{-\omega_G B_\xi (x, y)} .   
\end{aligned}
$$
\end{thm}
The family of measures $\{\nu_x\}_{x\in X}$ are called  Patterson-Sullivan measures.
Write $\ax(f)=E(f)o$.
In what follows, we make the standing assumption:
\begin{conv}\label{ConventionF1}
    Let $F$ be a set of three (mutually) independent contracting elements  $f_i, i=1,2,3$, that form a contracting system  
\begin{equation}\label{SystemFDef}
\mathbb F =\{g\cdot \ax(f_i):   g\in G \}
\end{equation}
where  the axis $\ax(f_i)$    depending on the choice of a base point $o\in X$   is $D$-contracting for some $D>0$.   
\end{conv}

In practice, we may  assume that $d(o,fo)$ is large as possible, by taking sufficiently high power of $f\in F$ and the contracting constant $D$ keeps unaffected.  

Let $r>0$ and $x, y\in X$. 
First of all, define the usual cone and shadow:  
$$\Omega_{x}(y, r)\quad :=\quad \{z\in X: \exists [x,z]\cap B(y,r)\ne\emptyset\}$$
and $\Pi_{x}(y, r) \subseteq \pU$ be the topological closure  in $\pU$ of $\Omega_{x}(y, r)$.

The partial shadows $\Pi_o^F(go, r)$ and cones $\Omega_o^F(go, r)$   given in Definition \ref{ShadowDef} depend on the choice of a contracting system $\mathbb F $ as in (\ref{SystemFDef}). Without index $F$,  $\Pi_o(go, r)$ denotes the usual shadow. 
 
We recall the notion of {barrier-free} elements introduced in \cite{yang2019statistically}.
\begin{defn}\label{def barrier-free}
    Given $f\in G$ and $r > 0$, we say that a geodesic $\gamma$ has an \textit{$(r,f)$-barrier} if there is an orbit point $ho \in Go$ such that 
    \[ d(ho,\gamma) \leq r \bigand d(hfo,\gamma) \leq r. \]
    For sake of simplicity, we say that $ho$ is an $(r,f)$-barrier for $\gamma$. We say that $\gamma$ is \textit{$(r,F)$-barrier-free} for a subset $F\subset G$ if  $\gamma$ has no \textit{$(r,f)$-barrier} for any $f\in F$.
\end{defn}
 
\begin{defn}[Partial cone and shadow]\label{ShadowDef}
For $x\in X, y\in Go$, the \textit{$(r, F)$--cone} $\Omega_{x}^F(y, r)$ is the set of elements $z\in X$ such that $y$ is an $(r,F)$-barrier for $[x, z]$.
The \textit{$(r, F)$--shadow} $\Pi_{x}^F(y, r) \subseteq \pU$ is the topological closure in $\pU$ of the cone $\Omega_{x}^F(y, r)$.
\end{defn}

The key fact in the theory of conformal density is the Sullivan's  shadow lemma.
\begin{lem}\cite[Lemma 6.3]{yang2022conformal}\label{ShadowLem}
Let $\{\nu_x\}_{x\in X}$ be an $\omega_G$--dimensional $G$--equivariant conformal density. Then there exist $r_0,  L_0 > 0$ with the following property. 

Assume that $d(o,fo)>L_0$ for each $f\in F$.  For given $r \ge  r_0$, there exist $C_0=C_0(F),C_1=C_1(F, r)$ such that  
$$
\begin{array}{rl}
   C_0 \mathrm{e}^{-\omega_G \cdot  d(o, go)}  \le   \nu_o(\Pi_o^F(go,r))  \le \nu_o([\Pi_o(go,r)])   \le C_1     \mathrm{e}^{-\omega_G \cdot  d(o, go)} 
\end{array}
$$
for any $go\in Go$.
\end{lem} 
\begin{proof}
    The above statement differs slightly from the cited version in \cite{yang2022conformal} where the upper bound considered the $[\cdot]$-closure  of the usual shadow $\Pi_o(go,r)$. 
    To prove this strengthened form we need to compute the Busemann functions for this enlarged set in the proof of the upper bound argument in \cite{yang2022conformal}.
    
    Observe that since the measure $\nu_o$ is supported on the conical limit points (by \cref{HSTLem}), and the difference of Busemann functions within the same equivalence class maintain uniformly bounded (by \cite[Lemma 5.6]{yang2022conformal}), the original proof remains valid.
\end{proof}

\subsection{Conical points and admissible ray}
We give  the definition  of a conical point relative to the above $D$-contracting system $\mathbb F $ in \cref{ConventionF1}.  

\begin{defn}\label{ConicalDef2}
A point $\xi \in \pU$ is called \textit{$(r, F)$-conical}   if for some $x\in Go$, the point $\xi$ lies in infinitely many $(r, F)$-shadows $\Pi_x^{F}(y_n, r)$ for $y_n\in Go$.  We denote by  $\Lambda_{r}^F(Go) $ the set of $(r, F)$-conical points.

\end{defn}
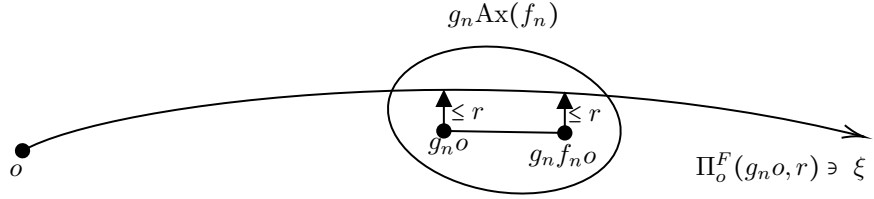
\begin{figure}
    \centering

\tikzset{every picture/.style={line width=0.75pt}} 

\begin{tikzpicture}[x=0.75pt,y=0.75pt,yscale=-1,xscale=1]

\draw    (100,103) .. controls (144.5,78) and (352.5,52) .. (524.5,96) ;
\draw [shift={(524.5,96)}, rotate = 194.35] [color={rgb, 255:red, 0; green, 0; blue, 0 }  ][line width=0.75]    (10.93,-3.29) .. controls (6.95,-1.4) and (3.31,-0.3) .. (0,0) .. controls (3.31,0.3) and (6.95,1.4) .. (10.93,3.29)   ;
\draw [shift={(100,103)}, rotate = 330.67] [color={rgb, 255:red, 0; green, 0; blue, 0 }  ][fill={rgb, 255:red, 0; green, 0; blue, 0 }  ][line width=0.75]      (0, 0) circle [x radius= 3.35, y radius= 3.35]   ;
\draw   (284.6,78.45) .. controls (287.68,58.59) and (316.32,46.54) .. (348.57,51.54) .. controls (380.83,56.54) and (404.48,76.69) .. (401.4,96.55) .. controls (398.32,116.41) and (369.68,128.46) .. (337.43,123.46) .. controls (305.17,118.46) and (281.52,98.31) .. (284.6,78.45) -- cycle ;
\draw    (312.5,93) -- (373.5,94) ;
\draw [shift={(373.5,94)}, rotate = 0.94] [color={rgb, 255:red, 0; green, 0; blue, 0 }  ][fill={rgb, 255:red, 0; green, 0; blue, 0 }  ][line width=0.75]      (0, 0) circle [x radius= 3.35, y radius= 3.35]   ;
\draw [shift={(312.5,93)}, rotate = 0.94] [color={rgb, 255:red, 0; green, 0; blue, 0 }  ][fill={rgb, 255:red, 0; green, 0; blue, 0 }  ][line width=0.75]      (0, 0) circle [x radius= 3.35, y radius= 3.35]   ;
\draw    (312.5,93) -- (312.5,75) ;
\draw [shift={(312.5,72)}, rotate = 90] [fill={rgb, 255:red, 0; green, 0; blue, 0 }  ][line width=0.08]  [draw opacity=0] (8.93,-4.29) -- (0,0) -- (8.93,4.29) -- cycle    ;
\draw    (373.5,94) -- (373.5,76) ;
\draw [shift={(373.5,73)}, rotate = 90] [fill={rgb, 255:red, 0; green, 0; blue, 0 }  ][line width=0.08]  [draw opacity=0] (8.93,-4.29) -- (0,0) -- (8.93,4.29) -- cycle    ;

\draw (303.28,94.91) node [anchor=north west][inner sep=0.75pt]  [rotate=-1.93]  {$g_{n} o$};
\draw (354.29,97.42) node [anchor=north west][inner sep=0.75pt]  [rotate=-1.93]  {$g_{n} f_{n} o$};
\draw (313,26.4) node [anchor=north west][inner sep=0.75pt]    {$g_{n}\mathrm{Ax}( f_{n})$};
\draw (92,108.4) node [anchor=north west][inner sep=0.75pt]    {$o$};
\draw (438,102.4) node [anchor=north west][inner sep=0.75pt]    {$\Pi _{o}^{F}( g_{n} o,r) \ni \ \xi $};
\draw (314.5,78.4) node [anchor=north west][inner sep=0.75pt]    {$\leq r$};
\draw (373.5,79.4) node [anchor=north west][inner sep=0.75pt]    {$\leq r$};

\end{tikzpicture}
    \caption{Conical points}
    \label{fig:conicpts}
\end{figure}

The importance of conical points lies in the fact that conical points are generic for a divergence type action. It is a part of  the Hopf-Tsuji-Sullivan dichotomy \cite[Theorem 1.10]{yang2022conformal}, where the converse (easier direction) is also true. Notice that every geometric action is of divergent type, which is the main situation considered in this paper.
\begin{lem}\label{HSTLem}
Assume that the action of $G$ on $X$ is of divergence type. Then for any sufficiently large $r$ depending on $\f$ and $\mathrm{diam}\{Fo\}\gg 0$, $\nu_o$ charges full measure on $[\Lambda_r^F(Go)]$ and every $[\xi]$-class is $\nu_o$-null.
\end{lem}

In \cite{yang2022conformal}, the notion of conical points admits a formulation using a notion of admissible rays in a certain qualitative sense. Here we only use one direction of it. 
\begin{defn}\label{admissible ray}
For any $r > 0$, a geodesic ray $\gamma$ is called an \emph{$(r,F)$-admissible ray} if it contains infinitely many distinct $(r,F)$-barriers. That is, there exists a sequence of distinct  $t_n\in G$ and a sequence of  $f_n\in F$ for $n\ge 1$ so that $d(t_no,\gamma), d(t_nf_no,\gamma)\le r$.  
\end{defn}

Let us first relate this notion with  \cref{def admissible sequence} of admissible sequences.
\begin{lem}\label{admissible ray is admissible sequence}
There exists  $D=D(\f)$ with the following property. Let $\gamma$ be an $(r,F)$-admissible ray. If $\min\{d(o,fo): f\in F\}>10D+2r$, then $\gamma=\cup_{i=1}^\infty(q_ip_i)$  is the union of a $(D,0)$-admissible sequence $(q_1,p_1,q_2,p_2,\cdots,p_n,q_n,\cdots)$.  
\end{lem}
\begin{proof}
Let $D_0$ be the contracting constant of the axes of $f \in F$. For each $(r,f_n)$-barrier $t_n$ as above, there exist points $x_n,y_n \in \gamma$ with $d(t_no,x_n),d(t_nf_no,y_n)  \leq r$ and $\|{[x_n,y_n]}_\gamma\| > \|f_n\| - 2{r}$. Since ${[x_n,y_n]}_\gamma$ has two endpoints in $N_r(t_n\ax(f))$ and $\ax(f)$ is $D_0$-contracting,  we see  ${[x_n,y_n]}_\gamma$ is $D$-contracting by \cref{geodesic near a contracting geodesic} with $D=86D_0 + 24{r}$. 

Write $q_n={[y_{n-1},x_n]}_\gamma$ and $p_n={[x_n,y_n]}_\gamma$ where $y_0:=\gamma_-$ is the initial point of $\gamma$. Without loss of generality, we may assume that $p_n$ is disjoint with $p_{n+1}$ otherwise we excise $p_{n+1}$ and $q_{n+1}$.
If $\|f\| - 2r > 10  D$, then  the sequence $(q_1, p_1, q_2, p_2, \dots)$ is $(D, 0)$-admissible in the sense of \cref{def admissible sequence}. 
\end{proof}

We say that a geodesic ray converges to (or ends at) a $[\cdot]$-class $[\xi]$ for some $\xi\in \partial_h X$ if the accumulation points of any unbounded sequence of points on the geodesic are contained in $[\xi]$.

\begin{lem}\cite[Lemma 4.7]{yang2022conformal}\label{ConicalPointsLem}
For any sufficiently large $r$, there exists $\hat r=\hat r(r,\f)>0$ with the following property. For any base point $o\in X$ and $\xi\in \Lambda_r^F(Go)$, there exists an $(\hat r,F)$-admissible ray $\gamma$ starting at $o$ ending at $[\xi]$. 
\end{lem}

Note that the constant $\hat{r}$ remains unchanged when replacing elements of $F$ with their sufficiently high powers (see \cite[Lemma 2.9]{yang2022conformal}).

\section{Proportionally  contracting rays and genericity}\label{sec prop contracting geodesic}

Let $(X,d)$ be a geodesic metric space.  We introduce and study  a class of geodesic rays called proportionally contracting rays, and prove that such geodesic rays are generic with respect to the Patterson-Sullivan measures. In the next section,   Ancona inequality will be proved along such rays.   

The materials in \cref{subsec prop contracting geodesics} involve only metric geometry without group action at all, while the remainder \cref{subsec genericity prop contracting geodesics} assumes a geometric action on $X$ with contracting elements.  

\subsection{Proportionally contracting geodesics}\label{subsec prop contracting geodesics}
We first introduce the following un-quantified version of proportionally contracting geodesics (\cref{def proportionally contracting}), and study the basic properties for later use. 

\begin{defn}[Frequently contracting rays]\label{def frequently contracting}
A geodesic ray $\gamma : [0,+\infty) \to X$ is called \textit{frequently $(D,L)$-contracting} for some $D,L>0$ if there exists an infinite sequence of disjoint subsegments $\{p_i\}_{i \in \mathbb{N}}$ along $\gamma$ so that each $p_i$ is $D$-contracting with length $\|p_i\| \geq L$.
\end{defn}

By \cref{admissible ray is admissible sequence}, an {$(r,F)$-admissible ray} is frequently $(D,L)$-contracting. We give below a similar decomposition for frequently contracting rays. In short words, the notions of admissible rays and frequently contracting rays are both specializations of admissible sequences of geodesics.  
\begin{lem}\label{freq contr rays is admissible}
Fix $L \geq 10D$. Let $\gamma$ be a frequently $(D,L)$-contracting ray.  Then $\gamma=\cup_{i=1}^\infty(p_iq_i)$  is the union of a $(D,0)$-admissible sequence $(q_1,p_1,q_2,p_2,\cdots,p_n,q_n,\cdots)$.  
\end{lem}

\begin{proof}
Let $p_i$ be given in definition and $\{q_i\}$ be the components of the complement $\gamma\setminus \cup_{i=1}^\infty p_i$. Thus, each $p_i$ is $D$-contracting with length $L \geq 10D$, while the closest point projection of $x\in \gamma \setminus p_i$ to $p_i$ is  exactly one of  the two endpoints of $p_i$. These verify all the conditions of \cref{def admissible sequence}.       
\end{proof}
 
With this decomposition, view $\gamma$ as a union of the admissible sequence and recall the left and right side from \cref{Left and Right} that $\mathcal L(p_i) = q_1 \cup p_1 \cup \dots \cup q_{i}$ and $\mathcal R(p_i) = q_{i+1} \cup p_{i+1} \cup \dots $. Denote the two endpoints of $p_i$ as $(p_{i})_- = p_i\cap q_i$ and $(p_{i})_+ = p_i \cap q_{i+1}$.

\begin{lem}\label{geodesic connects left and right}
Let $x,y \in X$ such that $\pi_{\gamma} (x)$ intersects the left side $\mathcal L(p_i)$ and $\pi_{\gamma}(y)$ intersects the right side $\mathcal R(p_i)$ for some $p_i \subset \gamma$.
Then any geodesic $[x,y]$ intersects both $N_{3D}((p_{i})_{-})$ and $N_{3D}((p_{i})_{+})$.
   
Moreover, there exist points $x',y' \in [x,y]$ with $d(x', (p_{i})_-) \leq 3D$ and $d(y', (p_{i})_{+}) \leq 3D$ such that $[x',y'] \subset N_{12D}(p_i)$ is $158D$-contracting.
\end{lem}

\begin{proof}
    Take $\bar x \in \pi_{\gamma}(x)$ such that $\bar x \in \mathcal{L}(p_i)$. 
    Since $d(x,\bar x) \leq d(x,\gamma) \leq d(x,p_i)$, the contracting property yields $\diam\{\pi_{p_i}(x)\cup  (p_{i})_{-} \} =\diam \{\pi_{p_i}(x) \cup \pi_{p_i}( \bar x) \} \leq D$. We have $\diam\{\pi_{p_i}(y)\cup  (p_{i})_{+} \} \leq D$ by symmetry.

    Thus $d(\pi_{p_i}(x),\pi_{p_i}(y)) \geq d((p_{i})_{-},(p_{i})_{+}) - 2D \geq 8D$.
    Applying \cref{contracting properties}, $[x,y]$ intersects $N_{2D}(\pi_{p_i}(x))$ and $N_{2D}(\pi_{p_i}(y))$. The triangle inequality then shows $[x,y]$ intersects both $N_{3D}((p_{i})_{-})$ and $N_{3D}((p_{i})_{+})$.
    The moreover statement follows directly from \cref{geodesic near a contracting geodesic}.
\end{proof}

In next three lemmas, let $\gamma_1$ and $\gamma_2$ be two frequently $(D,L)$-contracting rays with $L 
\geq 10D$. 

\begin{lem}\label{FCR infinite projection}
The following are equivalent:
\begin{enumerate}
    \item $\diam\{ \pi_{\gamma_1}(\gamma_2) \}= \infty $;
    \item $ \diam\{ \pi_{\gamma_2}(\gamma_1) \} = \infty$;
    \item $\|\gamma_1\cap N_{3D}(\gamma_2)\|=\infty$.
\end{enumerate}
\end{lem}

\begin{proof}
    The implication $(3) \Rightarrow (1)$ follows immediately from the definition of projections. 
    It suffices to prove $(1) \Rightarrow (3)$, since   $(2) \Leftrightarrow (3)$ by symmetry.

    Decompose $\gamma_1$ into an admissible sequence $(q_1,p_2,q_2,\dots)$ as in \cref{freq contr rays is admissible}.
    If $\diam\{ \pi_{\gamma_1}(\gamma_2) \}= \infty $,
    for all but finitely many $p_i$, the projection $\pi_{\gamma_1}(\gamma_2)$ intersects both the left and right sides of $p_i$.
    Applying \cref{geodesic connects left and right} to each such $p_i$ yields the end point  $(p_{i})_- \in \gamma_1\cap N_{3D}(\gamma_2)$.
    The infiniteness of the collection $\{p_i\}$ establishes $(3)$.
\end{proof}

We say that an unbounded subset $A$ converges to a $[\cdot]$-class $[\xi]$ for some $\xi\in \partial_h X$ if for every unbounded sequence $\{x_n\} \subset A $, all accumulation points of $\{x_n\}$ in the horofunction boundary lie in $[\xi]$.

\begin{lem}\label{frequently contracting to horofunction boundary}
    There exists a $[\cdot]$-class $[\xi]$ for some $\xi\in \partial_h X$ such that $\gamma_1$ converges to $[\xi]$. 
    Moreover, if a sequence $\{x_n\}$ converging to a boundary point $\xi'$, then \[\lim_{n \to \infty} d(\pi_{\gamma_1}(x_n),o) = \infty  \Leftrightarrow \xi' \in [\xi].\]
    In particular, $\gamma_2$ also converges to  $[\xi]$ if and only if $\diam\{ \pi_{\gamma_1}(\gamma_2) \}= \infty $.
\end{lem}

\begin{proof}
    There exists a sequence of $D$-contracting segments $\{p_i\}$ along $\gamma_1$ with $\|p_i\| \geq 10D$. 
    Passing to a sub-sequence, assume the endpoints $(p_{i})_{+}$ converge to a horofunction boundary point $\xi$.
    
    When $\lim_{n \to \infty}{d(\pi_{\gamma_1}(x_n),o)} = \infty$, we prove $\xi' \in [\xi]$.
    Fix an arbitrary point $z \in X$. The projection $\pi_{\gamma_1}(z)$ intersects the left side of $p_i$ for all but finitely many $p_i$. For any $i$ there exists $n(i) > 0$ such that $\pi_{\gamma_1}(x_{n(i)})$ intersects the right side of $p_i$. By \cref{geodesic connects left and right} the geodesic $[z,x_n]$ intersects $N_{3D}((p_{i})_{+})$. 
    The triangle inequality yields $|d(z,x_n)-d(z,(p_{i})_{+}) - d((p_{i})_{+},x_n)| \leq 6D$. Since $z$ is arbitrary, we also have $|d(o,x_n)-d(o,(p_{i})_{+}) - d((p_{i})_{+},x_n)| \leq 6D$.
    The difference of horofunction 
    \begin{align*}
        |b_{(p_{i})_{+}}(z) - b_{x_n}(z)| &= |d(z,(p_{i})_{+}) - d(o,(p_{i})_{+}) - d(z,x_n)+d(o,x_n)| \\ &\leq |d(z,(p_{i})_{+}) - d(o,(p_{i})_{+}) - d(z,x_n)+d(o,(p_{i})_{+}) + d((p_{i})_{+},x_n)|  + 6D \\
        & \leq |d(z,(p_{i})_{+})  - d(z,x_n) + d((p_{i})_{+},x_n)| + 6D \\
        & \leq 12D
    \end{align*}
    is uniformly bounded.
    Taking $i,n \to \infty$, we conclude that $|b_{\xi} - b_{\xi'}| \leq 12D$, and thus $\eta \in [\xi]$.

    When $\lim_{n \to \infty}{d(\pi_{\gamma_1}(x_n),o)} \neq \infty$, by passing to a subsequence, we may assume ${d(\pi_{\gamma_1}(x_n),o)}$ is bounded. Then $\pi_{\gamma_1}(x_n)$ intersects the left side of $p_i$ for all but finitely many $p_i$. By \cref{geodesic connects left and right}, the geodesic $[x_n,\gamma_1(N)]$ intersects $N_{3D}((p_{i})_{+})$ for $N> d(o,(p_{i})_{+})$. Hence:
    \[ b_{x_n}(\gamma_1(N)) \geq d(x_n,(p_{i})_{+}) + d((p_{i})_{+},\gamma_1(N)) -6D - d(o,\gamma_1(N)) = d(x_n,(p_{i})_{+}) -6D - d(o,(p_{i})_{+}). \]
    Taking $n \to \infty$ and $N \to \infty$ we get $\lim_{N \to \infty}b_{\xi'}(\gamma_1(N)) = +\infty$. While $b_{\gamma_1(M)}(\gamma_1(N)) = -N$ for any $M>N$. 
    Following the first case: $\gamma_1(M) \to \xi, M \to \infty$, hence $\lim_{N \to \infty}b_{\xi}(\gamma_1(N)) = -\infty$. We obtain $\xi' \notin [\xi]$.
\end{proof}

\begin{lem}\label{FCR finite projection}
    If $\diam\{ \pi_{\gamma_1}(\gamma_2) \} < \infty$, then there exists a bi-infinite geodesic $\alpha=\alpha_1\cup\alpha_2$ written as the union of half rays $\alpha_1, \alpha_2$  such that $\diam\{ \pi_{\gamma_1}(\alpha_1) \}= \infty $ and $ \diam\{ \pi_{\gamma_2}(\alpha_2) \} = \infty$.
    Moreover, $\alpha_1$ and $\alpha_2$ are frequently $(200D,L-10D)$-contracting.
\end{lem}

\begin{proof}
Decompose $\gamma_1$ into an admissible sequence $(q_1,p_2,q_2,\dots)$.
    Since $\diam\{ \pi_{\gamma_1}(\gamma_2) \} < \infty $, for all but finitely many $p_i$, the projection $\pi_{\gamma_1}(\gamma_2)$ is contained in the left side of $p_i$.
    Then for large enough $m$, $\pi_{\gamma_1}(\gamma_2(m))$ is contained in $\mathcal{L}(p_i)$ while $\pi_{\gamma_1}(\gamma_1(m))$ is contained in $\mathcal{R}(p_i)$.
    By \cref{geodesic connects left and right}, the geodesic $[\gamma_2(m),\gamma_1(m)]$ intersects a fixed ball $B((p_{i})_-,3D)$ for large $m$.
    
    Consequently, since the space $X$ is proper, a subsequence of $[\gamma_1(m), \gamma_2(m)]$ converges to a bi-infinite geodesic $\alpha$ passing through a point $o' \in B((p_{i})_-, 3D)$.
    One end $\alpha_1$ of the geodesic $\alpha$ intersects all but finitely many $\{B((p_{i})_-,3D)\}_{i \in \mathbb{N}}$ and contains infinitely many $200D$-contracting segments with length greater than $L - 10D$ by \cref{geodesic connects left and right}. Therefore $\diam\{ \pi_{\gamma_1}(\alpha_1) \}= \infty $ according to \cref{FCR infinite projection} and $\alpha_1$ is frequently $(200D,L-10D)$-contracting.
    Let $\alpha_2$ be the other end, then $ \diam\{ \pi_{\gamma_2}(\alpha_2) \} = \infty$ by symmetry.
\end{proof}

We now quantify the notion of  frequently contracting ray by measuring the proportion of contracting subsegments.   

\begin{defn}\label{def contracting length}
Let $\gamma$ be a geodesic segment in $X$. A decomposition of $\gamma = q_1p_1q_2 \dots q_np_nq_{n+1}$ into geodesic segments for some $n\ge 0$ is called a \textit{$(D,L)$-decomposition} for some $D,L>0$ if 
\begin{enumerate}
    \item Each $p_i$ ($1 \leq i\leq n$) is $D$-contracting with length $\|p_i\| \geq L$;
    \item The initial and final segments $q_1, q_{n+1}$ may be trivial (i.e. single points);
    \item
    The sum $\sum_{i=1}^n \|p_i\|$ realizes, with error within 1,  the supremum over all  decompositions with property (1) and (2).  
\end{enumerate}
The set $\{p_1, \dots, p_n\}$ will be referred to as  the $(D,L)$-\textit{contraction part} of $\gamma$, denoted by $\mathrm{contr}_{(D,L)}(\gamma)$ and  the sum $\sum_{i=1}^n \|p_i\|$ is called the $(D,L)$-\textit{contraction length} denoted by $\|\mathrm{contr}_{(D,L)}(\gamma)\|$. 
 We allow $n=0$ here if  there is no  subpath that is $D$-contracting.
\end{defn}

\begin{defn}[Proportionally contracting geodesics]\label{def proportionally contracting}
    For any $\theta \in (0,1]$ and $D,L>0$, we say that a geodesic $\gamma$ is \textit{$\theta$-proportionally $(D,L)$-contracting} if 
    \[ \|\mathrm{contr}_{(D,L)}(\gamma)\| \geq \theta\|\gamma\|>0. \]
    In particular, the length of $\gamma$ must be greater than $L$.
\end{defn}

If $\gamma = \gamma_1 \gamma_2$ is a geodesic formed by concatenating two geodesics, by the definition of contraction part we have
$ \|\mathrm{contr}_{(D,L)}(\gamma)\| \geq \|\mathrm{contr}_{(D,L)}(\gamma_1)\| + \|\mathrm{contr}_{(D,L)}(\gamma_2)\| $.
Hence when $\gamma_1$ and $\gamma_2$ are $\theta$-proportionally $(D,L)$-contracting geodesics, $\gamma$ is also $\theta$-proportionally $(D,L)$-contracting.

\begin{lem}\label{contracting length of sub-geodesic}
    If $\alpha$ is a subsegment of a geodesic $\gamma$, then  
    \[ \|\mathrm{contr}_{(2D,L)}(\alpha)\| \geq \|\alpha \cap \mathrm{contr}_{(D,L)}(\gamma)\|-2L. \]
    Here $\|\alpha \cap \mathrm{contr}_{(D,L)}(\gamma)\|$ denotes the sum of the lengths of the finitely many geodesics in $\alpha \cap \mathrm{contr}_{(D,L)}(\gamma)$.
\end{lem}

\begin{proof}
    Write the geodesic decomposition $\gamma = q_1p_1q_2 \dots p_nq_{n+1}$ as in \cref{def contracting length} with $\mathrm{contr}_{(D,L)}(\gamma) = \{p_i\}_{1 \leq i \leq n}$. Without loss of generality, we may assume $\alpha \cap p_1 \neq \emptyset$ and $\alpha \cap p_n \neq \emptyset$, thus $q_2p_3 \dots q_{n-1} \subset \alpha$.
    Since $\{p_i\}_{2 \leq i \leq n-1}$ are $D$-contracting segments of $\alpha$, we focus on $p_1 \cap \alpha$ and $p_n \cap \alpha$.
    
    If $\|p_1 \cap \alpha\| \geq L$ and $\|p_n \cap \alpha\| \geq L$, then $p_1' := p_1 \cap \alpha$ and $p_n' := p_n \cap \alpha$ are $2D$-contracting geodesics with length greater than $L$. 
    In this case: 
    \[  \|\alpha \cap \mathrm{contr}_{(D,L)}(\gamma)\| - \|\mathrm{contr}_{(2D,L)}(\alpha)\| \leq \|p_1\cap \alpha\| - \|p_1'\| + \|p_n\cap \alpha\| - \|p_n'\| = 0. \]
    
    If $\|p_1 \cap \alpha\| < L$ and $\|p_n \cap \alpha\| < L$, set $q_2' := \alpha \cap (p_1q_2)$ and $q_n' = \alpha \cap (q_{n}p_n)$.
    Then $\alpha = q_2'p_2q_3 \dots p_{n-1}q_{n}'$ is a decomposition of $\alpha$ as required in \cref{def contracting length}. 
    In this case:
    \[ \|\alpha \cap \mathrm{contr}_{(D,L)}(\gamma)\| - \|\mathrm{contr}_{(2D,L)}(\alpha)\| \leq \|p_1\cap \alpha\| + \|p_n\cap \alpha\| \leq 2L. \]

    Other cases can be proven similarly and are left to the interested reader.
\end{proof}

\subsection{Genericity of prop. contracting geodesics}\label{subsec genericity prop contracting geodesics}

From now on until the end of this section, assume that $X$ admits a geometric action of a non-elementary group $G$ with contracting elements.
We shall describe a way using group actions to produce  proportionally contracting geodesics, and then prove that they are generic in counting measure.

Let $F$ be a finite set in $G$. Recall a geodesic $\alpha$ in $X$ is \textit{$(r,F)$-barrier-free} if there exist no elements $f \in F$ and $g \in G$ such that 
\[ d(go,\alpha) \leq r \bigand d(gfo,\alpha)\leq r.\]

Recall $\|g\|=d(o,go)$ for any $g \in G$, and $\|p\|$ denote the length of a path $p$.
\begin{defn} \label{def PC elements}
    Fix constants $ \theta \in (0,1], L>0$. A geodesic $\gamma$ is called \textit{$(\theta,L)$-proportionally $(r,F)$-barrier-free}  if there are disjoint subsegments $q_i$ ($1\le i\le n$) in $\gamma$ such that each $q_i$ is $(r,F)$-barrier-free with length greater than $L$, and the following holds
    \[ \sum_{i=1}^n \|q_i\| \geq \theta \| g \|. \]
   
    Let $\mathcal{V}_{r,F}^{\theta,L}$ denote  the set of elements $g$ in $G$ for which some geodesic $[o,go]$ is $(\theta,L)$-proportionally $(r,F)$-barrier-free.
\end{defn}

Set $N(o,n)=\{g\in G: d(o,go)\le n\}$. We say a subset $A\subset G$ is \textit{exponentially generic} if for some $\lambda\in (0,1)$, 
$$
\forall n\gg 0,\;  \frac{|N(o,n)\setminus A|}{|N(o,n)|} \le \lambda^n.
$$
This is equivalent to that the growth rate of $G\setminus A$ is strictly less than that of $G$. We refer to \cite[Sec. 2.4]{gekhtman2018counting} for relevant discussions.

\begin{thm}\cite[Theorem 5.2]{gekhtman2018counting}\label{growth tight of proportionally barrier free}
    There exists $r > 0$ such that for any $\theta \in (0,1]$ and  for any non-empty $F \subset G$, there exits $L = L(\theta,F) > 0$ such that the complement $G\setminus \mathcal{V}_{r,F}^{\theta,L}$ is exponentially generic.
\end{thm}
\begin{rem}
In original version of \cite[Theorem 5.2]{gekhtman2018counting}, it was erroneously claimed that $L$ did  not depend on $F$. The proof of \cite[Theorem 5.2]{gekhtman2018counting} only proves the above statement.     
\end{rem}

Let $r>0$ be given by \cref{growth tight of proportionally barrier free}. 
Let us now fix a set $F=\{f_1,f_2,f_3\}$ of pairwise independent contracting elements. Write $F^{n} := \{f^n : f \in F\}$ in the sequel. 

\begin{lem}\label{qi barrier free}
For any $L>0$, there exists   $n=n(L)>0$ with the following property. If a geodesic segment $q$ satisfies $\|N_{r}(t\ax(f))\cap q\| < L$ for any $t\in G, f\in F$, then $q$ has no $(r, F^n)$-barrier. 
\end{lem}
\begin{proof}
If $q$ has an $(r, f^n)$-barrier $t\in G$ for $f \in F$, then $to, tf^no\in N_r (q)$. Hence $\|q\cap N_r(t\ax(f))\|\ge d(o,f^no)-2r$.      
This contradicts $\|q\cap N_{r}(t\ax(f))\|<L$ when we choose $n$ so that $d(o,f^no)>2r+L$. 
\end{proof}

Note that the family of axis $t\ax(f)$ with $f\in F, t\in G$ is uniformly contracting and has bounded intersection. We now define the constants $D_1=D_1(r),L_0>0$ for the next lemma: 
\begin{enumerate}
    \item 
    Any geodesic segment with two endpoints within $r$-neighborhood of $t\ax(f)$ is $D_1$-contracting. 
    \item 
    Any two distinct $N_{r}(t\ax(f))$ with $t\in G, f\in F$  have bounded intersection  strictly less than $L_0$.
\end{enumerate}

\begin{lem}\label{not barrier-free is proportional}
For any $\theta_1 \in (0,1), L_1>0$ and $L>L_0$, there exist $n=n(L)$ and $\theta_2 = \theta_2(L,L_1,\theta_1) > 0$ such that if a geodesic $\gamma$ is \emph{not} $(\theta_1,L_1)$-proportionally $(r,F^{n})$-barrier-free, then $\gamma$ is $\theta_2$-proportionally $(D_1, L)$-contracting. 
\end{lem}
\begin{proof}
We can decompose $\gamma$ as a concatenation of geodesic segments $\gamma=q_1 p_1 q_2\cdots p_m q_{m+1}$ for $m\ge 0$ with the following property:
\begin{enumerate}
    \item 
    Each $p_i$ has length $\ge L$, with endpoints in some $N_{r}(t_i\ax(f_i))$ with $t_i\in G, f_i\in F$.
    \item 
    Some $q_i$ could be trivial and $\|N_{r}(t\ax(f))\cap q_i\| < L$ for any $t\in G, f\in F$.
\end{enumerate}  

Indeed, let $t_1\ax(f)$ be the first axis (with orientation induced by $\gamma$) so that $\mathrm{diam}(N_{r}(t_1\ax(f))\cap \gamma)\ge L\ge L_0$. If it exists, then it is unique by the choice of $L_0$, otherwise $n=0$ and $q_1=\gamma$. Let $p_1$ be the maximal segment with two endpoints in $N_{r}(t_1\ax(f))$. We can then write $\gamma=q_1p_1\gamma_1$ where $q_1$ might be trivial. We do the same construction for $\gamma_1$ inductively. In the end we obtain the above decomposition. 

Let $n = n(L)$ be given by \cref{qi barrier free} so that each $q_i$ is $(r,F^{n})$-barrier-free.
By \cref{def PC elements}, since $\gamma$ is not a $(\theta_1,L_1)$-proportionally $(r,F^{n})$-barrier-free geodesic, we have $$\sum_{i=1}^{m+1} [\|q_i\|]_{L_1} \le \theta_1\|\gamma\|.$$ Here $[\cdot]_{L_1}$  denotes $L_1$-cut off function: $[\|q_i\|]_{L_1}=\|q_i\|$ if $\|q_i\|\ge L_1$, otherwise it is $0$. 

Let us consider those $q_i$ with $\|q_i\|<L_1$ (which may be trivial). In the above decomposition, such $q_i$ must be adjacent to some $p_i$ which has length at least $L>0$, so we obtain
$$
\sum_{i=1}^{m+1}  \left(\|q_i\|-[\|q_i\|]_{L_1}\right)  \le \frac{2L_1}{L}\sum_{i=1}^{m} \|p_i\|.
$$
Note $\sum_{i=1}^{m} \|p_i\|=\|\gamma\|-\sum_{i=1}^{m+1} \|q_i\|$. 
By the choice of the constant $D_1$, since the two endpoints of  $p_i$ lies in $N_{r}(t_i\ax(f_i))$, each $p_i$ is $D_1$-contracting.
We thus obtain $$\|\mathrm{contr}_{(D_1,L)}(\gamma)\|\ge \sum_{i=1}^{m} \|p_i\|>\frac{L(1-\theta_1)}{2L_1 + L}\|\gamma\|.$$
The proof is complete by setting $\theta_2 = \frac{L(1-\theta_1)}{2L_1 + L}$. 
\end{proof}

Let $\varepsilon$ be given by Convention \ref{cstvarepsilon}.

\begin{defn}\label{FracContrElemDefn}
Let $\theta \in (0,1)$ and $D,L>0$. 
We define the set ${\mathcal{PC}}(\theta,D,L)$ as the set of elements $g \in G$ with the following property. For any $x \in B(o,\varepsilon) $ and $y \in B(go,\varepsilon)$, any geodesic $[x,y]$ is $\theta$-proportionally $(D,L)$-contracting. 
\end{defn}

We now state the main result of this subsection. The genericity of ${\mathcal{PC}}(\theta, D, L)$ will give the genericity of proportionally contracting rays in \cref{subsec full measure}.

\begin{thm}\label{FC elements}
For any $L > 0$, there exist $\theta = \theta(L) > 0$ and $D = D(\mathbb{F}) > 0$ such that ${\mathcal{PC}}(\theta, D, L)$ is exponentially generic. Moreover, the constant $D$ is independent of $L$.
\end{thm}

\begin{proof}
Let $\mathcal{PC}_1(\theta, D, L)$ denote the set of elements $g \in G$ such that any geodesic $[o,go]$ is $\theta$-proportionally $(D,L)$-contracting for $\theta ,D,L >0$. 

\begin{claim}
    For any $L>0$, there exists $\theta_2 = \theta_2(L) > 0$ such that $\mathcal{PC}_1(\theta_2, D_1, L)$ is exponentially generic.
\end{claim}

\begin{proof}[Proof of the claim:]
We shall prove exponential generality for  a subset of $\mathcal{PC}_1(\theta_2,D_1,L)$. 
It suffices  to consider the case $L>L_0$, because the set $\mathcal{PC}_1(\theta_2, D, L)$ becomes smaller as $L$ gets larger.

For any $L>L_0$, let $n=n(L)$ be given by \cref{qi barrier free}.
For $\theta_1=1/2$, let $L_1=L_1(\theta_1,F^n)$ be given by \cref{growth tight of proportionally barrier free} so that the set $G\setminus\mathcal V_{r,F^{n}}^{\theta_1,L_1}$ is exponentially generic. If $\theta_2=\theta_2(L,L_1,\theta_1)$ is given by \cref{not barrier-free is proportional},  then $ G\setminus\mathcal V_{r,F^{n}}^{\theta_1,L_1} \subset \mathcal{PC}_1(\theta_2,D_1,L) $ and the claim is proved.
\end{proof}

Set $L_2 := 3000(D_1 + \varepsilon) + L$. We then have $ {\mathcal{PC}}(\theta, D, L_2-24D_1) \subset {\mathcal{PC}}(\theta, D, L)$ by definition.

Let $\theta_2 = \theta_2(L_2)$ be given by the claim such that $\mathcal{PC}_1(\theta_2, D_1, L_2)$ is exponentially generic. We are now going to prove the inclusion for some $\theta$:
\[
\mathcal{PC}_1(\theta_2, D_1, L_1) \subset {\mathcal{PC}}(\theta, D, L_2-24D_1).
\]

Indeed, for any $g \in \mathcal{PC}_1(\theta_2, D_1, L_2)$, the geodesic $[o,go]$ is $\theta_2$-proportionally $(D_1,L_2)$-contracting.
Without loss of generality, we assume $\varepsilon < D_1$.
Let $x \in B(o,\varepsilon)$ and $y \in B(go,\varepsilon)$. According to \cref{geodesic connects left and right}, 
every $D_1$-contracting geodesic segment of $[o,go]$ with length greater than $L_2$ induces a $158D_1$-contracting geodesic segment of $[x,y]$ with length greater than $L_2 - 24D_1$.
So for constants $\theta = \theta_2 \cdot \frac{L_2 - 24D_1}{L_2}$ and $D = 200D_1$, the geodesic $[x,y]$ is $\theta$-proportionally $(D,L_2-24D_1)$-contracting.

Therefore,  ${\mathcal{PC}}(\theta, D, L)$ is exponentially generic. The independence of $D$ from $L$ is immediate from its construction depending only on $F$ and $\varepsilon$.
\end{proof}

\section{Ancona inequality (III)  along geodesics with good points}\label{Sec proportionally contracting geodesic}

Let $(X,d)$ be a proper geodesic metric space on which $G$ acts geometrically. As promised, we  prove  the Ancona inequality in \cref{Ancona on geodesic with good points} along  geodesics with good points. Fix constants $D>0$ and $L > 10D$ throughout this section.

\subsection{Ancona inequality}

The notion of good points  will serve as middle points in the Ancona inequality.

\begin{defn}[Good points]\label{def good points}
   Let $0<\theta\le 1,R>0$. A point $x $ on a geodesic $\gamma$ is called a $(\theta,R,D,L)$-\textit{good point} if for all $y \in \gamma$ with $d(x,y) \geq R$, the subsegment ${[x,y]}_\gamma$  is $\theta$-proportionally $(D,L)$-contracting.
\end{defn}

Analogous to \cref{path far away from Morse subset}, the next preparatory result aims to establish the divergence of a proportionally contracting geodesic around good points.

\begin{lem}\label{path far away from proportionally contracting}
    For  any $0<\theta\le 1, c > 0$, there exists a constant $\mathfrak f = \frac{4Dc}{\theta} > 0$ with the following property. Let $\gamma$ be a $\theta$-proportionally $(D,L)$-contracting geodesic with endpoints $\gamma_{-}$ and $\gamma_{+}$ 
    
    Let $\alpha$  be  any path with endpoints $\alpha_-,\alpha_+$ so that $\alpha$ is disjoint with the neighborhood $N_{\mathfrak f}(\gamma)$ and $\gamma_{-}\in \pi_{\gamma}(\alpha_-), \gamma_{+}\in \pi_{\gamma}(\alpha_+)$. Then we have 
    \[ c \|\gamma\| \leq \|\alpha\|. \]
\end{lem}

\begin{proof}
    Let $\gamma = q_1p_1q_2 \dots q_np_nq_{n+1}$ be a $(D,L)$-decomposition given in \cref{def contracting length}, so we obtain  $$\|\mathrm{contr}_{(D,L)}(\gamma)\|=\sum_{i = 1}^n \|p_i\| \geq \theta \|\gamma\|.$$ 
    Note that $(q_1,p_1,q_2, \dots ,p_n,q_{n+1})$ is a $(D,0)$-admissible sequence. Choose a sequence of points $\alpha_- = x_0,x_1,x_2, \dots x_m = \alpha_+$ on the path $\alpha$ with a maximal integer $m$ such that $\|[x_i,x_{i+1}]_{\alpha}\| \leq \mathfrak f$ and $\|\alpha\| \geq (m-1)\mathfrak f$.  
    By \cref{projection small}, for each $i$, there exists a $j$ such that $\pi_{\gamma}([x_i,x_{i+1}]_{\alpha}) = \pi_{p_jq_jp_{j+1}}([x_i,x_{i+1}]_{\alpha})$.
    Here, we treat $p_1 $ and $ p_{n+1} $ as the endpoints of $\gamma$.
    
    As $\|[x_i,x_{i+1}]_{\alpha}\| \leq \mathfrak f$ and $d(x_i,p_t) \geq \mathfrak f$ for every $1\le t\le n$, the $D$-contracting property of $p_t$  implies  $\diam\{\pi_{p_t}([x_i,x_{i+1}]_{\alpha})\} \leq D$. Consequently, the projection $\pi_{\gamma}([x_i,x_{i+1}]_{\alpha})$ intersects at most two geodesics $p_j$ and $p_{j+1}$ in $\mathrm{contr}_{(D,L)}(\gamma)$, with each diameter less than $D$. 
    Observing that $\bigcup_{i = 1}^m \pi_\gamma([x_i,x_{i+1}]_{\alpha}) = \gamma$ covers $\mathrm{contr}_{(D,L)}(\gamma)$, we obtain $m\cdot 2D \geq \|\mathrm{contr}_{(D,L)}(\gamma)\|$. 
    By assumption,  $\|\mathrm{contr}_{(D,L)}(\gamma)\| \geq \theta \|\gamma\|$ and $\|\mathrm{contr}_{(D,L)}(\gamma)\| \geq L \geq 10D$, from which we deduce that $\|\mathrm{contr}_{(D,L)}(\gamma)\| \geq \frac{1}{2}(\theta \|\gamma\|+ 10D)$.
    Combining the above inequalities yields
    \[ \frac{1}{2}(\theta \|\gamma\|+ 10D) \leq \|\mathrm{contr}_{(D,L)}(\gamma)\| \leq 2mD \leq 2D(\|\alpha\|/\mathfrak f+1). \]
    It follows that 
    \[ \frac{\theta \mathfrak f}{4D}\|\gamma\| \leq \|\alpha\|, \]
    which completes the proof.
\end{proof}

We now have the following divergence properties around good points. 
\begin{lem}\label{path far away from proportionally contracting and cross good point}
    Let $x$ be a $(\theta,R,D,L)$-good point on a geodesic $\gamma$ for some  $\theta,R >0$. Then for any $c > 1$, there exists a constant $\mathfrak{f} = \frac{4Dc}{\theta}$ satisfying the following property.

    Let $\alpha$ be a path  with endpoints $g,h$ and let $\bar{g}\in \pi_\gamma(g)$ and $\bar{h}\in \pi_\gamma(h)$ denote one of their closest-point projections to $\gamma$. Assume  the following holds:
    \begin{enumerate}
        \item $\alpha \cap N_{\mathfrak{f}}(\gamma) = \emptyset$;
        \item $d(x,\bar{g}) \geq R \bigand  d(x,\bar{h}) \geq R$;
        \item $\bar{g}$ and $\bar{h}$ lie on opposite sides of $x$.
    \end{enumerate}
    Then $\|\alpha\| \geq c d(\bar{g}, \bar{h})$.
\end{lem}

\begin{proof}
    Since $x$ is a $(\theta,R)$-good point, the condition $(2)$ implies that 
    both ${[\bar{g},x]}_\gamma$ and ${[x,\bar{h}]}_\gamma$ are $\theta$-proportionally contracting.
    Their concatenation ${[\bar{g},\bar{h}]}_\gamma$ remains $\theta$-proportionally contracting.
    The conclusion follows directly from \cref{path far away from proportionally contracting} applied to ${[\bar{g},\bar{h}]}_\gamma$.
\end{proof}

Via \cref{Ancona general}, we now prove the Ancona inequality along geodesic with good points.  
\begin{thm}\label{Ancona on geodesic with good points}
Let $z_0$ be a $(\theta,R,D,L)$-good point on a geodesic $\gamma$ for some $\theta, R > 0$.   Choose $z \in G$ so that $d(z_0, zo) \leq \varepsilon$ by \cref{cstvarepsilon}. Then for any $k > 0$, there exists a constant $C$ such that the following holds.

Let $x, y \in G$ so that the orbit points $xo, yo$ are $k$-antipodal  along $(\gamma,zo)$. Then we have 
\[
C^{-1} \mathcal{G}(x, z)\mathcal{G}(z, y) \leq \mathcal{G}(x, y) \leq C \mathcal{G}(x, z)\mathcal{G}(z, y).
\]
\end{thm}

\begin{proof}
      We verify the conditions in \cref{Ancona general} for $Y = \gamma$ and $z_0$. The narrowness and quasi-geodesically-connected automatically hold since $\gamma$ is a geodesic.
      For the divergence, consider a path $\alpha$ with endpoints $\alpha_-$ and $\alpha_+$ such that $\pi_{\gamma}(\alpha_-)$ and $\pi_{\gamma}(\alpha_+)$ lie on different sides of $z_0$. 
      For any $K>0$, we need to construct a constant $\mathfrak{f}(K)$ such that 
      \begin{equation}\label{proportionallydiver}    
          K\cdot \mathbf{d}_{\gamma}(\alpha_-, \alpha_+) \leq \|\alpha\| + d(\alpha,\gamma) 
      \end{equation} 
      when $\alpha \cap N_{\mathfrak{f}}(\gamma) = \emptyset$ as in \cref{defofdivergence}.
      We proceed with the following case analysis.
      
      \begin{itemize}
          \item If $ \diam\{\pi_{\gamma}(\alpha_-)\cup z_0\} \geq R $ and $ \diam\{\pi_{\gamma}(\alpha_+)\cup z_0\} \geq R $, set $\mathfrak{f}(K) = \frac{4DK}{\theta}$ and hence $\|\alpha\| \geq K \mathbf{d}_{\gamma}(\alpha_-, \alpha_+)$ when $\alpha \cap N_{\mathfrak{f}}(\gamma) = \emptyset$ by \cref{path far away from proportionally contracting and cross good point}.
          \item If $ \diam\{\pi_{\gamma}(\alpha_-)\cup z_0\} \leq R $ and $ \diam\{\pi_{\gamma}(\alpha_+)\cup z_0\} \geq R $, choose $s \in \pi_{\gamma}(\alpha_-)$ and $t \in \pi_{\gamma}(\alpha_+)$ with $d(z_0,t) \geq R$. 
          By the goodness of $z_0$, the geodesic $[z_0,t]_{\gamma}$ is $\theta$-proportionally contracting. Since $d(z_0,t) \geq R \geq  d(z_0,s)$, the geodesic $[s,t]_{\gamma}$ is $\frac{\theta}{2}$-proportionally contracting. Hence $\mathfrak{f}(K) = \frac{8DK}{\theta}$ ensures that $\|\alpha\| \geq K \mathbf{d}_{\gamma}(\alpha_-, \alpha_+)$ when $\alpha \cap N_{\mathfrak{f}}(\gamma) = \emptyset$ by \cref{path far away from proportionally contracting}.
          \item If $ \diam\{\pi_{\gamma}(\alpha_-)\cup z_0\} \leq R $ and $ \diam\{\pi_{\gamma}(\alpha_+)\cup z_0\} \leq R $, set $\mathfrak{f}(K) =  2RK$. 
           We have 
          \[ K \cdot \mathbf{d}_{\gamma}(\alpha_-, \alpha_+) \leq 2RK \leq \mathfrak{f} \leq d(\alpha,\gamma). \]
      \end{itemize}
    According to the argument above, setting $\mathfrak{f}(K) = \frac{8DK}{\theta} + 2RK$, then $\gamma$ has $\mathfrak{f}$-divergence at $z_0$.
    The result now follows from \cref{Ancona general}.
\end{proof}
\subsection{Finding good points on prop. contracting rays}
This subsection gives a natural extension of proportionally contracting geodesics in \cref{def proportionally contracting} to   proportionally contracting rays, and prove that there is an unbounded sequence of good points on such a ray (so \cref{Ancona on geodesic with good points} could apply). 
We always assume that geodesic rays are parameterized by arc length in the rest of the paper. Recall that $\|\mathrm{contr}_{(D,L)}\|$ denotes the contraction length of a geodesic defined in \cref{def contracting length}.

\begin{defn}[Proportionally contracting rays]\label{def prop. contracting rays}
    A geodesic ray $\gamma : [0,+\infty) \to X$ is called a $\theta$-\textit{proportionally $(D,L)$-contracting ray} for some $\theta,D,L>0$ if the following  holds :
    \[\liminf_{n\rightarrow \infty}  \frac{\|\mathrm{contr}_{(D,L)}(\gamma[0,n])\|}{n}=\lim_{m\rightarrow \infty}  \inf_{n\ge m} \frac{\|\mathrm{contr}_{(D,L)}(\gamma[0,n])\|}{n}  \geq \theta>0. \]
    Here $\gamma[0,n]$ represents the initial subsegment $\gamma([0,n])$ for the sake of simplicity.
\end{defn}

By definition, for any $\delta>0$, a sufficiently long initial segment of a $\theta$-proportionally $(D,L)$-contracting ray is $(\theta-\delta)$-prop. $(D,L)$-contracting in \cref{def proportionally contracting}. Compare with \cref{def prop contracting after moment}.

The goal of this subsection is the following existence result for good points.

\begin{thm}\label{good points on proportionally geodesic ray}
    Fix $L \geq 100D$ and $\theta_0 >0$. Let $\gamma$ be a $\theta_0$-proportionally $(D,L)$-contracting ray. Then   $\gamma$ contains an unbounded sequence of  $(\theta,R,10D,L)$-good points, where $\theta:=\theta_0 / 4 $ and $ R: = 40L/\theta_0$.
\end{thm}

Notice that the definition of a good point $x$ (\cref{def good points}) demands \emph{all geodesic segments} of length greater than $R$ having $x$ as an endpoint are proportionally $(D,L)$-contracting. It has no reason that any point on a proportionally $(D,L)$-contracting ray  should be  good. Hence, the point of \cref{good points on proportionally geodesic ray} consists in finding many good points around which a uniform geometric control as above is made possible.

The proof relies on the following technical lemma.

\begin{lem}\label{continuous of bad points}
Fix $R, L > 0$. Let $t_1 < t_2 < t_3$ be parameters satisfying 
$t_1 + R \leq t_2 \leq t_3 \leq t_2 + L$. If the segment $\gamma[t_1,t_2]$ satisfies:
\[
\|\mathrm{contr}_{(D,L)}(\gamma[t_1,t_2])\| \leq \theta|t_2 - t_1|
\]
for some $\theta \in [0,1]$, then the extended segment $\gamma[t_1,t_3]$ satisfies
\[
\|\mathrm{contr}_{(D/2,L)}(\gamma[t_1,t_3])\| \leq \left(\theta + \frac{3L}{R}\right)|t_3 - t_1|.
\]
\end{lem}

\begin{proof}
    By \cref{contracting length of sub-geodesic}, we have 
    \begin{align*}
        \|\mathrm{contr}_{(D,L)}(\gamma[t_1,t_2])\| &\geq \|\gamma[t_1,t_2] \cap \mathrm{contr}_{(D/2,L)}(\gamma[t_1,t_3])\| - 2L \\
         & \geq \|\gamma[t_1,t_3] \cap\mathrm{contr}_{(D/2,L)}(\gamma[t_1,t_3])\| - 3L \\
        & = \|\mathrm{contr}_{(D/2,L)}(\gamma[t_1,t_3])\| - 3L.
    \end{align*}
    Combining this with $|t_3 - t_1| \geq |t_2 - t_1| \geq R$, we obtain the estimate
    \[ \|\mathrm{contr}_{(D/2,L)}(\gamma[t_1,t_3])\| \leq \|\mathrm{contr}_{(D,L)}(\gamma[t_1,t_2])\| + 3L \leq \theta|t_2 - t_1| + 3L \leq  \left(\theta + \frac{3L}{R}\right)|t_3 - t_1|. \]
\end{proof}
We are now ready to prove the main result of this subsection. 
\begin{proof}[Proof of \cref{good points on proportionally geodesic ray}] 
    Set $\theta = \theta_0 / 4,R = 40L/\theta_0$. 
    The goal is to produce an unbounded sequence of $(\theta, R,10D,L)$-good points.
    
    By way of contradiction, suppose there exists $N_0 > 0$ such that there are no $(\theta,R,10D,L)$-good points on $\gamma[N_0, +\infty]$.
    Then, by \cref{def good points} of good points, for each $m \geq N_0$, there exists $m'>0$ such that $|m-m'|\geq R$ and $\gamma[m,m']$ (or $\gamma[m',m]$) is not  $\theta$-proportionally $(10D,L)$-contracting.

    We divide the interval $[N_0, +\infty)$ into two categories:
    \begin{align*}
        &\mathbb{LB}:=\left\{m\geq N_0: \exists 0 \leq m' \leq m-R ,\; \|\mathrm{contr}_{(10D,L)}(\gamma[m',m])\| \leq \theta (m-m') \right\}, \\
        &\mathbb{RB}:=\left\{m\geq N_0: \exists m' \geq m+R ,\; \|\mathrm{contr}_{(10D,L)}(\gamma[m,m'])\| \leq \theta (m'-m) \right\},
    \end{align*}
    which we  refer to as \textit{left bad} and \textit{right bad} moments. 
    Note that $\mathbb{LB} \cup \mathbb{RB} = [N_0, +\infty)$. 

    \textbf{Case 1.} Assume the set $\mathbb{RB}$  of right bad moments is bounded. By enlarging $N_0$, we may assume without loss of generality that $\mathbb{RB} = \emptyset$, and consequently $\mathbb{LB}=[N_0, +\infty)$. 
    
    We first fix any $N>N_0$ in $\mathbb{LB}$. Setting $x_1 = N$, by the definition of $\mathbb {LB}$, we can find $x_2 \leq x_1 - R$ such that $\|\mathrm{contr}_{(10D,L)}(\gamma[x_2,x_1])\| \leq \theta|x_1-x_2|$. Once that $x_i \geq N_0$ is found, which implies $x_i \in \mathbb{LB}$, we continue to find $x_{i+1} \leq x_i - R$ such that $$\|\mathrm{contr}_{(10D,L)}(\gamma[x_{i+1},x_i])\| \leq \theta|x_i-x_{i+1}| .$$
    Inductively, we find a sequence of numbers $N = x_1>x_2> \dots > x_{m-1} \geq N_0 > x_m$ so that 
    \begin{enumerate}
        \item $x_i - x_{i+1} \geq R$ for each $i<m$;
        \item 
         $\|\mathrm{contr}_{(10D,L)}(\gamma[x_{i+1},x_i])\| \leq \theta|x_i-x_{i+1}|$. 
    \end{enumerate} 
    By \cref{contracting length of sub-geodesic}, we have 
    \[\|\mathrm{contr}_{(10D,L)}(\gamma[x_{i+1},x_i])\| \geq \|\gamma[x_{i+1},x_i] \cap \mathrm{contr}_{(D,L)}(\gamma[0,N])\| -2 L.\]
    Summing over $i$, we obtain 
    \begin{align*}
        \theta N &\geq \theta \sum_{i = 1}^{m-1} |x_i-x_{i+1}|
        \geq\sum_{i = 1}^{m-1} \|\mathrm{contr}_{(10D,L)}(\gamma[x_{i+1},x_i])\| \\
        &\geq \sum_{i = 1}^{m-1}\big(\|\gamma[x_{i+1},x_i] \cap \mathrm{contr}_{(D,L)}(\gamma[0,N])\| -2L\big) \\
        &= \|\gamma[x_{m},x_1] \cap \mathrm{contr}_{(D,L)}(\gamma[0,N])\| - 2(m-1)L\\
        &\geq \|\mathrm{contr}_{(D,L)}(\gamma[0,N])\| - N_0- 2(m-1)L.
    \end{align*}
    It is clear that $(m-1)R\leq \sum_{i = 1}^{m-1} |x_i-x_{i+1}| \leq N $, so we obtain 
    \begin{align*}
        \frac{\|\mathrm{contr}_{(D,L)}(\gamma[0,N])\|}{N} \leq \theta + \frac{2L}{R} + \frac{N_0}{N}.
    \end{align*}
    Now, letting $N$ tend to infinity, we get
    \begin{align*}
        \liminf_{N\rightarrow \infty} \frac{ \|\mathrm{contr}_{(D,L)}(\gamma[0,N])\|}{N} \leq \theta + 2L/R < \theta_0
    \end{align*}
    by the setting of $\theta,R$ and $L$, which contradicts to the definition of $\theta_0$.

    \textbf{Case 2.} Assume the set $\mathbb{RB}$ of right bad moments is unbounded. 
    
    Neither $\mathbb{LB}$ nor $\mathbb{RB}$ is necessarily a closed set, which would introduce non-essential technicalities into our proof. To overcome this issue, we invoke \cref{continuous of bad points} allowing us to extract the infimum of $\mathbb{RB}$ (and the supremum of $\mathbb{LB}$) by appropriately adjusting the contraction parameters.
    \Cref{continuous of bad points} tells us that if $t_2$ is a $(\theta,R,10D,L)$-left bad moment, then $t_3$ is a $(\theta + \frac{3L}{R},R,5D,L)$-left bad moment for any $t_3 \in [t_2,t_2 + L]$. Thus the supremum of any subset of $\mathbb{LB}$ is contained in 
    \[ \mathbb{LB}' := \left\{m\geq N_0: \exists 0 \leq m' \leq m-R ,\; \|\mathrm{contr}_{(5D,L)}(\gamma[m',m])\| \leq (\theta + \frac{3L}{R}) |m-m'| \right\}. \]
    By symmetry, the infimum of any subset of $\mathbb{RB}$ is contained in 
    \[\mathbb{RB}':=\left\{m\geq N_0: \exists m' \geq m+R ,\; \|\mathrm{contr}_{(5D,L)}(\gamma[m,m'])\| \leq (\theta + \frac{3L}{R}) |m'-m| \right\}.\]
    
    Following the approach from Case $1$, we will find a sequence of non-proportionally contracting segments and then compute the fraction to derive a contradiction. However, the details are more delicate.
    
    Let $x_1$ be the infimum of the right bad moment $\mathbb{RB}$. As $x_1 \in \mathbb{RB}'$, there exists $x_1' \geq x_1 + R$ such that $\|\mathrm{contr}_{(5D,L)}(\gamma[x_1,x_1'])\| \leq (\theta + \frac{3L}{R}) (x_1'-x_1)$. Notice that $x_1'$ might not be a right bad moment in $\mathbb{RB}$. Assuming that $x_i'$ is found, let $x_{i+1}$ be the infimum of the set $\mathbb{RB} \cap [x_i',+\infty)$ and find the corresponding $x_{i+1}' \geq x_{i+1} + R$ such that $\|\mathrm{contr}_{(5D,L)}(\gamma[x_{i+1},x_{i+1}'])\| \leq (\theta + \frac{3L}{R}) (x_{i+1}'-x_{i+1})$. It follows that $x_{i+1} \geq x_i'$ with equality only when $x_i'$ is a right bad moment. Note that all elements in $[x_i',x_{i+1})$ are left bad moments (if not empty), as $\mathbb{LB}\cup\mathbb {RB}=[N_0,\infty)$.

    Inductively, we construct a sequence of numbers $x_1 < x_1' \leq x_2 < x_2' \leq \cdots \leq x_m < x_m'$ for some integer $m$ such that:
    \begin{itemize}
        \item $x_i' - x_i \geq R$,
        \item $\|\mathrm{contr}_{(5D,L)}(\gamma[x_i,x_i'])\| \leq (\theta + \frac{3L}{R}) (x_i' -    x_i)$,
        \item $[x_i', x_{i+1}) \subset \mathbb{LB}$ and $x_{i+1}\in \mathbb{RB}'$.
    \end{itemize}

    Fix any $m\ge 1$ and write $A = [N_0,x_1) \cup \bigcup_{i = 1}^{m-1} [x_i',x_{i+1})$. Then $A$ consists of left bad moments, i.e. $A \subset \mathbb{LB}$. We are going to cover $A$ by intervals with small contraction proportion.  
    \begin{claim}
    There is a sequence $y_k' < y_k \leq y_{k-1}' < y_{k-1} \cdots \leq y_1' < y_1 \leq x_m$ such that:
    \begin{itemize}
        \item $y_i - y_i' \geq R$,
        \item $\|\mathrm{contr}_{(5D,L)}(\gamma[y_i',y_i])\| \leq \left(\theta + \frac{3L}{R}\right) (y_i - y_i')$,
        \item $A \subset \bigcup_{i = 1}^k [y_i',y_i]$.
    \end{itemize}   
    We refer to Figure \ref{goodpoints} for a schematic illustration.
    \end{claim}
    \begin{proof}[Proof of the claim]
    We construct $y_i$ inductively. Let $y_1$ be the supremum of $A$.
    We obtain $y_1 \in \mathbb{LB}'$ and there exists $y_1' \leq y_1 - R$ such that 
    $$\|\mathrm{contr}_{(5D,L)}(\gamma[y_1',y_1])\| \leq \left(\theta + \frac{3L}{R}\right) (y_1-y_1').$$ 
    Now for $i \geq 1$, if $y_i$ and $y_i'$ have been defined, we define $y_{i+1}$ to be the supremum of $A \cap [N_0,y_i']$ when this intersection is not empty, and find $y_{i+1}'\leq y_{i+1} -R$ such that $\|\mathrm{contr}_{(5D,L)}(\gamma[y_{i+1}',y_{i+1}])\| \leq \left(\theta + \frac{3L}{R}\right) (y_{i+1}-y_{i+1}')$.
    It follows that $y_{i+1} \leq y_{i}'$, with equality holding only when $y_i' \in \overline{A}$. Finally, we will get the number $y_k$ and $y_k' \leq y_k - R$ such that $A \cap [N_0,y_k']$ becomes an empty set. Thus, $A \subset \bigcup_{i = 1}^k [y_i',y_i]$. 
    \end{proof}

    \begin{figure}[ht]
        \centering
    \def\svgwidth{0.9\columnwidth}
    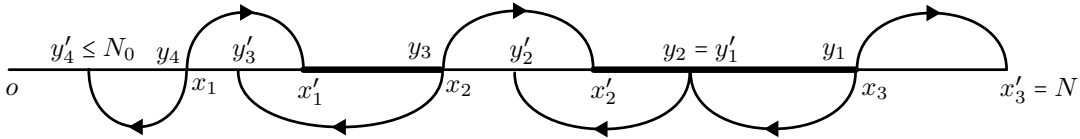

        \caption{Construction of sequences $x_i$ and $y_i$ with $m = 3,k = 4$. The bold   line represents the set $A$. In this example: $y_1 = x_3,y_2 = y_1',y_3 = x_2, y_4 = x_1$. }
        \label{goodpoints}
    \end{figure}
    Denote $N = x_m'$ for convenience in what follows. Notice that $N$ can be arbitrarily large as $m\to \infty$. 
    As a consequence of the above claim, the interval $[N_0,N]$ is covered by $\bigcup_{i = 1}^m [x_i,x_i'] \cup \bigcup_{i = 1}^k [y_i',y_i]$. Therefore, we obtain the lower bound
    \begin{align*}
        \sum_{i = 1}^{k} \|\gamma[y_i',y_i] \cap \mathrm{contr}_{(D,L)}(\gamma[0,N])\| +\sum_{i = 1}^{m} \|\gamma[x_i,x_i'] \cap \mathrm{contr}_{(D,L)}(\gamma[0,N])\|
    \geq& \|\mathrm{contr}_{(D,L)}(\gamma[0,N])\| - N_0.
    \end{align*}  
    Using \cref{contracting length of sub-geodesic},  we obtain
    \begin{align*}
        \left(\theta + \frac{3L}{R}\right) \cdot 2N &\geq  \left(\theta + \frac{3L}{R}\right) \left(\sum_{i = 1}^k |y_i - y_i'| + \sum_{i = 1}^m |x_i-x_i'|\right)  \\
        &\geq \sum_{i = 1}^{k} \|\mathrm{contr}_{(5D,L)}(\gamma[y_i',y_i])\| +\sum_{i = 1}^{m} \|\mathrm{contr}_{(5D,L)}(\gamma[x_i,x_i'])\|\\
        &\geq \sum_{i = 1}^{k} (\|\gamma[y_i',y_i] \cap \mathrm{contr}_{(D,L)}(\gamma[0,N])\| -2L) +\sum_{i = 1}^{m} (\|\gamma[x_i,x_i'] \cap \mathrm{contr}_{(D,L)}(\gamma[0,N])\| -2L)  \\
        &\geq \|\mathrm{contr}_{(D,L)}(\gamma[0,N])\| - N_0- 2(m+k)L. 
    \end{align*}
    
    Notice that $(m + k) R \leq \sum_{i = 1}^k |y_i - y_i'| + \sum_{i = 1}^m |x_i-x_i'| \leq 2N$ by construction. Finally we derive
    \begin{align}\label{without good points}
        \frac{\|\mathrm{contr}_{(D,L)}(\gamma[0,N])\|}{N} \leq 2 \theta + \frac{10L}{R} + \frac{N_0}{N}.
    \end{align}
    Recall that $\theta = \theta_0 / 4 $ and $ R = 40L/\theta_0$. As $N=x_m'$ can be arbitrarily large, we get 
    \begin{align*}
        \liminf_{N\rightarrow \infty} \frac{ \|\mathrm{contr}_{(D,L)}(\gamma[0,N])\|}{N}   \leq 2\theta + \frac{10L}{R} < \theta_0,
    \end{align*}
    contradicting again to the definition of $\theta_0$. The theorem is proved.
\end{proof}

\subsection{Full measurability of prop. contracting rays}\label{subsec full measure}

We have seen in \cref{FC elements} that proportionally contracting geodesics are generic in counting measures; we now prove that proportionally contracting rays are also generic in the Patterson-Sullivan measure on boundary.

Fix a base point $o$. For constants $L > 10D > 0$ and $\theta>0$, let $\mathcal{L}_{o}(\theta,D,L)$ denote the set of all boundary points $\xi \in \partial_h X$ such that there exists a $\theta$-proportionally $(D,L)$-contracting ray starting from the base point $o$ converges to $[\xi]$.

\begin{lem}\label{L_o full measure}
    Let $\{\nu_x\}_{x\in X}$ be an   $\omega_G$--dimensional $G$--equivariant conformal density on $\partial_h X$. 
    There exists $D=D(\mathbb{F})>0$ such that for any $L>10D$, there exists $\theta>0$ for which the set $\mathcal{L}_{o}(\theta,D,L)$ is of full $\nu_o$-measure.
\end{lem}

\begin{proof}
    We choose the constants $D$ and $\theta$ according to \cref{FC elements}.
    By \cref{HSTLem}, the conical points $[\Lambda_r^F(Go)]$ have full measure. Thus, it suffices to show that $[\Lambda_r^F(Go)]\setminus  \mathcal{L}_o(\theta,D,L)$ is of zero measure.
    Consider any $\xi \in [\Lambda_r^F(Go)] \setminus \mathcal{L}_o$. By \cref{ConicalPointsLem}, there exists a geodesic ray $\gamma$ starting at $o$ ending at $[\xi]$. Since $\xi \notin \mathcal{L}_o$, the ray $\gamma$ fails to be $\theta$-proportionally $(D,L)$-contracting. By definition, this implies there exist $\delta > 0$ and an unbounded sequence $\{x_i\} \subset \gamma$ such that for all $i \geq 1$:
    \[
    \|\mathrm{contr}_{(D,L)}([o,x_i]_{\gamma})\| \leq (\theta - \delta) d(o,x_i).
    \]
    If we choose    $g_i \in G$ with $d(g_io, x_i) \leq \varepsilon$, the above inequality implies $g_i \notin {\mathcal{PC}}(\theta,D,L)$ by \cref{FracContrElemDefn}. Since $\gamma$ tends to the locus $[\xi]$, the boundary point $\xi$ belongs to these infinitely many loci of shadows $[\Pi_o(g_i o,r)]$.

    Recall that ${\mathcal{PC}}(\theta,D,L)$ is exponentially generic by the choice of the parameters satisfying \cref{FC elements}. The shadow lemma (\cref{ShadowLem})   gives $\nu_{o} ([\Pi_o(g_i o,r)])\leq C_1 e^{-\omega \cdot d(o,g_io)}$, so  we obtain 
    \begin{align*}
        \sum_{g \notin {\mathcal{PC}}(\theta,D,L)} \nu_{o} ([\Pi_o(g o,r)]) \leq \sum_{n = 1}^{\infty} C e^{-n \omega } e^{n(\omega-\epsilon_1)} < \infty.
    \end{align*}
    
    Using Borel-Cantelli Lemma, the set of boundary points that belongs to infinitely many $[\Pi_o(g o,r)]$ with $g \notin {\mathcal{PC}}(\theta,D,L)$ has zero $\nu_0$ measure. Since $\xi\in [\Lambda_r^F(Go)]\setminus \mathcal{L}_o$ is contained in infinitely many such shadows, it follows that $\nu_o ( [\Lambda_r^F(Go)]\setminus  \mathcal{L}_o) = 0$. Thus, the set $\mathcal{L}_o$ is of full measure.
\end{proof}

The set $\mathcal{L}_o$ depends on the base point $o \in X$. We now introduce a base-point-independent version.
For constants $L > 10D > 0$ and $\theta > 0$, let $\mathcal{L}(\theta, D, L)$ denote the set of all boundary points $\xi \in \partial_h X$ such that there exists a $\theta$-proportionally $(D, L)$-contracting ray starting from some point in $X$ that converges to $[\xi]$.

An advantage of this definition is that, since we do not require the ray to start at a specific base point, the set $\mathcal{L}(\theta, D, L)$ is $G$-invariant.
It is straightforward that $\mathcal{L}_o(\theta, D, L) \subset \mathcal{L}(\theta, D, L)$. We next prove that these sets are quantitatively equivalent:
\[
\mathcal{L}_o(\theta, D, L) \subset \mathcal{L}(\theta, D, L) \subset \mathcal{L}_o\left(\frac{L - 20D}{L} \cdot \theta, 200D, L - 20D\right),
\]
as shown in \cref{L_o and L}.

\begin{lem}\label{L_o and L}
    If $\gamma_1$ is a $\theta_1$-proportionally $(D_1,L_1)$-contracting ray starting at an arbitrary point $x \in X$, then there exists a $\theta_2$-proportionally $(D_2,L_2)$-contracting ray $\gamma_2$ starting at $o$ converging to the same locus as $\gamma_1$, where $D_2 = 200D_1,L_2 = L_1 - 20D_1$ and $\theta_2 = \frac{L_1-20D_1}{L_1} \cdot \theta_1$.
\end{lem}

\begin{proof}
    Since $X$ is a proper metric space, the sequence of geodesics $[o,\gamma_1(n)]$ converges to a geodesic $\gamma_2$ as $n \to \infty$.
    Let $p \subset \gamma_1$ be a $D_1$-contracting segments of $\gamma_1$ with $\|p\| \geq L_1$.
    If $d(x,p) > d(x,o)$, then $\pi_{\gamma_1}(o)$ is contained in the left side of $p$ along $\gamma_1$.
    By \cref{geodesic connects left and right}, for all large $n$, the geodesic $[o,\gamma_1(n)]$ contains a $158D$-contracting segment $p' \subset N_{12D}(p)$ with $\|p'\| \geq p - 6D$.
    Set $D_2 = 200D_1$ and $L_2 = L_1-20D_1$.
    This implies 
    \[ \|\mathrm{contr}_{(D_2,L_2)}(\gamma_2[0,N + d(o,x)])\| \geq \frac{L_2}{L_1}\|\mathrm{contr}_{(D_1,L_1)}(\gamma_1[0,N])\|.\]
    We obtain 
    \[ \liminf_{N \to \infty}\frac{\|\mathrm{contr}_{(D_2,L_2)}(\gamma_2[0,N])\|}{N} \geq \frac{L_2}{L_1} \liminf_{N \to \infty}\frac{\|\mathrm{contr}_{(D_1,L_1)}(\gamma_1[0,N])\|}{N} \geq \frac{L_2}{L_1} \theta_1. \]
    So $\gamma_2$ is a $\frac{L_2\theta_1}{L_1}$-proportionally $(D_2,L_2)$-contracting ray.
\end{proof}

\begin{cor}\label{the map is defined on a full mearsure subset}
    There exists $D=D(\mathbb{F})>0$ such that for any $L>10D$, there exists $\theta>0$ for which the set $\mathcal{L}_{o}(\theta,D,L)$ has a $G$-invariant subset $\mathcal{L}$ of full $\nu_o$-measure.
\end{cor}

\begin{proof}
    By \cref{L_o and L}, the subset $\mathcal{L} :=\mathcal{L}(\frac{10L +D}{10L} \cdot \theta,\frac{D}{200},L + \frac{D}{10}) \subset \mathcal{L}_{o}(\theta,D,L)$ is $G$-invariant.
    Let $D = 200D_1(\mathbb{F})$ where $D_1$ is from \cref{L_o full measure}. Then for any $L> 10D$, by \cref{L_o full measure}. there exists $\theta$ such that $\mathcal{L} \supset \mathcal{L}_o(\frac{10L +D}{10L} \cdot \theta,\frac{D}{200},L + \frac{D}{10})$ has full measure, which completes the proof.
\end{proof}

\section{Application (I): embed a full measure subset  into Martin boundary}\label{Sec embedding}

In this section, let $X$ be a proper geodesic metric space,  equipped with the horofunction boundary $\hU$. Assume  that the group $G$ under consideration acts geometrically on  $(X,d)$ with contracting elements (\cref{sec groups with contracting}). There is a canonical $G$-invariant measure class denoted by $\nu_{ps}$ of the Patterson-Sullivan measures  on $\hU$  (\cref{ConformalDensityExists}). 

Fix a vertex $o\in X$ and consider the coarsely-defined map
\[ \Psi : Go \longrightarrow  G, \quad go \longmapsto g .\]
which is a quasi-isometry (quasi-inverse to the orbit map $\Phi: G\to Go$ in \cref{cstvarepsilon}). 

Let $\partial_{\mathcal M}G$ denote the Martin boundary of a finitely supported irreducible $\mu$-random walk on $G$. 
The principal result will be an embedding of a full $\nu_{ps}$-measure subset of $\hU$ into $\partial_{\mathcal{M}} G$.  

\begin{thm}\label{the map from horofunction to Martin}
    There exists a $\nu_{ps}$-full measure subset $\mathcal{L}$ of the horofunction boundary of $X$, and a $G$-equivariant map $\partial\Psi$ from $\mathcal{L}$ to the minimal Martin boundary of $G$ so that
    \begin{enumerate}
        \item $\partial\Psi$ is injective up to the finite difference equivalence.
        \item $\partial\Psi$ is continuous with respect to the direct limit topology on ${\mathcal L}$.
        \item For any point $\xi \in  {\mathcal L}$, there is a proportionally contracting ray accumulating into $[\xi]$.
    \end{enumerate}
\end{thm}
We shall refer to $\partial\Psi$ as a partial boundary map. The proof outline goes as follows. 
The subset $\mathcal{L}$ consists of ($[\cdot]$-classes of) the endpoints of a family of proportionally contracting rays in $\partial_h X$, which has $\nu_{ps}$-full measure by \cref{the map is defined on a full mearsure subset}. 
In \cref{subsec Martin kernels converge along proportionally contracting rays}, we shall prove that each prop. contracting ray tends to a  minimal Martin boundary point. This defines the desired map $\partial\Psi: \mathcal{L}\to \partial^m_{\mathcal{M}}G$ in \cref{sec map from horofunction boundary to Martin boundary}, where the well-definedness and injectivity of $\partial\Psi$ are established. Finally, \cref{sec continuity}  introduces a  direct limit topology on $\mathcal{L}$ and proves the continuity of $\partial\Psi$.

\subsection{Martin kernel converges along prop. contracting rays}\label{subsec Martin kernels converge along proportionally contracting rays}
In fact, we prove that any sequence of points in every {linear neighborhood} of such a ray will converge to the same Martin boundary point. 
 
\begin{defn}\label{def linearnbhd}[cf. \cref{linearnbhdDefn}]
    Given $k\ge 0$, the \textit{$k$-linear neighborhood} of a geodesic ray in $X$ is defined  by
    \[ \mathcal{N}_k(\gamma) := \left\{ x \in X :d(x,\gamma) \leq k \cdot\mathbf {d}_{\gamma}(x,\gamma(0)) \right\} \]
    where $\mathbf {d}_{\gamma}(x,\gamma(0))=\diam\{ \pi_{\gamma}(x)\cup \gamma(0)\}$ by definition.
     
\end{defn}
 
By abusing language, we say that a distinct sequence of  elements  $g_n\in G$ lies in a linear neighborhood of $\gamma$ if  the orbit points $\{g_nx\}$  are contained in $\mathcal N_k(\gamma)$ for some $k>0$ and for some base point $x\in X$. 

Let $\partial^m_{\mathcal{M}}G$ denote the set of minimal boundary points in the Martin boundary $\partial_{\mathcal{M}}G$.
In the proof, we will use a key lemma from \cite{cordes2022embedding}. We remark that this lemma is stated there for $g_n$ lying on a Morse geodesic ray, but the proof is valid only assuming these points satisfy Ancona inequalities.
\begin{lem}\cite[Lemma 4.1]{cordes2022embedding}\label{AnconaIneqImplyMinimal}
    Let $\{g_n \in G: n\ge 1\}$ be a sequence of elements  converging to a Martin boundary point $\zeta$. If there exists a constant $C>0$ such that 
    \[ C^{-1}\mathcal{G}(g_i,g_j)\mathcal{G}(g_j,g_k)\leq \mathcal{G}(g_i,g_k) \leq C\mathcal{G}(g_i,g_j)\mathcal{G}(g_j,g_k), \forall 1 \leq i \leq j \leq k,\]
    then $\zeta$ is minimal.
\end{lem}

We now give a criterion to determine when two sequences of elements converges to the same boundary point up to a finite ratio.  
\begin{lem}\label{Criterion of Martin convergence to locus}
Let $\{g_n\},\{h_n\}$ be two sequences of  distinct elements in $G$. Suppose there exists $C>0$  with the following property for any $x\in G$. 

For all large $m>0$ and $n \gg m$, the triple $(x,g_m,h_n)$  satisfies the Ancona inequality :
$$
C^{-1}\mathcal{G}(x,g_m)\mathcal{G}(g_m,h_n) \leq \mathcal{G}(x,h_n)\leq C \mathcal{G}(x,g_m)\mathcal{G}(g_m,h_n).    
$$
If $g_n$ converges to a Martin boundary point $\zeta$, then any accumulation point $\eta$ of $h_n$ is contained in the locus of $\zeta$. That is, $K_\zeta/K_\eta$ is   bounded from below and above.
\end{lem}
\begin{proof} 
Without loss of generality, assume $h_n\to \eta$.
Let us fix $x\in G$ first and bound the ratio $K_\zeta(x)/K_\eta(x)$ as follows.
By assumption, given large $m>0$, there exists $n_0=n_0(m)$ so that   the Ancona inequalities for the triples $(x,g_m,h_n)$ and $(e,g_m,h_n)$ give
\[
\mathcal{G}(x, h_n) \asymp_C \mathcal{G}(x, g_m) \mathcal{G}(g_m, h_n),
\]
\[
\mathcal{G}(e, h_n) \asymp_C \mathcal{G}(e, g_m) \mathcal{G}(g_m, h_n).
\]
We divide both sides of the two equations. Letting $n \to \infty$, the convergence  $\{h_n\} \to \eta$ shows $$K_{\eta}(x) \asymp_C \frac{\mathcal{G}(x, g_m)}{\mathcal{G}(e, g_m)}$$
and then letting $m\to \infty$ yields $K_\eta(x)\asymp_C K_\zeta(x)$.
As $x$ is arbitrary,  the ratio $K_\zeta(x)/K_\eta(x)$ is uniformly bounded from above and below. 
\end{proof}

Combining \cref{Criterion of Martin convergence to locus} with \cref{good points on proportionally geodesic ray} and \cref{Ancona on geodesic with good points}, we obtain:

\begin{thm}\label{EFCG converge}
Let $\gamma$ be a proportionally $(D,L)$-contracting ray in $X$ for $L \geq 100 D$.  Then there exists a  point $\zeta \in \partial^m_{\mathcal{M}}G$ so that any distinct sequence of elements   $h_n\in G$ in any linear neighborhood of $\gamma$ converges to $\zeta$.
\end{thm}
\begin{proof}
Let $z_n$ be an unbounded sequence of good points on $\gamma$ given by \cref{good points on proportionally geodesic ray}. We  choose a sequence of elements $g_n \in G$ with $d(g_no,z_n) \leq \varepsilon$ by the co-compact action.

By \cref{Ancona on geodesic with good points}, any $k$-antipodal pair of points  $xo, yo$ with $x,y\in G$  along $(\gamma,g_no)$ satisfies the Ancona inequality :
$$
C^{-1}\mathcal{G}(x,g_n)\mathcal{G}(g_n,y) \leq \mathcal{G}(x,y)\leq C \mathcal{G}(x,g_n)\mathcal{G}(g_n,y),   
$$
where $C$ depends on $k$.

By compactness, let us choose a sub-sequence of $g_n \in G$ (still denoted by $g_n$) so that it  converges to a  boundary point $\zeta$. (Passing to subsequence is not really necessary, as our goal is that any sequence in a linear neighborhood converges.)  Since any  tuple $(g_i,g_j)$ is $k$-antipodal along $(\gamma,g_mo)$ for $i<m<j$, the triple $(g_i,g_m, g_j)$ satisfies the Ancona inequality. By \cref{AnconaIneqImplyMinimal}, $K_\zeta$ is a minimal harmonic function. For any $k>0$, our goal is to prove that any $h_n$ with $h_no\in \mathcal N_k(\gamma)$ converges to $\zeta$.

According to the criterion   of \cref{Criterion of Martin convergence to locus}, it suffices to verify the following. Fix an arbitrary element $x \in G$. Given any large $m> 0$, $xo$ and $h_no$ are $(2k+1)$-antipodal along $(\gamma,g_m o)$ for all large $n$. Once this is proved, any accumulation point $\eta$ of $h_n$ lies in the locus of $\zeta$: $K_\zeta/K_\eta$ is bounded from below and above and thus the minimality of $K_\zeta$ implies $\eta=\zeta$.

The remainder is to prove the $(2k+1)$-antipodality of $(xo,h_n o)$. 
Indeed, since $\lim_{i \to \infty} d(o, g_io) = +\infty$, 
we fix $m$ such that $d(o,g_m o) \geq \mathbf{d}_{\gamma}(o , xo)$ and $\mathbf{d}_{\gamma}(x o, g_mo) \geq d(x o, \gamma)$. 
For every $n \in \mathbb{N}$, since $h_no \in \mathcal{N}_k(\gamma)$, we have 
\[ \mathbf{d}_{\gamma}(o , h_no) \geq \frac{1}{2k+1} (d(h_no, \gamma) +  \mathbf{d}_{\gamma}(o, h_no)) \geq \frac{1}{2k+1}d(o,h_no), \]
which tends to infinity as $n \to \infty$.

Observe that $\lim_{n \to \infty} d(o,\pi_{\gamma}(h_no)) = \infty$. To see it, let us decompose $\gamma = \cup_{i=1} (q_ip_i)$ as an admissible sequence. By \cref{projection small}, there exists an index $i=i(n)$ such that the projection $\pi_{\gamma}(h_no) \subset q_{i}p_{i}q_{i+1}$ is contained in three segments.
Then the observation follows from the  equivalence
\[ \lim_{n \to \infty} d(o,\pi_{\gamma}(h_no)) = \infty \Leftrightarrow
\lim_{n \to \infty} i(n) = \infty \Leftrightarrow
\lim_{n \to \infty} \mathbf{d}_{\gamma}(o,h_no) = \infty .\]

Consecutively,  for every large enough $n$, $\pi_{\gamma}(xo)$ and $\pi_{\gamma}(h_no)$ lie on the opposite sides of $g_mo$.

To see $(2k+1)$-antipodality,  we choose $n$ even larger such that $d(o,\pi_{\gamma}(h_no)) > 2 d(o,g_mo)$.
Computing the distance
\[ d(h_no,\gamma) \leq k \mathbf{d}_{\gamma}(o,h_no) \leq 2k \mathbf{d}_{\gamma}(g_mo,h_no), \]
and recalling that $\mathbf{d}_{\gamma}(x o, g_mo) \geq d(x o, \gamma)$, we have that $xo$ and $h_no$ are $(2k+1)$-antipodal along $(\gamma,g_m o)$. The claim is  proved and $h_n\to \eta=\zeta$ thus follows by minimality of $\zeta$.
\end{proof}

\subsection{The partial boundary map}\label{sec map from horofunction boundary to Martin boundary}
Recall that the constant $\varepsilon$ is fixed by the group action in \cref{cstvarepsilon}.
Fix $r > \varepsilon$ and contracting elements $F$ satisfying \cref{ShadowLem} and \cref{ConicalPointsLem}.
In this subsection we fix constants $D>0, L > 10^6 D$ and $\theta \in (0,1)$ satisfying \cref{the map is defined on a full mearsure subset}.

\begin{defn}[The boundary map $\partial\Psi: \mathcal{L}\to \partial_{\mathcal M} G$]\label{def from horofunction to Martin}
Let $\mathcal{L} \subset \partial_h X$ be given by \cref{the map is defined on a full mearsure subset} which is $G$-invariant and has full $\nu_{ps}$ measure. For every $\xi \in \mathcal{L}$, there exists a $\theta$-proportionally $(D,L)$-contracting ray $\gamma$   from the base point $o$ ending at the locus $[\xi]$.  Define $\partial\Psi(\xi)$ to be the Martin boundary point at which   the geodesic ray $\gamma$ ends, as established in Theorem \ref{EFCG converge}. 
\end{defn}

Let $\gamma_1$ and $\gamma_2$ be two proportionally contracting rays and $\zeta_1$ and $\zeta_2$ be the corresponding Martin boundary points at which $\gamma_1$ and $\gamma_2$ end.

\begin{lem}\label{the map is well defined}
    The map $\partial\Psi: \mathcal{L}\to \partial_{\mathcal M} G$ is well defined.
\end{lem}

\begin{proof}
If $\gamma_1$ and $\gamma_2$ tend to the same locus $[\xi]$, thus, $\diam\{ \pi_{\gamma_1}(\gamma_2) \} =\diam\{ \pi_{\gamma_2}(\gamma_1) \} = \infty$ by \cref{frequently contracting to horofunction boundary}.
The goal is to prove  $\zeta_1=\zeta_2$. Indeed, by \cref{FCR infinite projection} we can find an unbounded sequence of points $x_n \in \gamma_1$ such that $d(x_n,\gamma_2) \leq 3D$ for each $n\ge 1$. According to \cref{cstvarepsilon}, there exist $\{g_n\in G\}$ such that $d(g_no,x_n) \leq \varepsilon$.
    It is clear that $g_n$ are eventually contained in both the $1$-linear neighborhoods  $\mathcal{N}_1(\gamma_1)$ and $\mathcal{N}_1(\gamma_2)$.
    Applying again \cref{EFCG converge} to $\gamma_2$, we have $K(\cdot ,g_i)$ converges to the minimal harmonic function $K_{\zeta_1} = K_{\zeta_2}$. The proof is complete.
\end{proof}

The injectivity of the map  needs the following.
\begin{lem}\label{bddproj imply visual}
Assume  $ \diam\{ \pi_{\gamma_1}(\gamma_2) \} < \infty $ and $  \diam\{ \pi_{\gamma_2}(\gamma_1) \} < \infty$. Let $\alpha$ be the bi-infinite geodesic given in \cref{FCR finite projection}.
Then there is a two-sided unbounded sequence of good points on $\alpha$. 
Moreover, we have $\zeta_1 \neq \zeta_2$.
\end{lem}
\begin{proof} 
    Let $\alpha_1,\alpha_2$ denote the two half rays of the bi-infinite geodesic $\alpha$ separated at a point $o'\in \alpha$.
    By \cref{FCR finite projection}, $\alpha_1$ and $\alpha_2$ are proportionally $(200D,L-10D)$-contracting rays. 
    By \cref{good points on proportionally geodesic ray} applied to each ray, there exist constants $\theta_0 = \theta / 4,R,\widetilde{D} =2000D, \widetilde{L} = L - 10D$ and two unbounded sequences of
    $(\theta_0, R, \widetilde{D},\widetilde{L})$-good points $\{x_i:i\ge 1\} $ on $\alpha_1$ and 
    $\{y_i:i\ge 1\} $ on $\alpha_2$ respectively.
    The goal of the proof is to show that all but finitely many of these points remain good points on the full bi-infinite geodesic $\alpha$ (not just good points on the corresponding half-rays $\alpha_1, \alpha_2$).

     \begin{figure}[ht]
        \centering
    \def\svgwidth{0.8\columnwidth}
    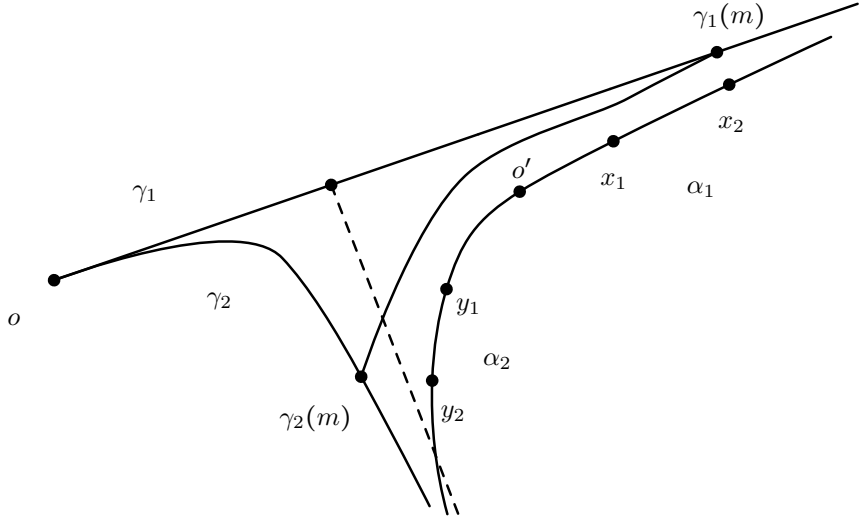

        \caption{When $ \diam\{ \pi_{\gamma_1}(\gamma_2) \} < \infty $ and $  \diam\{ \pi_{\gamma_2}(\gamma_1) \} < \infty$, the sequence of geodesics $[\gamma_1(m),\gamma_2(m)]$ converges to a bi-infinite geodesic $\alpha$.} 
        \label{injective_to_Martin}
    \end{figure}

    Without loss of generality, assume $d(o',x_i)\to \infty$ and $d(o',y_i)\to \infty$. Choose $M$ large enough such that 
    \[ \theta_0 d(o',x_M) \geq 2d(o',y_1) \geq R \bigand \theta_0 d(o',y_M) \geq 2d(o',x_1) \geq R. \]
    We now verify that $x_M$ is a $(\theta_0/2, R, \widetilde{D},\widetilde{L})$-good point on $\alpha$. Specifically, for any $z \in \alpha_2$, we show that $\|\mathrm{contr}([x_M,z]_{\alpha})\| \geq \frac{\theta_0}{2}\|[x_M,z]_{\alpha}\|$. As $x_M$ is a $(\theta_0, R, \widetilde{D},\widetilde{L})$-good points on $\alpha_1$, we get $\|\mathrm{contr}([o',x_M]_{\alpha})\| \geq {\theta_0}\|[o',x_M]_{\alpha}\|$. 
    
    If $z$ lies between $o'$ and $y_1$ on $\alpha$, then 
    \begin{align*}
        \|\mathrm{contr}([x_M,z]_{\alpha})\| &\geq \|\mathrm{contr}([x_M,o']_{\alpha})\| \geq {\theta_0}\|[x_M,o']_{\alpha}\| = \theta_0 d(o',x_M) \\
        &\geq \frac{\theta_0}{2}(d(o',x_M) + d(o',y_1)) \geq \frac{\theta_0}{2}\|[x_M,z]_{\alpha}\| .
    \end{align*}  
   
    If $y_1$ lies between $o'$ and $z$ on $\alpha$,
    then $\|\mathrm{contr}([y_1,z]_{\alpha})\|\geq \theta_0 \|[y_1,z]_{\alpha}\|$, since $y_1$ is a good point on $\alpha_2$. Thus,
    \begin{align*}
        \|\mathrm{contr}([x_M,z]_{\alpha})\| &\geq \|\mathrm{contr}([x_M,o']_{\alpha})\| + \|\mathrm{contr}([y_1,z]_{\alpha})\| \geq {\theta_0}\|[x_M,o']_{\alpha}\| + \theta_0 \|[y_1,z]_{\alpha}\| \\
        &= \theta_0(d(z,x_M) - d(o',y_1)) \geq \frac{\theta_0}{2}d(z,x_M) =\frac{\theta_0}{2}\|[x_M,z]_{\alpha}\|.
    \end{align*}
    The same argument applies to all $x_i,y_i$ with $i \geq M$, yielding good points tending to $\zeta_1$ and $\zeta_2$.

    Choose $g_i, h_i \in G$ so that $d(g_{i} o,x_i), d(h_{i} o,y_i) \leq \varepsilon$ for all $i  \geq 1 $. \cref{Ancona on geodesic with good points} gives the asymptotic relations for some constant $C$ and any $0<i<j$: 
    \[ \mathcal{G}(e,g_j) \asymp_C \mathcal{G}(e,g_i)\mathcal{G}(g_i,g_j) \bigand \mathcal{G}(g_i,h_{j}) \asymp_C \mathcal{G}(g_i,e)\mathcal{G}(e,h_{j}) .\]
    Then we obtain 
    \begin{align*}
        K(g_i, g_j) \asymp_C \mathcal{G}(e,g_i)^{-1} \bigand K(g_i, h_{j}) \asymp_C \mathcal{G}(g_i,e).
    \end{align*}
    Taking the limit as $j \to \infty$ and then $i \to \infty$, by \cref{TheGreenfunctiondecay}, we get
    \begin{align*}
       \lim_{i \rightarrow \infty} K_{\zeta_1}(g_i) = +\infty \bigand
       \lim_{i \rightarrow \infty} K_{\zeta_2}(g_i) = 0,
    \end{align*}
    which implies $\zeta_1 \neq \zeta_2$.
\end{proof}

\begin{lem}\label{the map is injective}
    The map $\Psi$ is injective up to the finite difference equivalence.
\end{lem}
\begin{proof}
    If two proportionally contracting rays $\gamma_1$ and $\gamma_2$ tend to the same Martin boundary point $\zeta_1=\zeta_2$, then  $\diam\{\pi_{\gamma_1}(\gamma_2)\} = \diam\{\pi_{\gamma_2}(\gamma_1)\} = \infty$ by \cref{bddproj imply visual}.
    Hence, $\gamma_1$ and $\gamma_2$ converge to the same $[\cdot]$-locus by \cref{frequently contracting to horofunction boundary}. 
\end{proof}

\subsection{Continuity of boundary map in direct limit topology}\label{sec continuity}
This subsection aims to establish the continuity of $\Psi$ with respect to a direct limit topology.

\begin{defn}\label{def prop contracting after moment}
    Let $\theta,D,L$ and $M$ be positive numbers. We say a geodesic ray $\gamma$ is \textit{$\theta$-proportionally $(D,L)$-contracting after the moment $M$}, if \[\|\mathrm{contr}_{(D,L)}(\gamma[0,k])\| \geq \theta \|\gamma[0,k]\| = \theta k ,\quad \forall k \geq M.\]

\end{defn}

The additional condition, compared to the proportionally contracting ray, is that the inequality holds for any moment $k$ greater than $M$, rather than for an undetermined  moment.
Denote the set of geodesic rays which are $\theta$-proportionally $(D,L)$-contracting after $M$ by $A(\theta, D, L, M)$.

\subsubsection*{Direct limit topology}
Let $\mathcal{L}_M$ be the subset of $\mathcal{L}$ consisting of elements $\xi$ such that there exists a geodesic ray in $A(\frac{\theta}{2}, D, L, M)$ converging to the locus $[\xi]$ with $\theta,D,L$ as in \cref{def from horofunction to Martin}.
It follows that \[\mathcal{L} = \bigcup_{M = 1}^{\infty} \mathcal{L}_M.\]
Let us endow $\mathcal{L}_M$ with subspace topology in the horofunction boundary $\partial_h X$. The natural inclusions $\mathcal{L}_M \hookrightarrow \mathcal{L}_{M'}$ are continuous for any $M<M'$, and thus form a direct limit system. We  then equip $\mathcal{L} = \bigcup_{M = 1}^{\infty} \mathcal{L}_M$ with the direct limit topology:
\[ \mathcal{L} = \lim_{\longrightarrow} \mathcal{L}_M. \]
We show that $\Psi$ is continuous with respect to the direct limit topology, which completes the proof of \cref{the map from horofunction to Martin}.

For every geodesic in $A(\theta,D,L,M)$, the distance between good points admits a uniformly exponential control as follows.

\begin{lem}\label{good pints on RCG after M}
    There exist positive constants $\theta', R$ and $ D'$ depending on $(\theta, D, L, M)$, such that for any $\gamma \in A(\theta, D, L, M)$, there exists an unbounded sequence of $(\theta',R,D',L)$-good points $x_i$ on $\gamma$. Moreover, 
    \[ M \left(\frac{8}{\theta}\right)^i \leq d(o,x_i) \leq M \left(\frac{8}{\theta}\right)^{i+1},\quad \forall i \in \mathbb{N}.\]
\end{lem}

\begin{proof}
    Let $\theta' = \theta / 8,R = 20L/\theta, D' = 10D$. If there are no $(\theta',R,D',L)$-good points in $\gamma[N_0,N_1]$ for some $N_1 > N_0 \geq M$, the proof in \cref{good points on proportionally geodesic ray} gives a positive $N > N_1$ such that (\ref{without good points}) holds:
    \[ \frac{\|\mathrm{contr}_{(D,L)}(\gamma[0,N])\|}{N} \leq 2 \theta' + \frac{10L}{R} + \frac{N_0}{N} \leq \frac{3}{4} \theta + \frac{N_0}{N}. \]
    Since $N \geq M$, by the assumption to $\gamma$, we have
    \[ \theta \leq \frac{\|\mathrm{contr}_{(D,L)}(\gamma[0,N])\|}{N} . \]
    Hence, we get $N_1 < N \leq \frac{4N_0}{\theta} $. Consequently, there must be a good point in $\gamma[N_0, \frac{8N_0}{\theta}]$ for any $N_0 \geq M$.
\end{proof}

The next two lemmas show that the map $\Psi$ is continuous when restricted to $\mathcal{L}_M$.

\begin{lem}\label{directlimitprojectionbig}
     Let $\gamma,\gamma_i \in A(\frac{\theta}{2}, D, L, M)$ converge to locus $[\xi],[\xi_i]$ respectively for $i \in \mathbb{N}$.
     If $\xi_i \to \xi, i \to \infty$,
     then the projection diameters diverge:
 \[\lim_{i \rightarrow \infty} \diam \{ \pi_{\gamma}(\gamma_i) \} = \infty . \]
\end{lem}

\begin{proof}
     Conversely, suppose a subsequence of $\diam \{ \pi_{\gamma}(\gamma_i) \}$ is bounded. Without loss of generality, assume that $\diam \{ \pi_{\gamma}(\gamma_i) \} \leq T$ for some $T>0$. There exists a $D$-contracting subsegment $p \subset \gamma$ with $\|p\| \geq L$ and $d(o,p) > T $ as shown in \cref{converging}. 
    Let the two endpoints of $p$ be $p_{-}$ and $p_{+}$, where $d(o,p_{-}) < d(o,p_{+})$.
    For any $N_0 \geq d(o,p_{+})$ and $N > 0$, $\pi_{\gamma}(\gamma_i(N)),p_-,p_+,\gamma(N_0)$ lie on $\gamma$ in this order. Therefore, the geodesic $[\gamma_i(N),\gamma(N_0)]$ passes the $3D$ neighborhood of $p_-$ by \cref{geodesic connects left and right}.
    \begin{figure}[ht]
        \centering
    \def\svgwidth{0.7\columnwidth}
    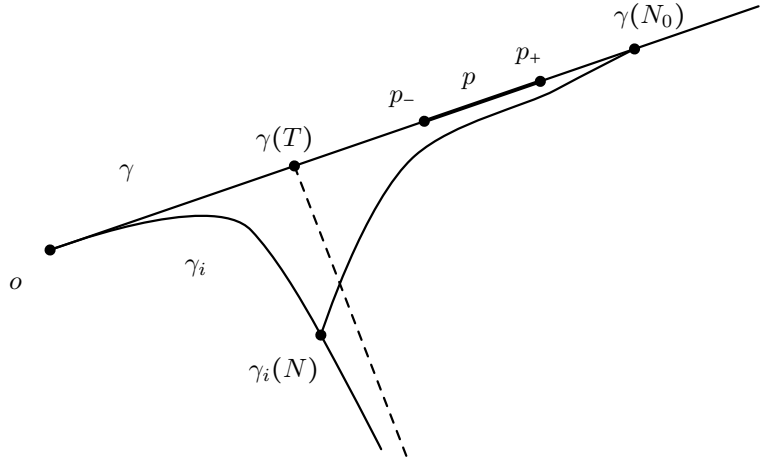

        \caption{The geodesic $[\gamma_i(N),\gamma(N_0)]$ passes a neighborhood of the contracting segment $p$. The bold part is a contracting segment.}
        \label{converging}
    \end{figure}
    
    Consider the Buseman function
        \begin{align*}
            d(\gamma(N_0), \gamma_i(N)) - d(o,\gamma_i(N)) &\geq d(\gamma(N_0),p_{-}) + d(p_{-},\gamma_i(N)) - 10D - d(o,\gamma_i(N))\\
            &\geq d(\gamma(N_0),p_{-})  - 10D +d(\gamma,\gamma_i(N)) - d(o,\gamma_i(N)) \\
            &\geq d(\gamma(N_0),p_{-})  - 10D - T.
        \end{align*}  
        Taking $N \to\infty$, we get $b_{\xi_i'}(\gamma(N_0)) \geq d(\gamma(N_0),p_{-})  - 10D - T$ for some $\xi_i' \in [\xi_i]$.
        Lemma 5.6 in \cite{yang2022conformal} states that, if $\eta$ and $\zeta$ are two horofunction boundary points in the same conical locus, then $\|b_{\eta}-b_{\zeta}\|_{\infty}$ is uniformly bounded by $20D$. It follows that $b_{\xi_i}(\gamma(N_0)) \geq d(\gamma(N_0),p_{-}) - 30D - T$ since $\xi_i' \in [\xi_i]$.
        
        Next, let $i \rightarrow \infty$. Because $\xi_i \rightarrow \xi$, we obtain 
        \begin{align}\label{directlimit topologyhorofunctionlarge}
            b_{\xi}(\gamma(N_0)) \geq d(\gamma(N_0),p_{-}) - 30D - T. 
        \end{align}      
        On the other hand, we have
        \[b_{\xi'}(\gamma(N_0)) = \lim_{t\to \infty}\big(d(\gamma(N_0),\gamma(N_0+t)) - d(o,\gamma(N_0+t))\big) = -N_0\]
        for some $\xi' \in [\xi]$. Then $b_{\xi}(\gamma(N_0)) \leq -N_0 + 20D$, which contradicts with \cref{directlimit topologyhorofunctionlarge}, since $N_0$ can be arbitrarily large.
\end{proof}

\begin{prop}\label{continuity for prop. contracting rays}
    If $\xi_i \to \xi$ for $\xi,\xi_i \in \mathcal{L}_M$, then $\Psi(\xi_i) \to \Psi(\xi)$.
\end{prop}

\begin{proof}
    There exist geodesic rays $\gamma, \gamma_i \in A(\frac{\theta}{2}, D, L, M)$ converge to $[\xi],[\xi_i]$ respectively by the definition of $\mathcal{L}_M$.
    Hence, we have
    $\lim_{i \rightarrow \infty} \diam \{ \pi_{\gamma}(\gamma_i) \} = \infty$ by \cref{directlimitprojectionbig}.

    Assume, for contradiction, that the proposition fails. Then there exists a subsequence of $\Psi(\xi_i)$ converging to a Martin boundary point $\eta \neq \Psi(\xi)$. Without loss of generality, assume that $\Psi(\xi_i)$ converges to $\eta$.

    Fix $\theta',R,D'$ as in \cref{good pints on RCG after M} for $A(\frac{\theta}{2}, D, L, M)$, and fix an arbitrary element $x \in G$.
    There exists a $(\theta',R,D',L)$-good point $h_1$ on $\gamma$ such that $d(o,h_1) > 3\big(d(o,xo) + \diam \{ o\cup\pi_{\gamma}(xo) \} +R\big)$, which implies $d(xo,\gamma) \leq \diam \{ h_1\cup\pi_{\gamma}(xo) \}$.
    Let $j$ be large enough such that $d(o,h_1) + 10(R + D + \varepsilon) \leq M (\frac{16}{\theta})^j$.
    There exists a subsegment $p \subset \gamma$ such that $p$ is $D$-contracting with $\|p\| \geq L$, and $d(o,p) > M (\frac{16}{\theta})^{j+1}  $. 
    Let the two endpoints of $p$ be $p_{-}$ and $p_{+}$ with $d(o,p_{-}) < d(o,p_{+})$.
    Also let $I$ be large enough such that $\diam \{\pi_{\gamma}(\gamma_I)\} \geq d(o,p_{+}) + 2L + R$ since $\lim_{i \rightarrow \infty} \diam \{ \pi_{\gamma}(\gamma_i) \} = \infty$.
    Then $\gamma_I$ intersects the balls $B(p_{-},3D)$ at a point $\widetilde{p_{-}}$ by \cref{geodesic connects left and right}. 

    \begin{figure}[ht]
        \centering
    \def\svgwidth{0.9\columnwidth}
    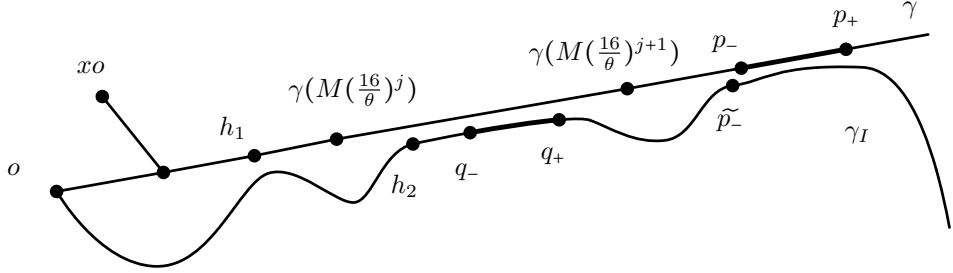

        \caption{The points $h_1$ and $h_2$ are two good points, and $h_2$ has a uniformly bounded distance from $\gamma$. The bold parts are contracting segments.}
        \label{continuity}
    \end{figure}

    By \cref{good pints on RCG after M}, there exists a $(\theta',R,D',L)$-good point $h_2$ on $\gamma_I$ such that $M (\frac{16}{\theta})^j \leq d(o,h_2) \leq M (\frac{16}{\theta})^{j+1}$. We claim that $d(h_2,\gamma)$ has a uniformly bound. 
    Since $h_2$ is a $(\theta',R,D',L)$-good point, there exists a subsegment $q \subset \gamma_I$, with two endpoints $q_{-}$ and $q_{+}$, such that $q$ is $D'$-contracting with $\|q\| \geq L$, and $d(h_2,q) \leq R$. 
    Due to $d(p_{-},\widetilde{p_{-}}) \leq 3D$, 
    the geodesic $[o,p_{-}]_{\gamma}$ intersects the ball $B(q_{-},3D)$ by \cref{geodesic connects left and right}. Hence, the distance $d(h_2,\gamma) \leq R+3D$ has a uniformly bound.

    Let $h_1',h_2',\gamma(N)'$ (for $  N >0$) be elements in $G$ such that
    \[d(h_1'o,h_1) \leq \varepsilon \bigand d(h_2'o,h_2) \leq \varepsilon \bigand d(\gamma(N)'o,\gamma(N)) \leq \varepsilon.\]
    Recall that $d(o,h_2) \geq M (\frac{16}{\theta})^j \geq d(o,h_1) + 10(R + D + \varepsilon) \geq d(o,h_1) + d(h_2,\gamma)$, so $h_2$ and $xo$ are $1$-antipodal along $(\gamma,h_1)$. Applying \cref{Ancona on geodesic with good points} yields a constant $C$ such that Ancona inequalities for
    the triples $(x,h_1',h_2')$ and $(e,h_1',h_2')$ hold:
    \begin{align}\label{h_2' to h_1'}
        K_{h_2'}(x) = \mathcal{G}(x,h_2')/\mathcal{G}(e,h_2') \asymp_C \big(\mathcal{G}(x,h_1')\mathcal{G}(h_1',h_2')\big)/\big(\mathcal{G}(e,h_1')\mathcal{G}(h_1',h_2')\big) = K_{h_1'}(x).
    \end{align}

    Next we claim $K_{\Psi(\xi)}(x) \asymp_C K_{h_1'}(x)$ and $K_{\Psi(\xi_I)}(x) \asymp_C K_{h_2'}(x)$. Indeed, for large $N$, applying \cref{Ancona on geodesic with good points}, the Ancona inequalities for
    the triples $(x,h_1',\gamma(N)')$ and $(e,h_1',\gamma(N)')$ show:
    \begin{align*}
        K_{\gamma(N)'}(x) = \mathcal{G}(x,\gamma(N)')/\mathcal{G}(e,\gamma(N)') \asymp_C \big(\mathcal{G}(x,h_1')\mathcal{G}(h_1',\gamma(N)')\big)/\big(\mathcal{G}(e,h_1')\mathcal{G}(h_1',\gamma(N)')\big) = K_{h_1'}(x).
    \end{align*}
    Taking $N \to \infty$ yields $K_{\Psi(\xi)}(x) \asymp_C K_{h_1'}(x)$. We have $K_{\Psi(\xi_I)}(x) \asymp_C K_{h_2'}(x)$ similarly, since $h_2$ is a good point along $\gamma_I$. Combining this with \cref{h_2' to h_1'}, we obtain
    \[ K_{\Psi(\xi_I)}(x) \asymp_C K_{h_2'}(x) \asymp_C K_{h_1'}(x)\asymp_C K_{\Psi(\xi)}(x). \]
    
    Finally, taking the limit as $I \rightarrow \infty$, we have 
    \[ K_{\eta}(x) \asymp_C K_{\Psi(\xi)}(x), \forall x \in G,\]
    which contradicts to the fact that $\Psi(\xi)$ is minimal and $\eta \neq \xi$. 
\end{proof}

\section{Ancona inequality (IV) in \texorpdfstring{$\mathrm{CAT(0)}$}{CAT(0)} cube complex with Morse hyperplanes}\label{sec Ancona in CAT(0)}

In this and next sections, assume that a non-elementary group $G$ acts geometrically and essentially on an irreducible $\mathrm{CAT(0)}$ cube complex $X$. 
Assuming further that $X$ contains at least one \textit{Morse hyperplane}, we shall derive a more powerful Ancona inequality in \cref{SS-Ancona-CAT-Morse-hyperplane} (than that in \cref{Sec proportionally contracting geodesic}).

Most of the first two subsections are expository, and could be ignored on a first reading. The main new results are contained in \cref{Contracting isometry minimal fixed points} and \cref{SS-Ancona-CAT-Morse-hyperplane}.

\subsection{CAT(0) cube complex and Roller boundary}
A   cube complex $X$ is a CW complex built out of unit Euclidean $n$-cubes $[0,1]^n$ with various $n\ge 1$ by side isometric gluing. By Gromov's criterion, $X$ is a CAT(0) space if and only if the link of each vertex is a flag simplicial complex.  The maximal dimension of $n$-cubes is called the dimension of $X$.   Unless otherwise mentioned, we only consider finite dimensional and locally compact CAT(0) cube complexes, that is, $X$ is a proper metric space. 

Mid-cubes in $[0,1]^n$ are the subsets with certain coordinate at $1/2$ (e.g. $[0,1]^{n-1}\times \{1/2\}$). 
We say two parallel edges of a square are related.
This relation generates an equivalence relation on the set of edges.
A \textit{hyperplane} or a \textit{wall} is the union of the
midcubes transverse to the edges belonging to an equivalence class.
The collection $\mathcal H$ of {hyperplanes} or  {walls}  endows $X$ with a wall structure and, conversely, $X$ could be recovered as the cubulation of these walls by Sageev-Roller duality. We refer to \cite{Sageev95}  for further details and here we only make use of the following basic result.
\begin{thm}\label{wall and hyperplane}\cite{Sageev95}
    Let  $ \Sigma $ be a $\CAT$ cube complex.
    \begin{enumerate}
        \item For every vertices $x$ and $y$ in $\Sigma$, the (combinatorial) distance between them is the number of hyperplanes that separate $x$ and $y$.
        \item 
        If $J$ is a hyperplane, then $J$ does not self-intersect and the complement $\Sigma \setminus J$ has exactly two connected components.
    \end{enumerate}
    These components, called the \textit{half-spaces} bounded by $J$ and denoted $J^+$ and $J^-$, are convex sub-complexes of $\Sigma$.
\end{thm}

Two hyperplanes $J_1$ and $J_2$ are \textit{transverse} if all the four intersections $J_1^{\pm} \cap J_2^{\pm}$ are  nonempty. We say that $J_1,J_2$ are \textit{strongly separated} if no hyperplane transverses both of them and $J_1\cap J_2=\emptyset$. Two half-spaces $J_1^+,J_2^+$ are \textit{strongly separated} if the hyperplanes $J_1,J_2$ are strongly separated.\\

\paragraph{\textbf{Roller boundary}} An \textit{orientation} of walls picks up exactly one half-space for each wall, and is called \textit{consistent} if any two such chosen half-spaces intersect.  We equip the space $\mathcal U:=2^{\mathcal H}$ of all orientations with the  product topology, where all consistent orientations form a closed subset. 
For every vertex $x$ of $X^0$, the set $U_x$ of half-spaces containing $x$ forms a consistent orientation in $2^{\mathcal H}$. The \textit{Roller compactification} denoted by $\bU$ of $X^0$ is thus the closure  of  the vertex set   $X^0$ (understood as  consistent orientations) in $\mathcal U$, and \textit{Roller boundary} is $\bR:=\bU\setminus X^0$. 

A point $\xi\in \bR$ is the limit of some sequence $x_n$ of vertices of $X$, and by definition, $U_\xi$ is the pointwise limit of $U_{x_n}$. Abusing language, we say that $\xi$ is in the half-space $J^+$ if $J^+\in U_{\xi}$. In this way we have a partition $\bU=J^-\cup J^+$ for each hyperplane $J$.

The CAT(0) metric induces an $\ell^2$-metric on the 1-skeleton $X^1$. In this paper, we are mainly interested in the $\ell^1$-metric on $X^0$: $$\forall x, y\in X^0: \; d(x,y):=\frac{1}{2}|U_x\Delta U_y|$$ where $\Delta$ denotes the  symmetric difference of sets.  If $X$ is finite dimensional, the $\ell^1$-metric and $\ell^2$-metric are quasi-isometric. The $\ell^1$-metric is exactly the combinatorial metric on $X^1$ and usually non-unique geodesic  metric.  However, the set of walls dual to the edges on any geodesic is more canonical and does not depend on the geodesic, which is exactly the set of walls separating $x$ and $y$. By abuse of language, we also say that the wall separating $x$ and $y$ is dual to the geodesic $[x,y]$.

We say that two boundary points $\xi,\eta \in \partial_\mathcal{R} X$ have \textit{finite symmetric difference} if the symmetric difference of their associated consistent orientations $U_{\xi} \Delta U_{\eta}$ is finite. 
According to \cite[Proposition 6.20]{fernos_random_2018}, by an unpublished result of Bader and Guralnik, the Roller compactification is homeomorphic to the horofunction compactification of $X^0$. 
Moreover, \cite[Lemma 11.5]{yang2022conformal} shows that these two natural relations coincide, as stated below.

\begin{prop}\label{Roller to horofunction}
    There exists a canonical  homeomorphism 
    $$ \Phi : X^0 \cup \partial_\mathcal{R} X \to X^0 \cup \partial_h X^0 $$ that restricts to the identity on $X^0$. 
    Furthermore, the finite symmetric difference relation on $\partial_\mathcal{R} X$ corresponds precisely to the finite difference relation on $\partial_h X^0$.
\end{prop}

\begin{rem}
Let $\partial_h X^0 \subset \partial_h X$ denote the subset of horofunction boundary points that are limits of vertex sequences.
Although $\partial_h X^0 \subsetneq \partial_h X$ may hold, the $1$-neighborhood density of $X^0$ ensures that for any $\eta \in \partial_h X$, there exists $\xi \in \partial_h X^0$ with $\|h_\eta - h_\xi\|_\infty \leq 1$. 
Thus, $\partial_h X^0$ is equivalent to $\partial_h X$ with respect to the finite difference equivalence. 
\end{rem}

We also use the symbol $[\xi]$ to denote the finite symmetric difference class of Roller boundary point $\xi \in \partial_\mathcal{R} X$.

Recall that two half-spaces $J_1^+,J_2^+$ are strongly separated if the hyperplanes $J_1,J_2$ are strongly separated. The notion   of regular boundary points was introduced by Fernos \cite{FERNÓS_2018}.  
\begin{defn}\label{regularpointsDefn}
A  boundary point $\xi \in \partial_{\mathcal R} X$ is called \textit{regular} if the consistent orientation of walls associated to $\xi\in 2^{\mathcal H}$ contains an infinite descending chain of pairwise strongly separated half-spaces $\{J_i^+\}_{i \in \mathbb{N}}$. 
\end{defn}
By \cite[Corollary 7.5]{FERNÓS_2018}, the intersection of an infinite descending chain of pairwise strongly separated half-spaces is a singleton. Hence, for any regular point $\xi$, we have  $\{\xi\} = \bigcap_{i \in \mathbb{N}} J_i^+$ for any infinite descending chain $\{J_i^+\}_{i \in \mathbb{N}}$ in Definition \ref{regularpointsDefn}. Consequently,  the equivalence class $[\xi]$ is trivial: any consistent orientation having finite  symmetric difference with $\xi$ is exactly  $\xi$ itself. 

\subsection{Essential  actions and contracting isometries}
In their work on rank rigidity theorem \cite{Caprace_2011}, Caprace-Sageev  demonstrated that the irreducibility of a CAT(0) cube complex $X$ is intimately related to the existence of contracting isometry on $X$. We   briefly recall the related notions and  results.

We say that an (not necessarily proper) action of $G$ on the CAT(0) cubical complex $X$ is \textit{essential} if no $G$-orbit is contained in a finite neighborhood of a half-space. 
A hyperplane $J$ in $X$ is called \emph{essential}
if $X$ is not contained in a finite neighborhood of a half-space bounded by $J$. 
When the action is cocompact, the action is essential if and only if every hyperplane in $X$ is essential.
If $G$ has no fixed points in the visual boundary of $X$ or has only finitely many orbits of walls, then the action could be made essential by passing to its (convex) essential core (see \cite[Prop 3.5]{Caprace_2011}).
We say that $X$ is \textit{irreducible} if it cannot be written as a nontrivial metric product of two $\CAT$ cube complexes (i.e. no factor is singleton). 

The \textit{barycentric subdivision} of the Euclidean cube $[-1,1]^n$ partitions it into $2^n$ smaller cubes of the form $\Pi_{i=1}^n [0,\pm 1]$. 
For any $\CAT$ cube complex $X$, its \textit{cubical subdivision} $X'$ is obtained by applying this barycentric subdivision to every cube in $X$. 
According to \cite[Remark 6.4]{haglund_isometries_2023}, isometries $\mathrm{Aut}(X) \subset \mathrm{Aut}(X')$ fall into two classes: those fixing at least one vertex in $X'$ (called elliptic) and those preserving a combinatorial geodesic in $X'$ by translation (called hyperbolic).
Thus after passing to cubical subdivision, we may assume every isometry of $G$ either fixes a point or preserves a geodesic.

For any hyperbolic element $g$, a geodesic preserved by $g$ is called an \textit{axis} of $g$.
We say an element $g$ \textit{skewers} a half-space $J^+$ (or the hyperplane $J$) if ${g}^n J^+ \subsetneq J^+$ for some $n > 0$.

\begin{lem}\label{skewer and intersect}
    For a hyperbolic element $g$ and a hyperplane $J$, the following conditions are equivalent:
    \begin{enumerate}
        \item $g$ skewers the hyperplane $J$;
        \item The axis of $g$ intersects $J$.
    \end{enumerate}
\end{lem}

\begin{rem}
While this lemma resembles \cite[Lemma 2.3]{Caprace_2011} in spirit, we emphasize a key distinction and the proof methodology nevertheless transfers directly.
The cited result concerns axes in the CAT(0) metric setting, whereas our current formulation operates purely in the combinatorial regime.
Notice that any combinatorial geodesic intersects any hyperplane at most once and then follow the same proof.
\end{rem}

We say a hyperbolic element $g$ is contracting if any axis of it is a contracting geodesic.
\begin{thm} \label{rank rigidity}
Let a non-elementary group $G$ act geometrically and essentially on an irreducible $\CAT$ cube complex $X$. Then:
\begin{enumerate}
    \item 
    An isometry is contracting in the $\ell^2$-metric if and only if it is contracting in the $\ell^1$-metric.
    \item  
    Every hyperplane is skewered by a contracting isometry.
    \item 
    $X$ contains a pair of {strongly separated} half-spaces.
\end{enumerate} 
\end{thm}

\begin{proof}
The equivalence of the assertion (1) is proved in \cite[Lemma 8.3]{gekhtman2018counting} using \cite{Caprace_2011} and \cite{genevois_contracting_2020}.
By \cite[Theorem 6.3]{Caprace_2011}, every hyperplane is skewered by a contracting isometry in the $\ell^2$-metric. By \cref{non-elementary}, $G$ contains infinitely many pairwise independent contracting elements, which have disjoint fixed points by \cite[Lemma 11.5]{yang2022conformal}, hence $G$ has no fixed points in the visual boundary.
The assertion (3) then follows from \cite[Prop. 5.1]{Caprace_2011}.  
\end{proof}

\subsection{Contracting isometry with minimal fixed points}\label{Contracting isometry minimal fixed points}

Let $g$ be a contracting isometry on $X$ in either metric by \cref{rank rigidity}. Its axis is a contracting bi-infinite geodesic ending at two distinct  points in the Roller boundary $\partial_{\mathcal R}X$.
More general, the $[\cdot]$-class of accumulations points of $g^n o$ in $\partial_{\mathcal R}X$  with $n\to \infty$ is called \textit{attracting point} denoted by $[g^+]$. Similarly, the \textit{repelling point} denoted by $[g^-]$ is the $[\cdot]$-class of accumulations points of $g^{-n} o$ in $\partial_{\mathcal R}X$ with $n\to \infty$. 

We now characterize when a contracting isometry has regular fixed points. We remark that even in hyperbolic cubical groups, contracting elements may not have regular fixed points in the Roller boundary.

\begin{lem}\label{regular fixed points}
Let $g \in G$ be a hyperbolic element. Then the following are equivalent:
\begin{enumerate}
    \item $g$ is a contracting isometry with regular fixed points;
    \item $g$ skewers three pairwise strongly separated hyperplanes;
    \item The axis of $g$ intersects three pairwise strongly separated hyperplanes;
    \item The axis of $g$ intersects infinitely many pairwise strongly separated hyperplanes.
\end{enumerate}
\end{lem}

\begin{proof}
    The equivalence of $(2)$ and $(3)$ is due to \cref{skewer and intersect}. The direction $(1)\Rightarrow (4)$ follows immediately the definition of regular points as limits of strongly separated half-space sequences. The direction $(4) \Rightarrow (3)$ is trivial.
    
    $(3) \Rightarrow (1)$: Let $(J_1, J_2, J_3)$ be a triple of pairwise strongly separated hyperplanes intersecting the axis $\ell_g$ of $g$. For sufficiently large $n$, the hyperplanes $(J_4, J_5, J_6) = g^n(J_1, J_2, J_3)$ satisfy the intersection points $J_i \cap \ell_g $ appear along $\ell_g$ in the order corresponding to the indices $1 \leq i \leq 6$. 

    If $J_3$ and $J_4$ are disjoint, then $J_3$ separates $J_2$ from $J_5$. Since $J_2$ and $J_3$ are strongly separated, this implies $J_2$ and $J_5$ are strongly separated.

    If $J_3$ intersects $J_4$, then $J_3$ must be disjoint from $J_5$ because $J_4$ and $J_5$ are strongly separated. Consequently, $J_3$ again separates $J_2$ from $J_5$, establishing their strong separation.

    In both cases, we obtain $J_2$ and $J_5 = g^nJ_2$ are strongly separated. Hence $\ell_g$ intersects infinitely many pairwise strongly separated hyperplanes $\{g^{in}J_2\}_{i \in \mathbb{Z}}$. Then $g$ is contracting by \cite[Lemma 6.2]{Caprace_2011}, which states that every isometry skewering a pair of strongly separated hyperplanes is contracting.
\end{proof}

\begin{rem}
     Condition $(2)$ in the lemma cannot be weakened to merely requiring two strongly separated hyperplanes, as this would invalidate conclusion $(1)$. See   \cite[Example 3.10]{CFI} for counter-examples.
\end{rem}

\begin{cor}\label{many elements fix regular points}
Assume that a non-elementary group $G$ acts  essentially and properly on an irreducible CAT(0) cube complex $X$. Then $G$ contains infinitely many pairwise independent contracting elements with regular fixed points.      
\end{cor}
\begin{proof}
Since $X$ is irreducible, \Cref{rank rigidity} ensures the existence of a pair of strongly separated half-spaces $J_1^+ \supset J_2^+$. 
By \cite[Double Skewer Lemma, Lemma 6.2]{Caprace_2011}, there is a contracting element $\mathfrak{t}\in G$ such that $J_2^+ \supsetneq \mathfrak{t} J_1^+$.
Since $J_2$ and $\mathfrak{t} J_2$ are strongly separated (as they are separated by $\mathfrak{t} J_1$), the axis of $\mathfrak{t}$ intersects infinitely many pairwise strongly separated hyperplanes $\{\mathfrak{t}^i J_2\}_{i \in \mathbb{Z}}$, hence $\mathfrak{t}$ has regular fixed points.  

To construct infinitely many ones, let us fix another contracting isometry $k$ independent with $\mathfrak{t}$. By \cite[Lemma 3.13]{yang2022conformal}, $\mathfrak{t}^nk^n$ are contracting isometries for all $n\gg 0$: by construction, $g_n:=\mathfrak{t}^nk^n$ preserves the following path by translation
$$\gamma_n:=\bigcup_{i\in \mathbb Z} (g_n)^i \big( [o,\mathfrak{t}^no]\ [\mathfrak{t}^no,\mathfrak{t}^nk^no]\big)$$
which is proved to be a contracting quasi-geodesic with uniform contracting constant. As the endpoints of $\gamma_n$ converge to $\mathfrak{t}^+$ and $k^-$ repressively as $n\to \infty$, $\{g_n: n\gg 0\}$ must be pairwise independent by passing to a subsequence.
Next we prove that $g_n$ skewers three pairwise strongly separated hyperplanes to finish the proof, which is claimed bellow.
\end{proof}
\begin{lem}\label{extension skewer}
Let $\mathfrak{t}$ be a contracting element that skewers three pairwise strongly separated hyperplanes. Let $k$ be another contracting isometry independent of $\mathfrak{t}$. Then for all sufficiently large $n$, the element $g_n:=\mathfrak{t}^nk^n$ skewers three pairwise strongly separated hyperplanes. 
\end{lem}
\begin{proof}
By \cref{skewer and intersect}, a hyperplane $K$ is skewered by a hyperbolic isometry $h$ if and only if the axis of $h$ meets both half-spaces $K^-$ and $K^+$.
Consider the quasi-geodesic $\gamma_n$ with uniform (independent of $n$) contracting constant. 
Because $\mathfrak{t}$ skewers three pairwise strongly separated half-spaces, for sufficiently large $n$, a triple of hyperplanes $J_1^{+} \supset J_2^{+} \supset J_3^{+}$ intersects the geodesic segment $[o, \mathfrak{t}^n o]$, and consequently intersects $\gamma_n$. Furthermore, points $o$ and $\mathfrak{t}^no$ can be arbitrarily deep into the two half-spaces $t_nJ_1^{-}$ and $t_nJ_3^{+}$ when $n$ is sufficiently large.

Since $g_n$ preserves $\gamma_n$ by translation, the axis of $g_n$ is contained in a uniform $M$-neighborhood of $\gamma_n$ (independent of $n$).
In particular, the axis $\ell_{g_n}$ intersects the two balls $B(o,M)$ and $B(\mathfrak{t}^no,M)$. Choose $n$ large enough such that the two balls are contained in the two half-spaces $t_nJ_1^{-}$ and $t_nJ_3^{+}$ respectively. 
Thus $\ell_{g_n}$ meets both half-spaces when $n$ is large, which shows that $g_n$ skewers three pairwise strongly separated hyperplanes.
\end{proof}

\subsection{Ancona inequality with Morse hyperplanes}\label{SS-Ancona-CAT-Morse-hyperplane}
In this subsection, assume that $X$ contains a Morse hyperplane. 
The goal of this subsection is to derive a more powerful Ancona inequality than that in \cref{Sec embedding}. 

Note that the existence of Morse hyperplane may be stronger than the existence of Morse half spaces.   
\begin{lem}
A hyperplane $J$ is Morse if and only if the associated two half spaces $J^-,J^+$ are Morse.    
\end{lem}
\begin{proof}
The proof is straightforward by \cref{wall and hyperplane} and left to the interested reader.    
\end{proof}

Once we have one Morse hyperplane, there will be infinitely many Morse hyperplanes which are skewered by pairwise independent contracting elements.

\begin{lem}\label{many contracting elements skewer}
If $X$ contains a Morse hyperplane, then there exist infinitely many pairwise independent contracting elements $\{{g}_n: n\ge 1\}$ such that each of which skewers a Morse hyperplane.
Moreover, we may require that each $g_n$ has regular fixed points.
\end{lem}

\begin{proof}
Given a Morse hyperplane, there exists a contracting element skewering it by \cref{rank rigidity}.
The construction in \cref{many elements fix regular points} yields infinitely many pairwise independent contracting elements skewering a Morse hyperplane when we replace the skewering a triple of pairwise strongly separated hyperplanes condition with the requirement of skewering a Morse hyperplane in \cref{extension skewer}.

For the ``moreover" part, it suffices to choose the contracting element $k$ in the proof of \cref{many elements fix regular points} to be the one that skewers a Morse hyperplane. Then, for all sufficiently large $n$, the elements $g_n = \mathfrak{t}^n k^n$ will both skewer a Morse hyperplane and have regular fixed points.
\end{proof}

The following two results give the Ancona inequality in two different settings.
Recall the constant $\varepsilon$ defined in \cref{cstvarepsilon} associated to the geometric action $G \act X$.

\begin{lem}\cite[Section 2.G]{CFI}\label{bridge between strongly separated hyperplanes}
Let $J_1^-$ and $J_2^+$ be two disjoint strongly separated half spaces. Then there exists a unique pair $p_1 \in J_1^-$ and $p_2 \in J_2^+$ minimizing the distance of $J_1^-$ and $J_2^+$, i.e., $d(p_1,p_2) \le d(x,y)$ for any $x\in J_1^-,y\in J_2^+$. 
In particular, for any $x_1 \in J_1^-$ and $x_2 \in J_2^+$ and for any choice of a geodesic $[p_1,p_2]$, the path $[x_1,p_1][p_1,p_2][p_2,x_2]$ is a geodesic.   
\end{lem} 

According to \cite{fernos_random_2018}, the interval $I(p_1,p_2) := \{z \in X:d(p_1,p_2) = d(p_1 ,z) + d(z,p_2) \}$ is called  the \textit{bridge}  between $J_1$ and $J_2$ denoted by $B(J_1,J_2)$  and $l=d(p_1,p_2)$ the \textit{length} of bridge.

\begin{thm}
Let $J_1, J_2$ be two strongly separated Morse hyperplanes, which bound disjoint half-spaces $J_1^-$ and $J_2^+$ respectively. Then there exists a constant $C$ depending on the length of bridge $B(J_1,J_2)$ and the Morse gauge of $J_1,J_2$ with  the following property. 
Let $x,y,z\in G$ so that $xo\in J_1^-,yo\in J_2^+$ and $zo$ be a point $\varepsilon$-within the bridge $B(J_1,J_2)$. Then 
    \[ C^{-1} \mathcal{G}(x,z)\mathcal{G}(z,y) \leq \mathcal{G}(x,y) \leq C \mathcal{G}(x,z)\mathcal{G}(z,y).\]    
\end{thm}
\begin{proof}
We shall use the Ancona inequality along a Morse subset with narrow points as \cref{Ancona on Morse subset}. Note that $B(J_1,J_2)$ is the union of all geodesics from $p_1$ to $p_2$. By assumption, choose a geodesic $[p_1,p_2]$ so that $d(zo,[p_1,p_2])\le \varepsilon$. Let $Y_1=J_1^-$ and $Y_2=[p_1,p_2]\cup J_2^{+}$. As $[p_1,p_2]$ is a finite segment, $Y_1$ and $Y_2$ are Morse. By \cref{bridge between strongly separated hyperplanes}, any point on $[p_1,p_2]$ is narrow point for $J_1^-$  and $J_2^+\cup [p_1,p_2]$. Hence the conclusion follows from \cref{Ancona on Morse subset}.
\end{proof}

\begin{thm}\label{Morse hyperplanes and contracting geodesic}
Let $J_1, J_2$ be two Morse hyperplanes, which bound disjoint half-spaces $J_1^-$ and $J_2^+$ respectively. Let $\gamma$ be a $D$-contracting geodesic segment with two endpoints $\gamma_{-} \in J_1, \gamma_{+} \in J_2$ so that $\pi_\gamma(J_1)$ has at least $10D$ distance to $\pi_\gamma(J_1)$. Then there exists $C$ with the following property. Let $xo\in J_1^-,yo\in J_2^+$ and $zo$ be a point $\varepsilon$-within $\gamma\setminus (J_1^-\cup J_2^+)$ for $x,y,z\in G$. Then 
    \[ C^{-1} \mathcal{G}(x,z)\mathcal{G}(z,y) \leq \mathcal{G}(x,y) \leq C \mathcal{G}(x,z)\mathcal{G}(z,y).\]    
\end{thm}

\begin{proof}
We apply \cref{Ancona on Morse subset} to the union $Y = J_1^- \cup \gamma \cup J_2^+$. By \cref{Morse subset union}, $Y$ forms a Morse subset. 
Any point $zo \in N_\varepsilon(\gamma \setminus (J_1^- \cup J_2^+))$ decomposes $Y$ into two components $Y_1$ and $Y_2$ naturally (up to $\varepsilon$-error). The $D$-contracting property (\cref{contracting properties}) ensures that any geodesic connecting $Y_1$ with $Y_2$ must intersect the $C'$-neighborhood of $zo$, where $C' = C'(D,\varepsilon)$. This establishes the narrowness of $zo$. Hence the conclusion follows from \cref{Ancona on Morse subset}.
\end{proof}

In the remainder of this subsection, we deduce another version of Ancona inequality from \cref{Morse hyperplanes and contracting geodesic} where  two Morse hyperplanes $J_1, J_2$ are provided by a special contracting element $\mathfrak{g}$. This version is particularly tailed to deal with the admissible $(r,F)$-rays ending at conical points. \\

\paragraph{\textbf{Standing assumption}}
From now on, let us fix a Morse hyperplane $J$ and $\mathfrak{g}$ be a contracting isometry skewering the hyperplane $J$, i.e. $\mathfrak{g} J^+ \subsetneq J^+$. The existence of $\mathfrak{g}$ is guaranteed by \cref{rank rigidity}. Let $\mathfrak l$ be an combinatorial axis of $\mathfrak{g}$ on which $\mathfrak{g}$ acts by translation. 

\begin{lem}\label{hyperplane projection small}
    The projection $\pi_{\mathfrak l}(J)$ has bounded diameter.
\end{lem}

\begin{proof}
    Since $\mathfrak{g}$ skewers the hyperplane $J$, the axis $\mathfrak l$ intersects $J$ in exactly one point $t$ by \cref{skewer and intersect}.
    If $x$ is an arbitrary point on $J$,  it is well-known that the concatenated path $[x,\pi_{\mathfrak l}(x)]  [\pi_{\mathfrak l}(x),t]_{\mathfrak{l}}$ is a $3$-quasi-geodesic (see the proof in \cref{Morse subset union}). Since $J$ is Morse, the path is contained in a fixed $m>0$ neighborhood of $J$. So the projection $\pi_{\mathfrak l}(J)$ is contained in $N_m(J) \cap \mathfrak l$.

    Next we prove $N_m(J) \cap \mathfrak l$ has finite diameter for any $m>0$. Since $J$ is skewed by $\mathfrak{g}$, these translations $\{\mathfrak{g}^i J\}_{i \in \mathbb{Z}}$ are disjoint. Hence the $m$ neighborhood of $J$ is between the two hyperplanes $ \mathfrak{g}^{-(m+1)} J$ and $\mathfrak{g}^{m+1} J$. On the other hand, the axis $\mathfrak l$ meets each $ \mathfrak{g}^{-(m+1)} J$ and $\mathfrak{g}^{m+1} J$ in a single point. So the intersection $N_m(J) \cap \mathfrak l$ is a finite segment of $\mathfrak l$, which has bounded diameter.
\end{proof}

\begin{thm}\label{Ancona on geodesic in CAT(0)}
    For any ${r} > \varepsilon$, there exist constants $C,N$ satisfying the following property.
    For any $M \geq N$, if $x,z,y\in G$ are three elements such that $z$ is an $({r},\mathfrak{g}^M)$-barrier of the geodesic $[xo,yo]$, then 
    \[ C^{-1} \mathcal{G}(x,z)\mathcal{G}(z,y) \leq \mathcal{G}(x,y) \leq C \mathcal{G}(x,z)\mathcal{G}(z,y).\]
\end{thm}

\begin{proof}
    Let $J$ be the Morse hyperplane in $X$ and $J^+$ be a half-space bounded by $J$.
    Assume the $\mathfrak{g}$-axis $\mathfrak l$ is $D$-contracting and passing through $o$.
    Since $\pi_{\mathfrak l}(J)$ has bounded diameter by \cref{hyperplane projection small}, we choose $N \geq N_0$ large enough and assume $\pi_{\mathfrak{l}}(J) \subset\mathfrak{l}_{[o,\mathfrak{g}^No]}$ as shown in \cref{connecting_Morse_hyperplane}.
    We enlarge $N$ such that the balls $B(o,{r})$ and $B(\mathfrak{g}^No,{r})$ are separated by $J$.
    Without loss of generality, we assume $z = e$. 

     \begin{figure}[ht]
        \centering
    \def\svgwidth{0.9\columnwidth}
    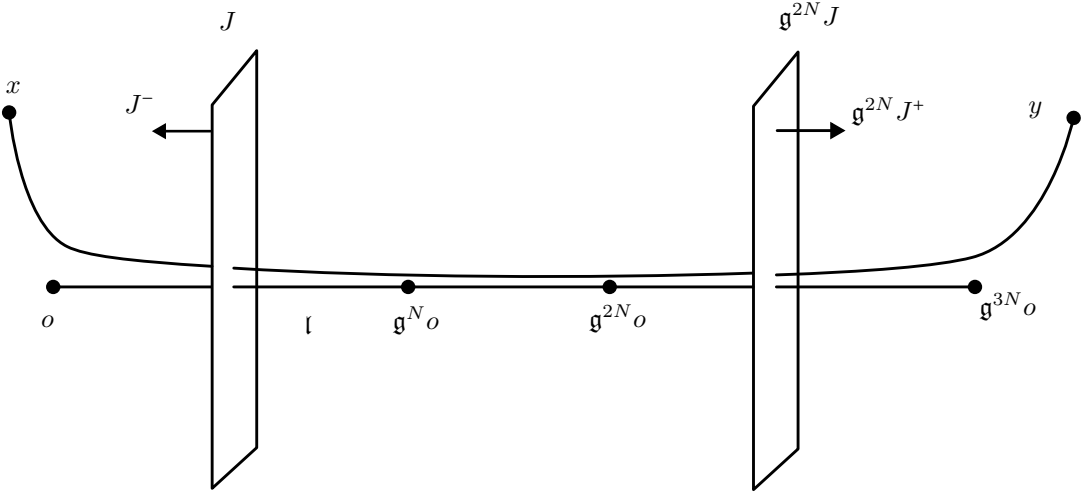

        \caption{The set $Y=J^- \cup\mathfrak{l} \cup \mathfrak{g}^{2N}J^+$ is Morse and narrow at $o$.}
        \label{connecting_Morse_hyperplane}
    \end{figure}

    If $[x,y]$ intersect the $r$ neighborhood of $o$ and $\mathfrak{g}^{3N}o$, by the convexity of half-spaces, we have $x \in J^-$ and $y \in \mathfrak{g}^{2 N}J^+$. 
    Enlarge $N$ such that $\|\mathfrak{g}^N\| = d(\mathfrak{g}^No,\mathfrak{g}^{2N}o) \geq 10 D$, thus $d(\pi_{\mathfrak{l}}(J), \pi_{\mathfrak{l}}(\mathfrak{g}^{2N}J)) \geq 10D$.
    Applying \cref{Morse hyperplanes and contracting geodesic}, there exists a constant $C$ such that 
    \[ C^{-1} \mathcal{G}(x,e)\mathcal{G}(e,y) \leq \mathcal{G}(x,y) \leq C \mathcal{G}(x,e)\mathcal{G}(e,y). \]
    The proof is completed by substituting ${3N}$ by ${N}$.
\end{proof}

Previously, we established the Ancona inequalities under various settings involving antipodal points, but now we present a significantly stronger version that only requires the position of the projection instead of the antipodality.

\begin{cor}\label{Ancona in CAT(0)}
     For any ${r} \geq \varepsilon$, there exist constants $C,N$ satisfying the following property.
     Assume that $z \in G$ is an $({r},\mathfrak{g}^N)$-barrier of a geodesic $\gamma$ and $z'$ be a point on $\gamma$ with $d(z',zo) \leq {r}$.
     If there are two orbit points $xo,yo$ such that $\pi_{\gamma}(xo)$ and $\pi_{\gamma}(yo)$ are on different sides of $z'$ along $\gamma$, and $\mathrm{min} \{d(\pi_{\gamma}(xo),z'),d(\pi_{\gamma}(yo),z')\} \geq \|\mathfrak{g}^N\|+2{r}$, then 
    \[ C^{-1} \mathcal{G}(x,z)\mathcal{G}(z,y) \leq \mathcal{G}(x,y) \leq C \mathcal{G}(x,z)\mathcal{G}(z,y).\]
\end{cor}

 \begin{figure}[ht]
        \centering
    \def\svgwidth{0.8\columnwidth}
    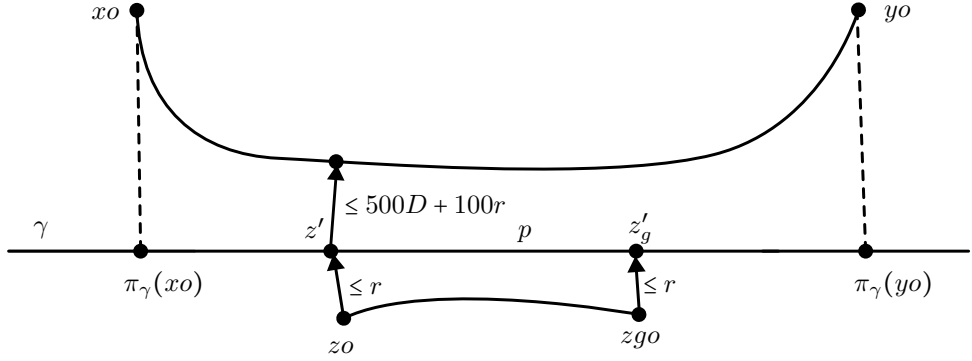

        \caption{Stability of Barriers}
        \label{fig Ancona_in_CAT0}
    \end{figure}

\begin{proof}
    Let $g = \mathfrak{g}^N$ and $z_g'$ be a point on $\gamma$ such that $d(z_g',z{g}o) \leq {r}$ as shown in \cref{fig Ancona_in_CAT0}. Since the geodesic $[zo,z{g}o]$ is $D$-contracting, we obtain ${[z',z_g']}_\gamma$ is $(86D+24{r})$-contracting by \cref{geodesic near a contracting geodesic}.
    
    Let two points $x,y \in G$ be as described in the lemma. It follows that the geodesic $[xo,yo]$ must intersect the $(500D+100{r})$ neighborhood of $z'$ and $z_g'$ by \cref{geodesic connects left and right}. Then $z$ is a $(500D+200{r},{g})$-barrier of $[xo,yo]$.
    Finally choose $N$ large enough to apply \cref{Ancona on geodesic in CAT(0)}.
\end{proof}



\subsection{Minimal actions on Martin boundary}\label{subsec Minimal actions on Martin boundary}

Assume a non-elementary group $G$ acts geometrically and essentially on an irreducible $\mathrm{CAT(0)}$ cube complex $X$ with a Morse hyperplane. We prove \cref{North south dynamics: intro} as  an application of Ancona inequality obtained in the preceding subsection.  

Let $g\in G$ be a contracting element and $\gamma$ be a $D$-contracting geodesic axis in $X$ for some $D>0$. By \cref{EFCG converge}, $\{g^n:n\ge 1\}$ (resp. $\{g^{-n}:n\ge 1\}$) tends to a minimal Martin boundary point $g^+$ (resp. $g^-$). We know $g^-\ne g^+$ by \cref{the map from horofunction to Martin}.

Assume in addition that $g$ skewers a Morse hyperplane. We shall follow the argument in \cite{yang2022conformal} to prove that the action of $g$ on the Martin boundary $\partial_{\mathcal M} G$ exhibits North-South dynamics as in \cref{North south dynamics}. In the process, we establish a few facts which are of independent interest.

First of all, the following verifies the Assumption A for $\gamma$ in \cite{yang2022conformal}.
\begin{lem}\label{assumption A}
Let $h_n\in G$ be a sequence of elements so that $\pi_\gamma(h_no)$ is unbounded on the positive half ray. Then $h_n$ converges to $g^+$. 
\end{lem}

\begin{proof}
    A proof is provided in \cref{converge along conical rays}.
\end{proof}

The following corollary is immediate.
\begin{cor}
Let $K$ be a compact subset in $\partial_{\mathcal M} G\setminus \{g^\pm\}$. Then for any sequence of $h_n\in G$ tending to some $\xi\in K$, $\{\pi_\gamma(h_no):n\ge 1\}$ is a bounded set.        
\end{cor}
\begin{lem}
There exists a number $L$ satisfying the following property.  Let $x_no\in X$ be a sequence of points tending to  some $\xi\in \partial_{\mathcal M} X\setminus \{g^\pm\}$. Then $\mathbf d_\gamma(x_no,x_mo)\le 3L$ for all but finitely many $n,m$.   
\end{lem}
\begin{proof}

Fix a large $l>0$ to be determined below. By \cref{skewer and intersect}, the axis $\gamma$ intersects a sequence of disjoint Morse hyperplanes $J_i$ for $i\in\mathbb Z$ so that  $J_{i+1}=g^lJ_{i}$ for each $i$. According to \cref{hyperplane projection small}, we can choose $p_i\in \gamma$ between $\pi_\gamma(J_i)$ and $\pi_\gamma(J_{i+1})$. Increasing   $l$ if necessary, let us assume $L=d(\pi_\gamma(J_i), \pi_\gamma(J_{i+1}))>10D$.

Assume to the contrary that $\mathbf d_\gamma(x_no,x_mo)> 3L$ holds for infinitely many $n,m$. Thus, for those $x_n,x_m$,  there exists some $i$ so that $\pi_\gamma(x_no)$ and $\pi_\gamma(x_mo)$ are separated by $J_i,J_{i+1}$ and $J_{i+2}$. Choose $a,b\in G$ so that $d(ao,p_i), d(bo,p_{i+1})\le \varepsilon$. Thus,   $(x_no,  bo)$ are antipodal in $J_i\cup J_{i+1}\cup\gamma$ along $ao$ and $(ao,  x_mo)$ are antipodal in $J_{i+1}\cup J_{i+2}\cup\gamma$ along $bo$, so \cref{Morse hyperplanes and contracting geodesic} shows that $$\mathcal G(b,x_n)\asymp\mathcal G(b,a) \mathcal G(a,x_n) \bigand \mathcal G(a,x_m)\asymp \mathcal G(a,b) \mathcal G(b,x_m).$$ If we divide by $\mathcal G(e,x_n)$ the two sides of the first formula,  letting $n\to\infty$ gives $K_\xi(b) \asymp K_\xi(a)\mathcal G(b,a)$. Similarly, we derive $K_\xi(a) \asymp K_\xi(b)\mathcal G(a,b)$ from the second formula. The Ancona inequality provides an implicit constant $C$ so that $\mathcal G(a,b)\mathcal G(b,a)>C$. Since $\lim_{g \to \infty} \mathcal{G}(e,g) = 0$ by \cref{TheGreenfunctiondecay}, we may choose $l$ large enough so that  $\mathcal G(a,b), \mathcal G(b,a)<\sqrt{C}$, so this  gives a contradiction.       
\end{proof}

We wish to extend the shortest projection map $\pi_\gamma: X\to \gamma$ to the boundary points $\xi\in \partial_{\mathcal R} X\setminus \{g^\pm\}$. We now define $\pi_\gamma(\xi)=\pi_\gamma(x_n)$ for any $n\gg 0$. By the above discussion, this is coarsely well-defined up to a bounded error $3L$. Moreover, we see that 
the extension of $\pi_\gamma$ coarsely commutes with the $\langle g\rangle $-action: $g\pi_\gamma(\xi)$ is coarsely equal to $\pi_\gamma(g\xi)$ up to $3L$-neighborhood.

\begin{lem}\label{North south dynamics}
The action of $g$ on the Martin boundary $\partial_{\mathcal M} G$ has North-South dynamics: for any open neighborhood $U,V$ of $g^-,g^+$, there exists $N_0$ so that $g^n(\partial_{\mathcal M} G\setminus U)\subset V$ for any $n>N_0$.      
\end{lem}
\begin{proof}
We fix the base point $o\in \gamma$. Let $K$ be the set of elements $h\in G$ so that $\pi_\gamma(ho)$ intersects $[o,go]_{\gamma}$. Hence $\cup_{n \in \mathbb{Z}}\; g^n K=G$. Moreover, if $\widetilde K$ denotes the accumulation points of $K$ in $\partial_{\mathcal M} G$, we see that $\widetilde K$ is disjoint with $g^\pm$ by \cref{assumption A}. 
Using  the fact that $\pi_\gamma$ coarsely commutes with the $\langle g\rangle $-action, the proper  action of $\langle g\rangle $ on $\gamma$ shows the action of $\langle g\rangle $ on $\partial_{\mathcal R} X\setminus \{g^\pm\}$ has  the desired North-South dynamics. We refer to \cite[Lemma 3.27]{yang2022conformal} for full details.
\end{proof}

It is a general fact that the North-South dynamics implies the existence of a unique minimal $G$-invariant closed subset. We follow the proof of  \cite[Lemma 3.30]{yang2022conformal}.
\begin{lem}
The topological closure denoted by $\Lambda(G)$ of $\{hg^-,hg^+: h\in G\}$ is the unique minimal $G$-invariant closed subset in $\partial_{\mathcal M} G$.     
\end{lem}
\begin{proof}
Recalling that $G$ is assumed to be non-elementary, the index of $E(g)$ in $G$ is infinite and the set $A:=\{hg^-,hg^+: h\in G\}$  is infinite by \cref{many contracting elements skewer}.

Let $\Lambda $ be any $G$-invariant closed subset in $\partial_{\mathcal M} G$. We are going to prove $\Lambda(G)\subset \Lambda$. First, $\Lambda$ contains at least 3 points by the South-North dynamics in \cref{North south dynamics}. Second, we choose $x\in \Lambda\setminus g^{\pm}$, so $g^n x\to g^+$ by \cref{North south dynamics}. Thus, $g^+\in \Lambda$ and then $A$ is a subset in $\Lambda$. The conclusion follows.
\end{proof}

\begin{lem}
The $[\cdot]$-closure of $\Lambda$ is the whole Martin boundary.    
\end{lem}
\begin{proof}
Let $\xi\in \partial_{\mathcal M} G$ and $h_n\in G\to \xi$. We consider the sequence of axis $h_n\gamma$ in $X$ where $\gamma$ is the axis of $g$.   If $\{h_n\gamma:n\ge 1\}$ is a finite set, then $h_no$ lies in a finite neighborhood of some $h_i\gamma$, so $h_n$ tends to $h_ig^+$ or $h_ig^-$. Thus, $\xi\in \Lambda(G)$. Let us assume that $h_n\gamma$ is  distinct up to taking a subsequence. In this case, we claim $h_ng^+$ converges to a point in $[\xi]$. To this end, it suffices to show that there exists a constant $C$ such that $K_{h_ng^{+}}(x) \asymp_C K_{h_n}(x)$ for any $x \in G$ and all sufficiently large $n$.

Fix an arbitrary element $x \in G$, and then fix $n$ large enough so that $d(o,h_n\gamma) > 3d(o,xo)$. By the contracting property we have $\mathbf{d}_{h_n\gamma} (o,xo) \leq D$. Without loss of generality, we assume $\|g\| \geq D$.
Since $\gamma$ is the axis of $g$, there exists an integer $k$ such that $\diam\{ h_ng^ko\cup \pi_{h_n\gamma}(o) \cup \pi_{h_n\gamma}(xo) \} \leq 2\|g\|$.
Let $C_1$ and $N$ be the constants from \cref{Ancona in CAT(0)} with $r = \varepsilon$.
Applying \cref{Ancona in CAT(0)} to the $g^N$-barrier $h_ng^{k+N}$ along $h_n\gamma$, we obtain for any $m > 3N$:
\begin{align*}
    &\mathcal{G}(x,h_ng^{k + m}) \asymp_{C_1} \mathcal{G}(x,h_ng^{k+N})\mathcal{G}(h_ng^{k+N},h_ng^{k + m}), \\
    &\mathcal{G}(e,h_ng^{k + m}) \asymp_{C_1} \mathcal{G}(e,h_ng^{k+N})\mathcal{G}(h_ng^{k+N},h_ng^{k + m}).
\end{align*}
Dividing the first inequality by the second, we obtain $K_{h_ng^{k + m}}(x) \asymp_{C_1} K_{h_ng^{k+N}}(x)$.
Hence there exists a constant $C_2$ such that $K_{h_n g^{k + m}}(x) \asymp_{C_2} K_{h_ng^{k}}(x)$ for all $m \geq 0$.
By symmetry, this also holds for $m \leq 0$. Therefore,
$$ K_{h_n}(x) \asymp_{C_2} K_{h_n g^{k}}(x) \asymp_{C_2} K_{h_ng^{k + m}}(x).$$
Taking $m \to +\infty$, we conclude that $K_{h_n}(x) \asymp_{C_2} K_{h_ng^{+}}(x)$, which completes the proof.
\end{proof}

\section{Application (II): partial boundary maps from Roller boundary} \label{sec The map (II)}
Resuming the setup as in \cref{sec Ancona in CAT(0)},  a non-elementary group $G$ acts geometrically and essentially on an irreducible $\mathrm{CAT(0)}$ cube complex $X$ with a Morse hyperplane. We equip (the vertex set of)  $X$ with the Roller boundary $\partial_\mathcal{R} X$
and $G$ with the Martin boundary $\partial_{\mathcal M} G$.  

Let $\Lambda$ be a subset in $\partial_\mathcal{R} X$. Let $\Psi: X\to G, go \mapsto g$ be the coarsely defined map.  We say that $\Psi$ \textit{extends continuously} to $\Lambda$ if for any $\xi\in \Lambda$ and for any sequences $g_no, h_no\to \xi$, we have $g_n,h_n$ tend to the same Martin boundary point. If we denote the limit point by $\partial\Psi(\xi)$, this defines a (partial) boundary map   denoted as $\partial\Psi: \Lambda\to \partial_{\mathcal M} G$.  

The central result of this section is the following theorem.

\begin{thm}\label{the map from horofunction to Martin in CAT}
Let a non-elementary group $G$ act geometrically and essentially on an irreducible $\mathrm{CAT(0)}$ cube complex $X$ containing a Morse hyperplane.
Then there exists a $\nu_{PS}$-full measure subset $\Lambda$ of the Roller boundary $\partial_\mathcal{R} X$ so that $\Psi: X\to G$ extends continuously to a boundary map $\partial \Psi: \Lambda\to \partial^{m}_\mathcal{M} G$. Moreover, $\partial \Psi$ is  a homeomorphism to the image. 
\end{thm}

For technical convenience, we consider the action $G \act (X^0,d_{\ell^1})$ on the vertex set of the Davis complex with $\ell^1$ metric and discrete geodesics on $X^0$.

\subsection{The partial boundary map}
In this subsection we first specify the construction of $(r,F)$-conical points for a particular choice of $F$ as follows, and then prove the main result \cref{the map from horofunction to Martin in CAT}.

Let us choose three pairwise independent contracting elements $F = \{f_1,f_2,f_3\}$, so that   
\begin{enumerate}
    \item Each of them skewers a Morse hyperplane.
    \item The fixed points of each $f_i$ in the Roller boundary are regular.
\end{enumerate} 
Existence of such $F$ is established by \cref{many contracting elements skewer}.

Let $\Lambda_r^F(Go)$ denote the set of all $(r,F)$-conical limit points. Assume that  $r=r(\f)$ is a large number so that \cref{ShadowLem} holds. By \cref{ConicalPointsLem},   for every point $\xi \in \Lambda_r^F(Go)$, there exists an $(\hat r, F)$-admissible ray $\gamma_{[\xi]}$ ending at $[\xi]$, and the constant $\hat{r}$ remains unchanged when replacing elements of $F$ with their sufficiently high powers. Let $D$ be the contracting constant of the axes of $f \in F$. 

\begin{prop}\label{when F has regular fixed points}
Assume that every element of $F$ has regular fixed points. Then
 for any $r > 0$, after replacing elements in $F$ with sufficiently high powers, 
every $(r,F)$-admissible ray tends to a regular point in the Roller boundary.
\end{prop}

\begin{proof}
Let $\gamma$ be an $(r,F)$-admissible ray containing infinitely many $(r,f)$-barriers for some $f \in F$. 
Consider a barrier $(to, tfo)$ with $t \in G$ and $f \in F$. Choose points $a,b \in \gamma$ such that $d(to,a), d(tfo,b) \leq r$. 
Then for sufficiently high powers $f^n$, the segment $[to, tf^no]$ intersects at least $2r+3$ pairwise strongly separated hyperplanes by \cref{regular fixed points}.
At most $2r$ hyperplanes intersecting $[to, tf^no]$ fail to intersect $[a,b]$.
Consequently, $[a,b]$ must intersect a triple of pairwise strongly separated hyperplanes.

Since $\gamma$ has infinitely many such barriers, it intersects infinitely many triples of pairwise strongly separated hyperplanes. 
Using the trick in \cref{regular fixed points} again, we choose their ``middle'' hyperplanes (with respect to the traversal order along $\gamma$). These chosen hyperplanes form an infinite family that are pairwise strongly separated, which demonstrates that $\gamma$ tends to a regular point.
\end{proof}

By \cref{when F has regular fixed points}, since the finite difference partition on the regular boundary is minimal,  we have  $[\xi] = \{\xi\}$ for   every conical point $\xi \in \Lambda_r^{F}(Go)$. 


\begin{lem}\label{converge along conical rays}
    Let $\gamma$ be an $(r,F)$-admissible ray. Then there exists a minimal Martin boundary point $\zeta$ so that for any sequence of elements $\{h_n\}$, if $\{h_n\}$ converges to a Martin boundary point $\zeta'$, then \[\lim_{n \to \infty} d(\pi_{\gamma}(h_no),o) = \infty  \Leftrightarrow \zeta' = \zeta.\]
    In particular, any distinct sequence of $h_n$ in a linear neighborhood of $\gamma$ converges to $\zeta$.
\end{lem}

\begin{proof}
    First assume $\lim_{n \to \infty} d(\pi_{\gamma}(h_no),o) = \infty$.
    The proof follows the same structure of \cref{EFCG converge}.
    Let $\{g_n\}$ be an unbounded sequence of $(r,f)$-barriers on $\gamma$ for some $f \in F$. Let us take a sub-sequence of $g_n$ so that  $g_n \rightarrow \zeta$. By \cref{Ancona in CAT(0)} the Ancona inequality holds for any triple $(g_i,g_n,g_j)$ with $i < n < j$, so  $\zeta$ is a minimal point by \cref{AnconaIneqImplyMinimal}. 

    Note that for any $x \in G$, there exist $n$ sufficiently large and $j$ sufficiently large (respective to $n$) such that $$\mathrm{min}\{d(\pi_{\gamma}(xo),g_no),d(g_no,\pi_{\gamma}(h_jo))\} \geq  \|f\| + 2r.$$
    By \cref{Ancona in CAT(0)}, there is a uniform constant $C$ such that 
    \[ \mathcal{G}(x,h_j) \asymp_C \mathcal{G}(x,g_n)\mathcal{G}(g_n,h_j). \]
    This verifies \cref{Criterion of Martin convergence to locus} for the sequence $h_j$. Thus, any accumulation point $\zeta'\in \partial_{\mathcal M}G$ of $h_j$ lies in $[\zeta]$. The minimalness of $\zeta$ implies  $\zeta' = \zeta$, concluding the proof.

    Now assume $\lim_{n \to \infty} d(\pi_{\gamma}(h_no),o) \neq \infty$. By passing to a subsequence, we may assume $\{d(\pi_{\gamma}(h_no),o)\}_{n \geq 0}$ is bounded. Then there exist $n$ sufficiently large and $j$ sufficiently large (depending on $n$) such that for any $i > 0$ and $n < j$, the projections $\pi_{\gamma}(h_io)$ and $\pi_{\gamma}(g_jo)$ lie on opposite sides of $g_no$ along $\gamma$.

By \cref{Ancona in CAT(0)}, there exists a uniform constant $C$ such that
$$
\mathcal{G}(g_j,h_i) \asymp_C \mathcal{G}(g_j,g_n) \mathcal{G}(g_n,h_i) \bigand \mathcal{G}(e,g_j) \asymp_C \mathcal{G}(e,g_n) \mathcal{G}(g_n,g_j).
$$
We divide the LHS and RHS of two formulae  by $\mathcal{G}(e,h_i) $ and $\mathcal{G}(e,g_j) $ respectively:
$$
K_{h_i}(g_j) \asymp_C  K_{h_i}(g_n) \cdot \mathcal{G}(g_j,g_n) \bigand K_{g_j}(g_n) \asymp_C   \mathcal{G}(e,g_n)^{-1}.
$$
Taking $i \to \infty$ followed by $j \to \infty$ in the first formula yields 
$$
\lim_{j \to \infty} K_{\zeta'}(g_j) = K_{\zeta'}(g_n)\lim_{j \to \infty}  \mathcal{G}(g_j,g_n) = 0.
$$
On the other hand $K_{\zeta}(g_n) =  \lim_{j \to \infty} K_{g_j}(g_n) \asymp_C  \mathcal{G}(e,g_n)^{-1}$, so $\lim_{n \to \infty} K_{\zeta}(g_n) = \lim_{n \to \infty}  \mathcal{G}(e,g_n)^{-1} = +\infty$. This shows $\zeta \neq \zeta'$.
\end{proof}

\begin{defn}[{The boundary map $\partial\Psi: \Lambda_r^{F}(Go)  \to \partial_{\mathcal M} G$ }]\label{def from horofunction to Martin for CAT}

For any conical point $\xi \in \Lambda_r^{F}(Go)$,
by \cref{ConicalPointsLem} and \cref{when F has regular fixed points}, there exists an $(\hat r, F)$-admissible ray $\gamma_{\xi}$ ending at $\xi$. 
Define $\partial\Psi(\xi) \in \partial_{\mathcal{M}} G$ to be the Martin boundary point which is converged by $\gamma_{\xi}$ as established in \cref{converge along conical rays}.
\end{defn}

We now prove $\Psi$ possesses the properties claimed in \cref{the map from horofunction to Martin in CAT}.

\begin{lem}\label{well defined in CAT(0)}
The map $\partial\Psi$ is well defined and injective.
\end{lem}
\begin{proof}
Let $\gamma_1$ and $\gamma_2$ be two $(r,F)$-admissible rays converging to $\xi_1,\xi_2 \in \Lambda_r^{F}(Go)$ and $\zeta_1,\zeta_2 \in \partial_{\mathcal{M}}G$, respectively.
By \cref{frequently contracting to horofunction boundary} and \cref{converge along conical rays}, we have 
\[ \xi_1 = \xi_2 \Leftrightarrow \diam\{\pi_{\gamma_1}(\gamma_2)\} = \infty \Leftrightarrow \zeta_1 = \zeta_2, \]
which completes the proof.
\end{proof}

\begin{lem}\label{the conical case is continuous}
    The map $\Psi$ extends continuously to $\Lambda$.
\end{lem}

\begin{proof}
    Let $\{g_n\}$ be a sequence of elements in $G$ whose orbit points $\{g_n o\}$ converge to $\xi \in \Lambda_r^{F}(Go)$ that is ended by an admissible ray $\gamma$. By \cref{frequently contracting to horofunction boundary}, the projection $\pi_{\gamma}(g_no)$ tends to infinity, thus we obtain $g_n \rightarrow \partial\Psi(\xi)$ by \cref{converge along conical rays}, which proves the proposition.
\end{proof}

\begin{lem}\label{inverse continuous}
    The inverse map $(\partial\Psi)^{-1}$ is a continuous extension of the orbit map $\Phi:g \mapsto go$.
\end{lem}

\begin{proof}
Let $\{g_n\}$ be a sequence of elements in $G$ converging to $\zeta \in \partial_{\mathcal M}G$ that is approached by an admissible ray $\gamma$.
By \cref{converge along conical rays}, we have $\lim_{n \to \infty} d(\pi_{\gamma}(g_n o), o) = \infty$, and thus $g_no \to [\xi]$ by \cref{frequently contracting to horofunction boundary}.
\end{proof}

We now conclude the proof of \cref{the map from horofunction to Martin in CAT}.
The map $\Psi$ is defined on conical points as in \cref{def from horofunction to Martin for CAT}. Then the full measurability of $\Lambda$ follows from \cref{HSTLem}.
The map $\partial\Psi$ is a homeomorphism by \cref{the conical case is continuous} and \cref{inverse continuous}.

\subsection{Application to right-angled Coxeter groups}\label{Sec Application to right-angled Coxeter groups}
We apply the previous theory to a special  class of right-angled Coxeter groups with Morse hyperplanes in its Davis complex.  We start with a brief introduction to right-angled Coxeter groups and refer to \cite{davis_geometry_2008} for a comprehensive treatment.
\begin{defn}
Let $\Gamma$ be a finite simplicial graph with vertex set $\{V(\Gamma)\}$ and edge set $\{ E(\Gamma) \subset V(\Gamma) \times V(\Gamma) \}$.  The \textit{right-angled Coxeter group} denoted by $G_{\Gamma}$ (RACG for shorthand) is given by the group representation:
\[ G_{\Gamma}:= \langle v \in V(\Gamma) | v^2 = 1 \text{ for all } v \in V(\Gamma), [v,w] = 1 \text{ if and only if } (v,w) \in E(\Gamma) \rangle. \]    
The graph $\Gamma$ will be referred to as \textit{defining graph} of $G_{\Gamma}$. 
\end{defn}

A subgraph $\Lambda \subset \Gamma$ is an \textit{induced subgraph} if whenever two vertices of $\Lambda$ are connected in $\Gamma$, they are connected in $\Lambda$.
We say that $G_{\Lambda}$ is a \textit{special subgroup} of $G_{\Gamma}$ if $\Lambda$ is an induced graph of $\Gamma$.

Next we recall the construction of Davis complex $\mathcal{D}_{\Gamma}$ of right-angled Coxeter group $G_{\Gamma}$ as given in  \cite[Ch. 7]{davis_geometry_2008}. See also \cite[Section 3]{dani_divergence_2015} and \cite[Section 7]{tran_strongly_2019} for more details.

A $\textit{k-clique}$ means a complete graph with $k$ vertices.
The Davis complex has $1$-skeleton as the Cayley graph of $G_{\Gamma}$ with the standard generating set. For every $k$-clique $T \subset \Gamma$, since the special subgroup $G_{T}$ is isomorphism to $\mathbb{Z}_2^k$, we glue a unit $k$-cube to every coset $gG_T$. The Davis complex is a $\mathrm{CAT(0)}$ cube complex and the group $G_{\Gamma}$ acts properly and cocompactly on it. 
Moreover, for every induce subgraph $\Lambda \subset \Gamma$, the Davis complex $\mathcal{D}_{\Lambda}$ embeds isometrically in $\mathcal{D}_{\Gamma}$. 
For example, if $\{s,t\}$ are two disconnected vertices of $\Gamma$, the group $G_{\{s,t\}}$ is isomorphic to the infinite dihedral group $D_{\infty}$ and the Davis complex $\mathcal{D}_{\{s,t\}}$ is a geodesic in $\mathcal{D}_{\Gamma}$.

For every vertex $v \in V(\Gamma)$, let $J_v$ be the hyperplane in $\mathcal{D}_{\Gamma}$ dual to the edge $(e, v)$. Since $G$ acts transitively on the vertices of $\mathcal{D}_{\Gamma}$, every hyperplane is a translate of some $J_v$.
Let $\mathrm{star}(v)$ denote the induced subgraph of $\Gamma$ consisting of $v$ and all vertices adjacent to $v$. Then, as also noted in \cite[Remark 7.5]{tran_strongly_2019}, the following properties hold:
\begin{enumerate}
    \item The endpoints of all edges dual to $J_v$ (viewed as vertices in the Cayley graph) are exactly the subgroup $G_{\mathrm{star}(v)}$.
    \item $G_{\mathrm{star}(v)}$ is the stabilizer of $J_v$ under the action of $G$.
    \item The action of $v$ flips $J_v$, exchanging the two half-spaces bounded by $J_v$.
\end{enumerate}
\

In particular, the Hausdorff distance of $J_v$ and $G_{\mathrm{star}(v)}$ in $\mathcal{D}_{\Gamma}$ is less than $1$.

We now characterize when the action of the right-angled Coxeter group $G = G_\Gamma$ on its Davis complex $X = \mathcal{D}_\Gamma$ is essential and when there is a Morse hyperplane in $X$.

First, we examine the essentiality of the action. Suppose there exists a vertex $v \in \Gamma$ adjacent to every other vertex. Then, $G$ decomposes as a direct product of two subgroups, one of which is $\langle v \mid v^2 = e \rangle \cong \mathbb{Z}/2\mathbb{Z}$. Correspondingly, the Davis complex splits as a product $I \times D'$, where $I$ is the edge $(e, v)$ in $\mathcal{D}_\Gamma$. In this case, the hyperplane dual to $I$ is non-essential because $X = \mathcal{D}_\Gamma$ lies within a neighborhood of this hyperplane. Thus, the action fails to be essential.

We now prove that the action is essential if no such vertex exists in $\Gamma$.

\begin{lem}\label{infinite index}
    If $\Gamma$ contains no vertex adjacent to all other vertices, then every proper special subgroup of $G$ has infinite index.
\end{lem}

\begin{proof}
    Let $\Lambda$ be an induced subgraph of $\Gamma$. 

    \textbf{Case 1:} Suppose $\Gamma$ decomposes as a join $\Lambda * \Lambda_1$. Since no vertex in $\Gamma$ is adjacent to all others, $\Lambda_1$ cannot be a complete graph. Consequently, the subgroup $G_{\Lambda_1}$ is infinite, and thus $G_{\Lambda}$ has infinite index in $G_{\Gamma} = G_{\Lambda} \times G_{\Lambda_1}$.

    \textbf{Case 2:} If $\Gamma$ does not decompose as such a join, there exist vertices $v \in \Lambda$ and $w \notin \Lambda$ that are non-adjacent in $\Gamma$. 
    For contradiction, assume $G_{\Lambda}$ has finite index in $G_{\Gamma}$. By the pigeonhole principle, there exists $t \in \mathbb{Z}^+$ such that $(wv)^t \in G_{\Lambda}$. 

    Observe that the inclusion $\mathcal{D}_{\{w,v\}} \hookrightarrow \mathcal{D}_{\Gamma}$ is an isometric embedding, whose image is a geodesic in $\mathcal{D}_{\Gamma}$. The points $e$, $w$, and $(wv)^t$ lie sequentially on this geodesic. 
    Let $J_w$ be the hyperplane in $\mathcal{D}_{\Gamma}$ dual to the edge $(e,w)$. Since $w \notin \Lambda$, $J_w$ separates $w$ from $\mathcal{D}_{\Lambda}$. In particular, $J_w$ must separate $w$ from both $e$ and $(wv)^t$, implying that the geodesic $\mathcal{D}_{\{w,v\}}$ intersects $J_w$ twice. 
    However, this contradicts the convexity of the half-space bounded by $J_w$, which completes the proof.
\end{proof}

\begin{lem}\label{infinite index to essential}
    If $G_{\mathrm{star}(v)}$ is an infinite index subgroup, then the hyperplane $J_v$ is essential in $\mathcal{D}_{\Gamma}$.
\end{lem}

\begin{proof}
    Suppose for contradiction that there exists $R > 0$ such that for every vertex $v \in \mathcal{D}_\Gamma$, the distance $d(v, J)$ is bounded by $R$. Then for any $v \in G$, there exists $h \in G_{\mathrm{star}(v)}$ such that $d(v,h) = d(h^{-1}v, e) \leq R + 1$. This implies that $h^{-1}v$ belongs to a finite set of group elements, contradicting the fact that $G_{\mathrm{star}(v)}$ has infinite index in $G$.

    Therefore, for any $R > 0$, there exists a vertex $v \in \mathcal{D}_\Gamma$ with $d(v, J) \geq R$. Since the left multiplication by $v$ exchanges the half-spaces $J^+$ and $J^-$, it follows from the definition that $J$ is an essential hyperplane.
\end{proof} 

From \cref{infinite index} and \cref{infinite index to essential}, we have the following.

\begin{prop}\label{RACG essential}
    $\Gamma$ contains no vertex adjacent to all other vertices if and only if $G_{\Gamma}$ acts essentially on $\mathcal{D}_{\Gamma}$.
\end{prop}

\begin{proof}
    The ``if'' part is proved in the paragraph before \cref{infinite index}, hence we only need to prove the ``only if'' part.
    Since $\Gamma$ contains no vertex adjacent to all other vertices, the graph $\Gamma$ is not a star. 
    This implies every hyperplane in $\mathcal{D}_{\Gamma}$ is essential by \cref{infinite index} and \cref{infinite index to essential}. Moreover, this is equivalent to the essentiality of the action.
\end{proof}

Next we characterize whether a RACG has a Morse hyperplane in its Davis complex from its defining graph.
Genevois gives a characterization for a special subgroup to be Morse.

\begin{prop}\cite[Proposition 4.9.]{genevois2020hyperbolicities}\label{Morse subgroup}
Let $\Lambda \subset \Gamma$ be an induced subgraph. The subgroup $G_{\Lambda} \subset G_{\Gamma}$ is Morse if and only if every induced square ($4$-cycle) of $\Gamma$ containing two diametrically opposite vertices in $\Lambda$ must be included into $\Lambda$.
\end{prop}
Recall that for every hyperplane $J_v$ in $\mathcal{D}_{\Gamma}$, the Hausdorff distance of $J_v$ and $G_{\mathrm{star}(v)}$ in $\mathcal{D}_{\Gamma}$ is less than $1$. Hence the cube complex has a Morse hyperplane $J_v$ is equivalent to that the Coxeter group has a Morse star subgroup $G_{\mathrm{star}(v)}$. 

\begin{prop}\label{has Morse hyperplane}
    The complex $\mathcal{D}_{\Gamma}$ has no Morse hyperplanes if and only if every vertex of $\Gamma$ is contained in an induced square.
\end{prop}

\begin{proof}
    If every vertex $v \in V(\Gamma)$ is contained in an induced square $vv_1v_2v_3$, then the opposite vertices points $\{v_1,v_3\}$ are contained in $\mathrm{star}(v)$ and $v_2 \notin \mathrm{star}(v)$. 
    By \cref{Morse subgroup}, the star subgroup $G_{\mathrm{star}(v)}$ is not Morse. 

    Conversely, if for every $v \in V(\Gamma)$, the star subgroup $G_{\mathrm{star}(v)}$ is not Morse, then there exists an induced square $v_0v_1v_2v_3$ such that $\{v_1,v_3\} \subset \mathrm{star}(v)$ and $\{v_0,v_2\} \not\subset \mathrm{star}(v)$ by \cref{Morse subgroup}. Assume $v_2 \notin  \mathrm{star}(v)$, which implies $v\notin \{v_1,v_3\}$ since $v_2$ is adjacent to both $v_1$ and $v_3$. Hence $v v_1 v_2 v_3$ is an induced square containing $v$.
\end{proof}

\begin{cor}\label{RACG with Morse hyperplane}
    Let $\Gamma$ be an irreducible graph (i.e., not a graph join) that contains at least one vertex not contained in any induced 4-cycle. Then the right-angled Coxeter group $G_{\Gamma}$ admits a geometric and essential action on the Davis complex $\mathcal{D}_{\Gamma}$, which is an irreducible $\CAT$ cube complex containing at least one Morse hyperplane.
\end{cor}

\begin{proof}
    Since $\Gamma$ is not a join, there is no vertex adjacent to all other vertices, which means essential action by \cref{RACG essential}. Furthermore, \cref{has Morse hyperplane} guarantees the existence of at least one Morse hyperplane in $\mathcal{D}_\Gamma$.
    The irreducibility of $\mathcal{D}_\Gamma$ follows from \cite[Proposition 2.11]{behrstock_thickness_2017}. For completeness, we include the proof: suppose $\mathcal{D}_\Gamma$ decomposes as a product of sub-complexes. Then the link of the identity vertex $e$ in $\mathcal{D}_\Gamma$ would have a 1-skeleton that is a graph join. However, this link is exactly the graph $\Gamma$ itself, contradicting our assumption that $\Gamma$ is not a join.
\end{proof}

Now we consider a non-elementary RACG $G_{\Gamma}$ acting on its Davis complex $X=\mathcal{D}_{\Gamma}$. We check for which graph $\Gamma$, the corresponding action satisfying the hypothesis of \cref{the map from horofunction to Martin in CAT}.

First, the group $G_{\Gamma}$ is elementary if and only if $\Gamma$ is a pair of disjoint points or a suspension of a clique \cite[Theorem 8.7.3]{davis_geometry_2008}.
Hence for irreducible graph $\Gamma$ with at least three vertices, the group $G_{\Gamma}$ is non-elementary.
By \cref{RACG with Morse hyperplane}, when $\Gamma$ is an irreducible graph that contains at least one vertex not contained in any induced 4-cycle, $G_{\Gamma}$ acts geometrically and essentially on its Davis complex $X=\mathcal{D}_{\Gamma}$, which is irreducible and containing one Morse hyperplane.

Identifying the $1$-skeleton of $\mathcal{D}_{\Gamma}$ with the Cayley graph of $G_{\Gamma}$, the map $\Psi$ in \cref{the conical case is continuous} coincides the identity map of the group. An application of \cref{the map from horofunction to Martin in CAT} yields the following theorem.

\begin{thm}\label{the map for RACG}
Let $\Gamma$ be an irreducible graph with at least three vertices and $G$ be the associated right-angled Coxeter group.
Assume that $\Gamma$ contains at least one vertex not contained in any induced 4-cycle. 
Then there exists a $\nu_{PS}$-full measure subset $\Lambda$ of the Roller boundary $\partial_\mathcal{R} \mathcal{D}_{\Gamma}$ of the Davis complex, and the identity $\Psi$ extends continuously to a boundary map from  $\Lambda$ to the minimal Martin boundary of $G$ $$\partial \Psi:\Lambda \to \partial^{m}_{\mathcal M} G$$  such that $\partial \Psi$ is a homeomorphism  to the image $\partial \Psi(\Lambda)$.
\end{thm}


\bibliographystyle{alpha}
\bibliography{ref}

\end{document}